%% file: main.tex
\documentclass[11pt,a4paper]{article}
\usepackage{graphicx} 
\usepackage{amsthm,amsmath}
\usepackage{amssymb}
\usepackage{multirow}
\usepackage[table]{xcolor}
\usepackage{amscd}
\usepackage{tikz}
\usepackage{tikz-cd}
\usepackage{placeins}

\usepackage[utf8]{inputenc}
\usepackage[T1]{fontenc}
\usepackage{lmodern}
\usepackage{bm}
\usepackage{array}
\usepackage{booktabs}
\usepackage{makecell}
\usepackage{rotating}
\usepackage{float}
\usepackage[hidelinks]{hyperref}

\renewcommand\arraystretch{1.15}
\newcolumntype{C}[1]{>{\centering\arraybackslash}p{#1}}

\DeclareRobustCommand{\SkipTocEntry}[5]{}

\title{Geometry and topology of the tempered Iwahori-spherical representations of a split semisimple $p$-adic group}
\author{Dominic Majda, Graham A. Niblo, Roger Plymen and Nick Wright}
\date{June 2026}

\newtheorem{theorem}{Theorem}

\newtheorem{lemma}[theorem]{Lemma}
\newtheorem{corollary}[theorem]{Corollary}

\newbox{\omitbox}

\def\SSpin{\mathrm{SSpin}}
\def\HSpin{\mathrm{HSpin}}
\def\Spin{\mathrm{Spin}}
\def\SO{\mathrm{SO}}
\def\PSO{\mathrm{PSO}}
\def\Sp{\mathrm{Sp}}
\def\PSp{\mathrm{PSp}}
\def\D{\mathcal{D}}
\def\SL{\mathrm{SL}}

\def\suspcone{Y_{t\geq0}}
\def\conesusp{Y_{s\geq0}}
\def\conecone{Y_{s,t\geq0}}
\def\suspsusp{Y}

\def\SSpin{\mathrm{SSpin}}

\def\t{\mathfrak t}

\def\fA{\mathfrak{A}}

\def\R{\mathbb{R}}
\def\Z{\mathbb{Z}}
\def\Q{\mathbb{Q}}
\def\Atilde{(\widetilde{A})}

\def\TT{\mathbb{T}}

\def\mbar{\overline{m}}

\def\cZ{\mathcal{Z}}
\def\fA{\mathfrak{A}}

\def\RP{\R\mathrm{P}}

\let\mid=:

\newcommand{\C}{\mathbb{C}}
\def\cZ{{\mathcal Z}}
\def\q {{/\!/}}

\def\mbar{\overline{m}}
\def\coveringtorus{\tilde{T}}
\def\ev{\mathrm{ev}}
\def\odd{\mathrm{odd}}

\def\ZDSpin{\cZ_{\Spin(2n)}}
\def\ZDSO{\cZ_{\SO(2n)}}

\def\ZCSp{\cZ_{\Sp(n)}}

\def\TSpinn{\coveringtorus_{n/2}}

\def\wplus{w^+}
\def\wminus{w^-}
\def\wplusminus{w^\pm}
\def\wminusplus{w^\mp}
\def\neg{\mathrm{neg}}

\def\bT{\mathrm{T}}

\def\G{\mathcal{G}}
\def\bG{\mathbf{G}}

\def\bundle#1#2#3{{\gamma\!\strut_{#1}(#2,#3)}}
\def\unitbundle#1#2#3{{\gamma^1\!\!\!\strut_{#1}(#2,#3)}}

\def\rank{\mathrm{rank}}

\begin{document}

\maketitle

\begin{abstract}
Let $\G$ be a connected split $p$-adic group of type $B_n$, $C_n$ or $D_n$. Amongst the tempered representations of $\G$ a key role is played by the Iwahori-spherical block. We provide a compact Hausdorff model for this space which allows us to compute the $K$-theory ranks for the corresponding $C^*$-algebra. 
The model, an extended quotient, is stratified by sectors.
Underlying the approach of this paper is the interplay between the geometric extended quotient and the spectral extended quotient in the context of the ABPS conjecture.

 We give geometric descriptions of every sector thus equipping the spectrum with a cellular structure. We classify the geometric structures arising in our model, and specifically discover real projective spaces along with cones and suspensions of these. These examples require a minor modification to our homotopy sector conjecture: we show that Langlands dual sectors are rationally (indeed dyadically) homotopy equivalent in all cases ($A_n, B_n, C_n, D_n, E_6, E_7$ and $E_8$).

\end{abstract}

\tableofcontents

\section{Introduction}

The ABPS conjecture \cite{ABPS-smooth-dual} proposes a correspondence between the representation theory of reductive $p$-adic groups and the extended quotients of certain associated tori.  Extending the work of Solleveld \cite{Sol}, we construct and study these extended quotients for split semisimple $p$-adic groups of types $B_n$, $C_n$ and $D_n$, thus providing, in the spirit of noncommutative geometry, a compact Hausdorff space which models the spectrum of an associated noncommutative $C^*$-algebra: the reduced Iwahori-spherical $C^*$-algebra of the group.

This $C^*$-algebra arises as follows:  let $\G$ be a connected split semisimple $p$-adic group.
The reduced Iwahori-spherical $C^*$-algebra $\fA(\G)$ is the unital sub-$C^*$-algebra of $C^*_r(\G)$ whose spectrum is the space of tempered Iwahori-spherical representations. The spectrum is equipped with a natural topology which makes this into a compact, typically non-Hausdorff, space.

Attached to $\G$ we have, via Langlands duality, a compact torus acted upon by the Weyl group $W$ of  $\G$.   According to the ABPS conjecture, the resulting extended quotient parametrises the spectrum of the reduced Iwahori-spherical $C^*$-algebra $\fA(\G)$. 

On the other hand the extended quotient breaks up into sectors indexed by the conjugacy classes of $W$.  The sectors equip the spectrum of $\fA(\G)$ with a cellular structure. We give detailed geometric descriptions of  each of these sectors whenever the $p$-adic group $\G$ is of type $B_n, C_n, D_n$.     Of special interest are the groups $\PSO(18, \Q_p)$ and $\Spin(18,\Q_p)$ and indeed all the semi-spin groups $\HSpin^\pm(4k,\Q_p)$ (sometimes referred to as half-spin groups). By convention we will use the notation $\HSpin$ in the $p$-adic case, and the notation $\SSpin$ for the real and complex semi-spin groups.

\bigskip

Specifically, we let $\bG = \G^\vee$ denote the Langlands dual of $\G$. This is a complex reductive group, and we let $G$ denote a maximal compact subgroup of $\bG$, with maximal torus $\bT_G$.
We consider the geometric extended quotient $\bT_G\q W$ and the spectral extended quotient $(\bT_G\q W)_2$. By \cite[Theorem 15.1]{ABPS3}, the points of $(\bT_G\q W)_2$ parametrise the spectrum of the tempered Iwahori-spherical algebra in a natural way.

Choosing bijections between conjugacy classes and irreducible representations for the isotropy groups of the action yields an identification of the two extended quotients $\bT_G\q W$ and $(\bT_G\q W)_2$. By \cite[Theorem 5.2]{ABPS} in the case that $\G$ has connected centre and the residual characteristic satisfies the Roche conditions \cite{Roche} such a choice can be made to yield a \emph{continuous} bijection:
\[
\mu:\bT_G\q W
\to \widehat{\fA(\G)}.
\]
This allows us to view the space of representations as equipped with a cellular structure in this case. The ABPS conjecture, \cite[Section 15]{ABPS-smooth-dual}, claims that such a continuous parametrisation can always be constructed.

In this paper we compute the extended quotients $\bT_G\q W$ which give parametrisations of the tempered part of the Iwahori-spherical block for the $p$-adic groups
\begin{itemize}
    \item $\Spin(2n+1,F)$, $\SO(2n+1,F)$, (type $B_n$)
    \item $\Sp(n,F)$, $\PSp(n,F)$, (type $C_n$)
    \item $\Spin(2n,F)$, $\SO(2n,F)$, $\PSO(2n,F)$ and $\HSpin^\pm(2n,F)$ (type $D_n$).
\end{itemize}
In the cases of $\SO(2n+1,F)$, $\Sp(n,F)$ and $\SO(2n,F)$ this calculation was originally due to Solleveld \cite{Sol}.

Note that the groups $\SO(2n+1,F), \PSp(n,F)$ and $\PSO(2n,F)$ have trivial (and hence connected) centre and therefore the parametrisation is proven to be continuous in these cases.

The $A_n$ and $E_6$ cases were previously considered in \cite{NPWAn} and \cite{NPWE6} respectively while the remaining cases ($F_4, G_2, E_7, E_8$) were computed in the first author's PhD thesis \cite{Majda}.

In the course of this work we exhibit the existence of real projective spaces in these models of the spectrum. 
Specifically, for $\G=\Spin(18,F)$, where $F$ is a finite extension of $\Q_p$ with $p>2$, the sector of signed\footnote{See Section \ref{signed permutations} for a  discussion of signed permutations.} cycle type $(5,3,1)$ exhibits a subset of the tempered Iwahori spectrum which maps continuously and bijectively to $\R \mathrm{P}^3$. We compare this with the more recent work of Aubert and Plymen \cite{AP} where they exhibit the existence of a Klein bottle in the tempered dual of $\SL(8,F)$. We note that in their example the Klein bottle is not in the Iwahori-spherical block; indeed (see \cite{NPWAn}) the Iwahori sectors in that group are unions of tori up to homotopy.

In exploring the geometry of the Iwahori sectors we find various other interesting phenomena.  For example in the cases of $\Spin(12,F)$ or $\Spin(13,F)$ with duals $\PSO(12,\C)$ and $\PSp(6,\C)$ respectively, the sector indexed by the conjugacy class of signed cycle type $(3,2,1)$ is in both cases the cone on $\RP^2$, and hence has a distinguished cone point with non-manifold geometry in this model of the Iwahori-spherical block. In the cases of $\Spin(32,F)$ and $\HSpin^\pm(32,F)$, the sector of signed cycle type $(7,5,3,1)$ has the geometry of $\Sigma\RP^3$ and thus has two distinguished cone points. It may be of interest to consider how these distinguished points map into the space of representations, and what this implies in the context of the ABPS conjecture.
\bigskip

Turning now to $K$-theory, applying Solleveld \cite[Eq. 4.9]{ABPS2}, the $K$-theory of the reduced Iwahori-spherical $C^*$-algebra $\fA(\G)$ can always be computed by
\[
K_j(\fA(\G)) \otimes_\Z \Q \cong K^j_W(\bT_G) \otimes_\Z \Q
\]
which after tensoring with $\C$ is computed (using the equivariant Chern character \cite{BaumConneschern}) by the cohomology $\bigoplus_k H^{j+2k}(\bT_G\q W,\C)$. This is a $K$-theoretic analogue of the parametrisation of $\widehat{\fA(G)}$ above.

In \cite{NPWKdual} we established a $W$-equivariant $KK$-duality between $C(\bT_G)$ and $C(\bT_{G^\vee})$ where $G^\vee$ is the real Langlands dual of $G$. In \cite{NPWAn} we showed that this duality holds at the level of sectors, again after tensoring with $\C$. We observed in the examples of \cite{NPWAn} and \cite{NPWE6} that the sector-wise duality is mediated by homotopy equivalences at the level of sectors, and in \cite{NPWE6} we conjectured that this holds in general. Further evidence for the conjecture was furnished in the thesis of the first author \cite{Majda} who verified this in the case of $E_7$. Since $E_8,F_4,G_2$ are self-dual this leaves the cases of $B_n,C_n$ and $D_n$ which are considered here.

We show that in almost every case the homotopy sector conjecture holds, however there are exceptions enumerated below, arising from the existence of projective spaces in the $\PSO(2n,\R)$ sectors as mentioned above, and spheres in the sectors of the dual $\Spin(2n,\R)$. However even in these cases there is a dyadic rational homotopy equivalence. Indeed we obtain:
\begin{theorem}\label{theorem}
    Let $\bT_G$ and $\bT_{G^\vee}$ be maximal tori for Langlands dual compact connected semisimple Lie groups ${G}, {G}^\vee$ and $w\in W({G})$. Then the sectors $\bT_G^w/Z(w)$ and $(\bT_{G^\vee})^w/Z(w)$ are homotopy equivalent except in the following case:
    \begin{enumerate}
        \item The Lie groups are $\Spin(2n)$ and $\PSO(2n)$.
        \item The element $w$ has signed cycle type with no negative cycles, with at least $3$ odd positive cycle lengths, and with all odd cycle lengths having multiplicity $1$.
    \end{enumerate}

    The corresponding sectors are then homotopy equivalent to $S^{\delta_\odd}$  for the $\Spin(2n)$ group, where ${\delta_\odd}$ is the number of odd positive cycles,  and to $\RP^{\delta_\odd}$ for $\PSO(4k+2)$ and $\Sigma \RP^{\delta_\odd-1}$ for  $\PSO(4k)$.
\end{theorem}

\begin{corollary}\label{corollary} Under the hypotheses of Theorem \ref{theorem}, the sectors $\bT_G^w/Z(w)$ and $(\bT_{G^\vee})^w/Z(w)$ are dyadically homotopy equivalent, hence rationally homotopy equivalent, in the following sense. There is a map between the sectors inducing isomorphisms:
\begin{align*}
\pi_0(\bT_G^w/Z(w))&\cong \pi_0((\bT_{G^\vee})^w/Z(w))\\
\pi_i(\bT_G^w/Z(w)) \otimes_\Z \Z[1/2] &\cong \pi_i((\bT_{G^\vee})^w/Z(w))\otimes_\Z \Z[1/2]&\text{for all }i\geq 1
\end{align*}
and the fundamental groups are abelian.

Moreover the sectors are either homotopy equivalent or they are connected spaces whose homotopy groups differ only by $2$-torsion.  In the case that $n$ is odd we can further strengthen this to say that the map on higher homotopy groups is integrally an isomorphism. 
\end{corollary}

\begin{corollary}
    Let $\G_1,\G_2$ be  split semisimple $p$-adic groups whose root data are dual to one another. If the ABPS conjecture \cite[\S15]{ABPS-smooth-dual} holds for $\G_1, \G_2$ then  the sectors parametrising their tempered Iwahori-spherical blocks are rationally homotopy equivalent in the above sense. 
\end{corollary}

\bigskip

The counterexamples identified above lead us to the following observation concerning torsion in $K$-theory. For the $\PSO(2n)$ group, the $K$-theory and cohomology of sectors of the form $\RP^\delta$ or $\Sigma \RP^{\delta-1}$ (for $\delta \geq 3$) contain torsion.  The corresponding sectors for $\Spin(2n)$ have no torsion in their (K-)homology, however, by the results of \cite{NPWKdual} there is an isomorphism
\[
  K_*^W(\bT_{\Spin(2n)}) \cong K^*_W(\bT_{\PSO(2n)})
\]
  and dually
  \[
  K_*^W(\bT_{\PSO(2n)})\cong K^*_W(\bT_{\Spin(2n)}).
  \]
We are left with a dichotomy: either the torsion in the cohomology of sectors is invisible at the level of equivariant $K$-theory of the $\PSO(2n)$ torus, or the torsion does appear here and therefore it must also appear in the equivariant $K$-homology group of the Spin torus but is invisible in the homology of the corresponding sectors. There is a similar dichotomy for the $K$-homology in the $\PSO(2n)$ case and the $K$-theory of the $\Spin(2n)$ case.  This leaves open the question of how these dichotomies are resolved.

\bigskip

Our computations allow us, in the context of groups of type $B_n,C_n$ and $D_n$, to compute the $K$-theory of $\fA(\G)$.
Focusing on groups of type $D_n$, the simply connected case is $\Spin(2n)$ whose centre has order $4$.  This is alternately a cyclic group, when $n$ is odd, and a Klein-$4$ group when $n$ is even.  Hence in the former case there are exactly $3$ groups of type $D_n$, viz. $\Spin(2n)$, $\SO(2n)$, $\PSO(2n)$, while in the latter case there is an additional pair of groups with centre of order $2$, denoted depending on context by $\SSpin^{\pm}(2n)$ or $\HSpin^{\pm}(2n)$. By convention we will use $\HSpin$ in the $p$-adic case and $\SSpin$ for the real and complex cases: the existence of this split $p$-adic group is far from obvious, however this is established in the book by Milne in the section on semisimple algebraic groups over arbitrary fields \cite[\S23]{Mil}.

The Langlands dual of $\Spin(2n,F)$ is $\PSO(2n,\C)$ with maximal compact subgroup $\PSO(2n,\R)$ and conversely the Langlands dual of $\PSO(2n,F)$ is $\Spin(2n,\C)$ with maximal compact subgroup $\Spin(2n,\R)$. The dual of $\SO(2n,F)$ gives $\SO(2n,\C)$ while the dual of $\HSpin^{+}(2n,F)$ will give $\SSpin^{+}(2n,\C)$ when $n\equiv 0$ mod $4$ and $\SSpin^{-}(2n,\C)$ when $n\equiv 2$ mod $4$. The duality for $\HSpin^{-}(2n,F)$ is analogous.

As an example we consider the case of $D_4$. 
Our computations give the following $K$-theory groups. 
\bigskip

\begin{table}[h]
\[
\begin{array}{c|c|c|c}
  \G&{G}&K_0(\fA(\G))\otimes \C &K_1(\fA(\G))\otimes \C\\[.5ex]
\hline
\Spin(8,F)   &\PSO(8,\R)   & \C^{33\strut } &0\\[.5ex]
 \SO(8,F)   & \SO(8,\R)   & \C^{30}&0\\[.5ex]
\HSpin^\pm(8,F)   &\SSpin^\pm(8,\R)   & \C^{30}&0\\[.5ex]
\PSO(8,F)   &\Spin(8,\R)   & \C^{33}&0\\
\end{array}
\]
\bigskip

    \caption{The $K$-theory of the reduced Iwahori-spherical algebra $\fA(\G)$ for $p$-adic groups of type $D_4$.}
\end{table}

Of course in this (and only this) rank we have an isomorphism
\[
\SO(8)\cong\SSpin^+(8)\cong\SSpin^-(8),
\]
however the identification of the Lie groups is induced by the triality automorphism of the Weyl group $W(D_4)$.  This obscures the homeomorphisms between sectors, since the cycle types and conjugacy classes are permuted. This is discussed in detail in Section \ref{D4 Lie groups}.

\bigskip

In constructing the $B_n, C_n$ and $D_n$ sectors it is often useful to compare the centraliser actions on the maximal tori of the corresponding types.  There are close parallels between these three classes that can be formalised by the following observation.  There is a $W(B_n)=W(C_n)$-equivariant identification of the maximal torus of the simply connected $C_n$ form with that of the adjoint $B_n$ form.  This allows us to regard the $B_n$ tower and the $C_n$ tower as a single object. The classical part of the $D_n$ tower (excluding the semi-spin groups) can then be $W(D_n)$-equivariantly identified with this object, allowing us to transfer computations between the three families.

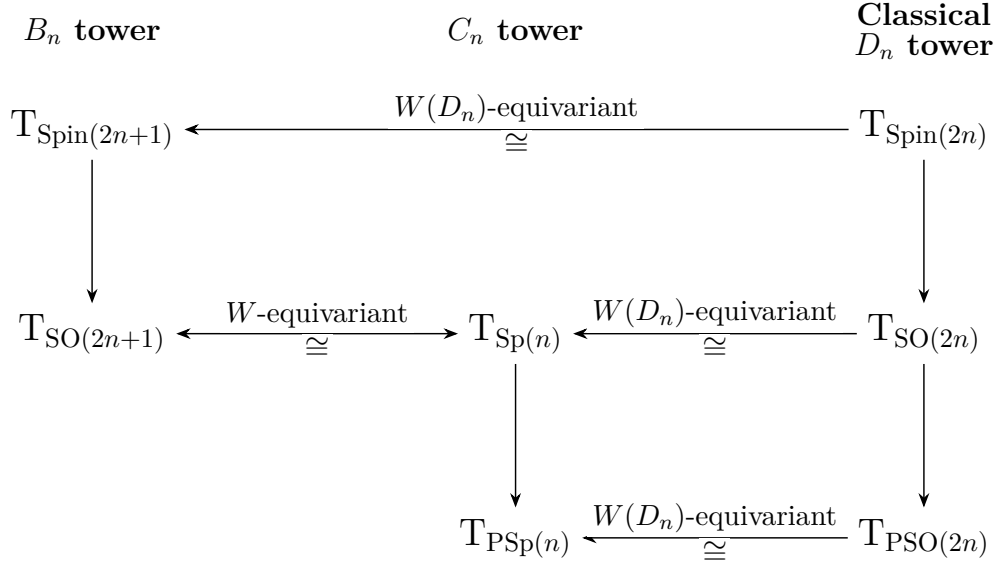
\begin{figure}[htbp]
\begin{center}
\begin{tikzpicture}[
  >=Stealth,
  node distance=1.2cm,
  tower/.style={align=center, font=\large},
  group/.style={font=\Large},
  arrowlabel/.style={font=\normalsize, fill=white, inner sep=2pt},
  eq/.style={font=\large, fill=white, inner sep=1pt},
  line/.style={->, line width=0.55pt},
  equiv/.style={->, line width=0.55pt},
  biequiv/.style={<->, line width=0.55pt}
]

\node[tower] (Bhead) at (0,9.0) {{\bf $B_n$ tower}};
\node[tower] (Chead) at (5.6,9.0) {{\bf $C_n$ tower}};
\node[tower] (Dhead) at (11.0,9.0) {{\bf Classical}\\[-1mm]{\bf $D_n$ tower}};

\node[group] (Bspin) at (0,7.7) {$\bT_{\mathrm{Spin}(2n+1)}$};
\node[group] (Bso)   at (0,5.0) {$\bT_{\mathrm{SO}(2n+1)}$};

\node[group] (Csp)   at (5.6,5.0) {$\bT_{\mathrm{Sp}(n)}$};
\node[group] (Cpso)  at (5.6,2.3) {$\bT_{\mathrm{PSp}(n)}$};

\node[group] (Dspin) at (11.0,7.7) {$\bT_{\mathrm{Spin}(2n)}$};
\node[group] (Dso)   at (11.0,5.0) {$\bT_{\mathrm{SO}(2n)}$};
\node[group] (Dpso)  at (11.0,2.3) {$\bT_{\mathrm{PSO}(2n)}$};

\draw[line] (Bspin) -- (Bso);
\draw[line] (Csp) -- (Cpso);
\draw[line] (Dspin) -- (Dso);
\draw[line] (Dso) -- (Dpso);

\draw[equiv] (Dspin.west) -- node[arrowlabel, above] {$W(D_n)$-equivariant} node[eq, below] {$\cong$} (Bspin.east);

\draw[biequiv] (Bso.east) -- node[arrowlabel, above] {$W$-equivariant} node[eq, below] {$\cong$} (Csp.west);

\draw[equiv] (Dso.west) -- node[arrowlabel, above] {$W(D_n)$-equivariant} node[eq, below] {$\cong$} (Csp.east);

\draw[equiv] (Dpso.west) -- node[arrowlabel, above] {$W(D_n)$-equivariant} node[eq, below] {$\cong$} (Cpso.east);

\end{tikzpicture}
\end{center}
\caption{The relationship between the towers of type $B_n$, $C_n$ and $D_n$}
    \label{fig:BCDTowers}
\end{figure}

Computing a sector is most easily done when the corresponding centraliser acts with a strict fundamental domain, since the sector can then be identified with this. In passing we introduce Lemma \ref{lemma0} that allows us to do this more broadly by suitably amplifying both the fixed set and the centraliser in such a way that the quotient remains unchanged. 
Applications of this idea include extending a centraliser from the $D_n$-Weyl group to the $B_n$-Weyl group or inflating the lattice to move down the tower. This is particularly helpful when computing $\Spin$ sectors by comparing them with the $\SO$ case.

The paper is organised as follows. In Section 2 we recall the basic structure of the signed permutation groups arising as 
Weyl groups of type $B_n, C_n$ and $D_n$, with a note on conjugacy classes in these groups. This is followed by a technical subsection in which we elucidate the structure of the quotient of a polysimplex (that is, a direct product of simplices) by a cellular involution. This will play a key role in constructing the sectors. In Section 3 we construct the sectors for the simply connected and adjoint real forms of the Lie groups of type $B_n$ and in Section 4 we do the same for type $C_n$. In both sections we analyse the sectors by conjugacy class representative, constructing the sectors of each of the two forms for each conjugacy class in turn, in terms of the lengths and multiplicities of the (positive and negative) cycles appearing in the conjugacy class description. 

In Section 5 we take the same approach to determine the sectors for the various real Lie groups of type $D_n$. Here there is a subtlety that for some cycle types there are two conjugacy classes rather than one.
Notwithstanding this, we again describe the sectors in terms of the length and multiplicity parameters arising in those cycle structures associated with the conjugacy classes. 

The results of Sections 3-5 are sufficient to provide a complete description of the topology of the sectors and thereby a parametrisation of the Iwahori-spherical block in each case. 

In the final section, we provide applications, beginning in Section \ref{D4 Lie groups} with a full description of the parameter space $\bT_G \q W$ for the groups of type $D_4$. In particular we explore the correspondence between sectors of $\SO(8)$ and the sectors (for sometimes different cycle types) of the isomorphic group $\SSpin^+(8)$. We use the description of the sectors to give a calculation of $K$-theory of the tempered Iwahori algebras corresponding to these groups.

In Section \ref{K-theory ranks} we use our results to give a formula for the ranks of the $K$-theory groups $K_*^W(C(\bT_G))$, where $G$ is $\Spin(2n),\SO(2n),\PSO(2n),\Spin(2n+1),\SO(2n+1),\Sp(n)$ and $\PSp(n)$.

In Section \ref{sector conjecture} we turn to our original homotopy sector conjecture, and discuss it in light of the results of Section  5. While the conjecture falters at this point we show that a slightly weaker result holds, the rational homotopy sector theorem, refining the stratified $K$-theory isomorphism established in \cite{NPWAn}.

Finally in Section \ref{injectivity of sectors} we introduce the concept of injectivity of sectors, proving that this implies continuity of the inverse of the ABPS map $\mu$ restricted to any such sector. In particular in the case of connected centre we establish a homeomorphism between each injective sector and its image in the tempered Iwahori spectrum. As an example of this we exhibit a homeomorphic copy of the $3$-sphere in the Iwahori spectrum for $\PSO(18,F)$. In Appendix \ref{appendixA} we give a table of sectors for $\Spin(18)$ and $\PSO(18)$. We also give an example of a sector which is not injective.

A list of tables is included at the end of the paper to aid the reader.

\section{Preliminaries}

\subsection{The extended quotient(s)}

The Baum-Connes conjecture posits an isomorphism between the equivariant $K$-theory of the classifying space $\underbar{EG}$ for proper actions of a discrete group $\Gamma$ and the $K$-theory of the reduced $C^*$-algebra of the group.  The presence of torsion in $\Gamma$ leads to non-trivial isotropy for the action on $\underbar{EG}$ which is invisible in the classical quotient, but plays an important role in the representation theory. In order to address this Baum and Connes introduced a modified quotient that keeps track of stabilisers and provides a better geometric model for the $K$-theory of the reduced $C^*$-algebra. This model has two forms.

We start by considering the spectral extended quotient (also known as the type-2 extended  quotient). It is constructed as follows:

Let $\Gamma$ be a discrete group acting properly on a Hausdorff topological space $X$. Let $I_2(\Gamma,X)$ be the set of pairs $(\rho, x)$ where $x\in X$ and $\rho$ is an irreducible representation of the isotropy group at $x$. Then $\Gamma$ acts on  $I_2(\Gamma,X)$ by translation on the second factor and by conjugation on the first, i.e., $g\cdot \rho(h) = \rho(g^{-1}hg)$. The spectral extended quotient, denoted $(X\q \Gamma)_2$,  is the quotient defined by this action. Note that there is a natural map from this extended quotient to the classical quotient $X/\Gamma$ given by projection onto the second factor. It is easy to see that the pre-image of a point $\Gamma x$ is parameterised by the equivalence classes of irreducible representations of the isotropy group at that point, which can be identified (though not in general in a natural way) with the conjugacy classes in the isotropy group. 

Alternatively we can build the geometric extended quotient. We define $I(\Gamma,X)$ to be the subspace of the Hausdorff space $\Gamma\times X$ given by  $\{(\gamma,x)\mid \gamma(x)=x\}$. The diagonal  action of $\Gamma$ on $\Gamma\times X$ defined by conjugation on the first factor and translation on the second, restricts to $I(\Gamma,X)$ and the geometric extended quotient of $X$ by $\Gamma$ is defined to be the corresponding quotient. At each point $x\in X$ we have a fibre given by pairs $(\gamma,x)$ modulo the conjugation action of the isotropy group.  As noted above this can be identified with the fibre in the spectral extended quotient.

The sectors are a feature of the geometric extended quotient. As well as admitting a projection onto the classical quotient $X/\Gamma$ by mapping onto the second factor, there is also a map to the set of conjugacy classes given by the map to the first factor.  The fibres of the latter map are the sectors.

\subsection{Signed permutation groups}\label{signed permutations}

The Weyl groups of types $B_n,C_n$ and $D_n$ are identified with groups of signed permutations, see \cite{carter}. For reference we give a quick introduction to the signed permutation groups in this section.

For a given $n$ the group of signed permutations of the set $\{\pm 1, \pm 2,\dots ,\pm n\}$ is the group of permutations of these $2n$ elements which commute with the involution $\iota:=(1\;{-1})(2\;{-2})\dots(n\;{-n})$.

The notion of cycle in the context of signed permutations differs from that in the symmetric group. Let $\rho$ be a signed permutation and consider its cycles as an element of the symmetric group $S_{2n}$. Since conjugation by $\iota$ fixes $\rho$ it permutes the cycles, indeed as $\iota$ is an involution each cycle is either fixed, or flipped with its negation.  The $\iota$-fixed cycles are called the \emph{negative cycles of $\rho$} while the products $\mu(\mu^\iota)$ for non-fixed cycles $\mu$ are called \emph{positive cycles of $\rho$}.  For example $(1\;{-2}\;3\;4\;{-1}\;{2}\;{-3}\;{-4})$ is a negative cycle, whereas $(1\;{-2}\;3\;4)({-1}\;{2}\;{-3}\;{-4})$ is a positive cycle. Both of these examples are said to have length $4$  as signed permutations.

Let $\sigma_p$ denote the negative cycle $(p\;{-p})$ for $p=1,\dots,n$, and for $J\subseteq\{1,\dots,n\}$ let $\sigma_J=\prod_{p\in J} \sigma_p$. The elements $\sigma_p$ generate a group isomorphic to $(\Z/2)^n$ and the group of signed permutations is isomorphic to $(\Z/2)^n\rtimes S_n$ (where of course $S_n$ permutes the factors of $(\Z/2)^n$). This is given by the isomorphism $\langle \sigma_p : p=1,2,\dots,n\rangle\cong (\Z/2)^n$, combined with the map taking the cycles $(p_1\;\dots\;p_k)$ in $S_n$ to the positive cycles $(p_1\;\dots\;p_k)({-p_1}\;\dots\;{-p_k})$. Note that here the numbers $p_1,\dots,p_k$ are all positive. The $S_n$ subgroup is exactly the unique special subgroup of type $A_{n-1}$ in the Weyl group of type $B_n$, or respectively $C_n$, and we refer to these positive elements as the \emph{special positive elements} of the signed permutation group.

Taking our examples from before the negative cycle can be written as 
\[(1\;{-2}\;3\;4\;{-1}\;{2}\;{-3}\;{-4})=(1\;{-1})(2\;{-2})(3\;{-3})(1\;2\;3\;4)({-1}\;{-2}\;{-3}\;{-4})\]
while the positive cycle is written as \[(1\;{-2}\;3\;4)({-1}\;{2}\;{-3}\;{-4})=(2\;{-2})(3\;{-3})(1\;2\;3\;4)({-1}\;{-2}\;{-3}\;{-4}).\]
The length of a (signed) cycle as a signed permutation is defined to be the length of the single cycle which is its image in the quotient $S_n$. We note that the positive cycles are those whose $(\Z/2)^n$ part is even (i.e.\ gives an even number of transpositions) while the negative cycles are those whose $(\Z/2)^n$ part is odd.

Given an $n$-dimensional vector space with a chosen basis $\{e_1,\dots e_n\}$ there is a canonical linear representation of the signed permutation group, where the elements of $S_n$ (in the semi-direct product picture) act by permuting the basis vectors, while the element $\sigma_p$ negates the basis vector $e_p$.  In this representation, the signed permutations act exactly as signed permutations on the set $\{\pm e_1,\dots \pm e_n\}$. As matrices the elements of the signed permutation group have a single non-zero entry in each row and column which can be either $+1$ or $-1$.

Let $\rho, \tau$ be signed permutations, and let $\mu$ be a cycle of $\rho$ as an element of $S_{2n}$.  Note that $(\mu^\tau)^\iota=(\mu^\iota)^\tau$ since $\tau$ commutes with $\iota$. It follows that conjugation of $\rho$ by $\tau$ takes positive cycles (which are pairs of $S_{2n}$ cycles) to positive cycles and negative cycles to negative cycles, and of course also preserves cycle lengths.  Hence, ordering the cycles by length, each conjugacy class has a well-defined cycle type,
\[
(n_1,\dots,n_1,\dots,n_l,\dots,n_l,\mbar_1,\dots,\mbar_1,\dots, \mbar_k,\dots,\mbar_k)
\]
where the numbers $n_1<n_2<\dots$ are the lengths (as defined above) of the positive cycles and the numbers $m_1<m_2<\dots$ are the lengths of negative cycles.  For brevity, and to explicitly indicate multiplicities we introduce the notation
\[
((n_1)^{d_1},\dots, (n_l)^{d_l},(\mbar_1)^{c_1},\dots, (\mbar_k)^{c_k})
\]
for the cycle type.

Moreover in the full group of signed permutations two elements with the same signed cycle type are conjugate hence these cycle types parametrise conjugacy classes.

\subsection{Involutions on products of simplices}\label{involutions}

In constructing the quotients of fixed sets in a torus by centralisers, one is often driven to consider the action of an involution on balls and spheres and their products. In this section we show how to simplify those calculations.

Let $\Delta^n$ be an $n$-simplex equipped with a simplicial involution $\sigma$.  The corresponding permutation of the vertex set is a product of disjoint transpositions and $1$-cycles.

Let $v_1^\pm,\dots v_q^\pm$ denote the pairs of transposed vertices (arbitrarily assigning the labels $+$ and $-$ to elements of each pair) and let $u_1,\dots,u_r$ denote the fixed vertices.

Let $m_i$ denote the midpoint of the $e_i:=[v_i^-,v_i^+]$ edge for $i=1,\dots, q$.

We think of the edge $e_1$ as $\Sigma C(\emptyset)=C\Sigma (\emptyset)$ where $C$ denotes the cone, $\Sigma$ the suspension\footnote{By convention the cone on the empty set consists of a single point (the cone point) and the suspension of the empty set consists of two points.}, the cone point is $m_1$ and the poles of the suspension are $v_1^\pm$. The join of the edges is then given by $(\Sigma C)^q(\emptyset)$ where the cone points are given by $m_1,\dots, m_q$, the poles are given by $v_1^\pm,\dots,v_q^\pm$ and the involution flips the poles of each suspension, while fixing the cone points. The simplex as a whole is thus given by $C^r(\Sigma C)^q(\emptyset)$.

We observe that cones and suspensions always commute.  Specifically for a space $X$ with involution, we define involutions on the cone and the suspension of $X$ by acting with the given involution on $X$, fixing the cone variable and negating the suspension variable.  Combining these actions we claim that $C\Sigma(X)$ is equivariantly homeomorphic to $\Sigma C(X)$.

The spaces  $\Sigma C(X)$ and $C\Sigma(X)$ are quotients of $X\times [0,1]\times [-1,1]$ and $X\times [-1,1]\times [0,1]$ respectively by appropriate relations.  The subtlety here is that the map $(x,t,s)\mapsto (x,s,t)$ does not respect these relations.

We define the following subspaces of the plane:
\begin{align*}
    \suspsusp&:=\{(t,s) : |s|+|t|\leq 1\}\\
    \suspcone&:=\{(t,s) \in \suspsusp : t\geq 0\}\\
    \conesusp&:=\{(t,s) \in \suspsusp: s\geq 0\}\\
    \conecone&:=\{(t,s) \in \suspsusp: s,t\geq 0\}
\end{align*}
We begin by noting that $\Sigma\Sigma(X)$ is homeomorphic to $(X\times \suspsusp)/\sim$ where $(x,t,s)\sim (x',t,s)$ for all $x,x'\in X$ when $|s|+|t|=1$: the homeomorphism is given by the map
\[
\Sigma \Sigma(X)\to (X\times \suspsusp)/\sim,\quad (x,t,s)\mapsto (x,(1-|s|)t,s)
\]
and this is equivariant, equipping $\suspsusp$ with the action negating both $s$ and $t$.  If we instead equip $\Sigma \Sigma(X)$ with an action fixing one or both pairs of suspension points then the map will again be equivariant if we adjust the action on $\suspsusp$ accordingly.

The above homeomorphism restricts to give equivariant homeomorphisms between $\Sigma C(X)$ and $(X\times \suspcone)/\sim$, and between $C\Sigma (X)$ and $(X\times \conesusp)/\sim$.  Note that here we equip $\suspcone$ with the action negating $s$ and $\conesusp$ with the action negating $t$.

The involution on $\suspsusp$ exchanging $s,t$ induces an equivariant identification of $(X\times \suspcone)/\sim$ with $(X\times \conesusp)/\sim$, and hence yields the required equivariant identification of $\Sigma C(X)$ and $C\Sigma (X)$.

Explicitly the composition is
\[
\Sigma C(X)\to C\Sigma (X), \quad (x,t,s)\mapsto (x,\frac{s}{1-(1-|s|)t},(1-|s|)t).
\]

Note that $CC(X)=\Sigma C(X) \cap C\Sigma (X)$ which is therefore identified with $(X\times \conecone)/\sim$ (since $\suspcone\cap \conesusp=\conecone$).

We now claim that $CC(X)$ is equivariantly homeomorphic to $C(X)\times [0,1]$ (with trivial action on $[0,1]$) and that $\Sigma C(X)$ is equivariantly homeomorphic to $C(X)\times [-1,1]$ (with negation action on $[-1,1]$).  We will want to consider products of spaces of this type with diagonal involution, and to that end it is simpler to work with the Cartesian product picture.

The homeomorphism for $CC(X)$ is defined by $\phi:C(X) \times [0,1] \to (X\times \conecone)/\sim$:
\[
\phi(x,t,s) = \begin{cases}
    (x,1-s,s-\frac 12 (1-t)) & s+t\geq 1\\
    (x,t,\frac 12 s) & s+t \leq 1.
\end{cases}
\]
It is then easy to construct the corresponding equivariant map $\psi:C(X) \times [-1,1] \to (X\times \suspcone)/\sim$:
\[
\psi(x,t,s) = \begin{cases}
    (x,1-s,s-\frac 12 (1-t)) & s\geq 1-t\\
    (x,t,\frac 12 s) & t-1\leq s\leq 1-t\\
    (x,1+s,s+\frac 12 (1-t)) & s\leq t-1.
\end{cases}
\]
We remark that if both $\suspcone$ and $[-1,1]$ are equipped instead with trivial actions then the map $\psi$ is again equivariant and similarly we can deduce that
$\Sigma C(X)$ where the action fixes cone points \emph{and suspension points} is equivariantly homeomorphic to $C(X)\times [-1,1]$ with trivial action on the second factor.

\bigskip

Returning to the case of a simplex $\Delta^n$ with involution transposing $q$ pairs of vertices and fixing $r$ vertices we now have
\[
\Delta^n \cong C^r(\Sigma C)^q(\emptyset)\cong C^{q+r}\Sigma^q(\emptyset).
\]
Since $CC(X)\cong C(X)\times I$, if $q+r\geq 2$ we have 
\[
C^{q+r}\Sigma^q(\emptyset)\cong CCC^{q+r-2}\Sigma^q(\emptyset)\cong CC^{q+r-2}\Sigma^q(\emptyset)\times [0,1]
\]
and we can repeat this to obtain $C\Sigma^q(\emptyset)\times [0,1]^{q+r-1}$.  The identification of $C^{q+r}\Sigma^q(\emptyset)$ with the latter also holds trivially if $q+r=1$. Recall that the involution on the $[0,1]$ factors is trivial.

Similarly $C\Sigma^q(\emptyset)\cong C\Sigma^{q-1}(\emptyset)\times [-1,1]$ and repeating this we obtain $C(\emptyset)\times [-1,1]^q$.

Thus we conclude that the simplex $\Delta^n$ is equivariantly homeomorphic to the cuboid $[-1,1]^q\times [0,1]^{p+q-1}$ where the involution acts diagonally by negation on the cube $[-1,1]^q$ and trivially on the second cube. The quotient space is therefore the product $C(\RP^{q-1})\times [0,1]^{q+r-1}$.

The advantage of this perspective is that it immediately extends to the case of a product of simplices $\Delta_1\times\dots\times \Delta_k$ equipped with an involution acting diagonally. Writing $q_1,r_1\dots, q_k,r_k$ for the corresponding numbers of transposed pairs and fixed vertices for the factors of the involution, we obtain an equivariant identification with 
\[
\prod_{i=1}^k [-1,1]^{q_i}\times [0,1]^{q_i+r_i-1}
\]
where the involution acts diagonally on all $[-1,1]$ factors.  This yields the quotient $C(\RP^{q-1})\times [0,1]^{q+r-k}$, where $q=\sum\limits_{i=1}^k q_i$ and $r=\sum\limits_{i=1}^k r_i$.

Setting $m=q+r-k$ and $n=q$ this quotient is of course homeomorphic to the quotient of a ball of dimension $m+n$ by the standard `reflection' in a subspace of dimension $m$, i.e., by the involution with signature $(0, m,n)$. Given the prevalence of quotients of this form in the calculations below, we introduce the notation $V(m,n)$ for this singular space, noting that it is a manifold away from the singular locus of dimension $m$, each point of which has link homeomorphic to a real projective space of dimension $n-1$.

\section{The $B_{n}$ case}\label{B_n case}

\subsection{Roots, Weyl group and maximal tori}

Let $\t\cong \R^n$ be a Cartan subalgebra in the Lie algebra of type $B_n$. The Dynkin diagram for $B_n$ has $n-1$ long roots and one short root:
\[
r_1=(1,-1,0,\dots,0),\dots r_{n-1}=(0,\dots,0,1,-1), r_n=(0,\dots ,0,1)\in \t^*,
\]
where we identify $\t^*$ with $\t$ using the inner product.  With this identification the coroots are $r_i^\vee = r_i$ for $i=1,\dots n-1$ and $r_n^\vee=2r_n$.

The Weyl group $W=W(B_n)$ is the group of signed permutations with its canonical representation on the coordinate basis of $\R^n$:
\[
e_1=(1,0\dots,0),\dots,e_n=(0,\dots,0,1)=r_n.
\]

\bigskip

For $B_n$ the connection index is $2$ and hence there are two compact connected semisimple Lie groups of type $B_n$.  Specifically there are $\Spin(2n+1)$ (simply connected) and $\SO(2n+1)$ (adjoint).

The nodal lattice for a maximal torus in the simply connected $B_n$ group is the even parity lattice in $\Z^n$.  On the other hand the nodal lattice for the adjoint case $\SO(2n+1)$ is given by vectors having integer pairing with the root lattice, hence this is the full integer lattice. In the latter case the maximal torus, which we denote by $T$, is thus naturally identified with $\TT^n$.

We can also identify $\t$ with the subspace $\R^{n+1}_0$ of sum zero vectors in $\R^{n+1}$, extending $r_1^\vee\dots,r_{n-1}^\vee$ by $0$ and $r_n^\vee$ by $-2$ in the final coordinate. The nodal lattice in the simply connected case is then identified with $(\Z^n\times 2\Z)\cap \R^{n+1}_0$ and hence the maximal torus is identified with
\[
\coveringtorus=\{(\alpha_1,\dots,\alpha_n,\beta)\in \TT^{n+1} : \alpha_1\dots\alpha_n\beta^2=1\}.
\]
The map $\R^{n+1}_0\to \coveringtorus$ is given by
\[
(x_1,\dots,x_n,y)\mapsto (e^{2\pi i x_1},\dots,e^{2\pi ix_n},e^{1\pi iy}).
\]
Note that the maximal torus $T$ for the \emph{adjoint} type group $\SO(2n+1)$ is double covered by $\coveringtorus$, via the map which omits the $\beta$ coordinate.

\bigskip

The action of the permutation subgroup of $W$ on $\R^n$ extends to $\R^{n+1}$ acting trivially on the final factor and this preserves the subspace $\R^{n+1}_0$. The action of the negative cycle $\sigma_p$, which corresponds to a reflection in the short root $e_p$, extends to
\[
\tilde{\sigma}_p:\R^{n+1}_0\to \R^{n+1}_0,\quad (x_1,\dots,x_n,y) \mapsto (x_1,\dots, -x_p,\dots,x_n,y+2x_p)
\]
and the corresponding map on the torus $\coveringtorus$ is given by 
\[
(\alpha_1,\dots,\alpha_n,\beta)\mapsto (\alpha_1,\dots,\alpha_p^{-1},\dots,\alpha_n,\beta\alpha_p).
\]
Note that the final term only has a factor of $\alpha_p$ not $\alpha_p^2$ due to the construction of the map $\R^{n+1}_0\to \coveringtorus$.

This action is constructed to be that induced by the above identification of $\coveringtorus$ with the maximal torus in the simply connected $B_n$ group.

\bigskip

\subsection{Fixed sets and actions of centralisers}\label{B_n fixed sets} For each conjugacy class in $W$ we take a representative $w$ and we will consider its fixed sets in the tori $\coveringtorus, T$ and the centraliser of $w$ in $W$. Using the notation from Section \ref{signed permutations}
let $w$ be a representative of the conjugacy class of type
    \[
((n_1)^{d_1},\dots (n_l)^{d_l},(\mbar_1)^{c_1},\dots (\mbar_k)^{c_k}).
\]
For simplicity in each class we choose a representative $w$ so that its positive cycles are special positive elements in the sense of Section \ref{signed permutations}. The element $w$ determines a partition of the basis of $\R^n$ as follows: Let $I_{n_i,j}$ denote the indices of coordinates corresponding to the $j$th $n_i$-cycle ($j=1,\dots d_i$), and similarly for $I_{\mbar_i,j}$. Clearly for each corresponding subspace, the element which negates the corresponding basis vectors (i.e.\ the product of negative $1$-cycles) commutes with $w$.

The centraliser of $w$ is isomorphic to the product of factors:
\[
\begin{cases}
(\langle \sigma_{I_{n_i,j}}\rangle \times (\Z/n_i\Z))\wr_{d_i}S_{d_i} & \text{for each positive cycle}\\
(\Z/2m_i\Z)\wr_{c_i}S_{c_i}& \text{for each negative cycle.}
\end{cases}
\]
Note that a negative cycle of length $m_i$ has order $2m_i$ with the $m_i$th power yielding the $-1$ element on the corresponding subspace.

\bigskip

In the Lie algebra (here viewed as $\R^n$) a vector $(x_1,\dots x_n)$ is fixed by $w$ if
\begin{itemize}
    \item for all $i,j$ and $p,q\in I_{n_i,j}$ we have $x_p=x_q$;
    \item for all $i,j$ and $p\in I_{\mbar_i,j}$ we have $x_p=0$.
\end{itemize}
In the adjoint case a point $(\alpha_1,\dots,\alpha_n)$ of the maximal torus  is fixed if
\begin{itemize}
    \item for all $i,j$ and $p,q\in I_{n_i,j}$ we have $\alpha_p=\alpha_q$;
    \item for all $i,j$ and $p,q\in I_{\mbar_i,j}$ we have $\alpha_p=\alpha_q=\pm 1$.
\end{itemize}
The second condition follows from the requirements that $\alpha_p=\alpha_p^{-1}$ (since the point must be fixed by the element $w^{m_i}$) and then given that each $\alpha_p$ is unchanged by inversion the element $w$ permutes the coordinates $I_{\mbar_i,j}$.

Since the values of $\alpha_p$ are constant over each $I_{n_i,j}$ we can introduce notation $\alpha_{n_i,j}$ for the common value and similarly we write $\alpha_{\mbar_i,j}$ for the value over $I_{\mbar_i,j}$.
\bigskip

In the simply connected case a point $(\alpha_1,\dots,\alpha_n,\beta)$ (such that $\alpha_1\dots\alpha_n\beta^2=1$) of the maximal torus $\coveringtorus$ is fixed if
\begin{itemize}
    \item for all $i,j$ and $p,q\in I_{n_i,j}$ we have $\alpha_p=\alpha_q$;
    \item for all $i,j$ and $p,q\in I_{\mbar_i,j}$ we have $\alpha_p=\alpha_q=\pm 1$.
    \item the product $\prod_{i=1}^k \prod_{j=1}^{c_i}\alpha_{\mbar_i,j}=1$, where as before $\alpha_{\mbar_i,j}$ denotes the value of $\alpha_p$ for any $p$ in $I_{\mbar_i,j}$.
\end{itemize}
The third condition follows from the fact that the $\beta$ coordinate is fixed, while each negative cycle has the effect of multiplying $\beta$ by $\alpha_{\mbar_i,j}$.

\bigskip

Note that in the adjoint case the fixed set has $\prod_{i=1}^k 2^{c_i}$ components arising from the choices of $\pm1$ for $\alpha_{\mbar_i,j}$. On the other hand in the simply connected case there are half as many possible choices for these values since the product must be one; however, there is additionally the question of choosing a square root $\beta$ of $\prod_{p=1}^n \alpha_p^{-1}$. The choice of $\beta$ may or may not contribute additional components.

We rewrite the value of $\beta^2$ as $(\prod_{i=1}^l\prod_{j=1}^{d_i}\alpha_{n_i,j}^{-n_i})(\prod_{i=1}^k\prod_{j=1}^{c_i}\alpha_{\mbar_i,j}^{-m_i})$.  If any of the values $n_i$ is odd then traversing a loop for any of the corresponding $\alpha_{n_i,j}$ variables has the effect of negating the square root $\beta$.  Hence in these cases the choice of $\beta$ does not yield additional components so there are half as many components as in the adjoint case.

On the other hand when all the values $n_i$ are even we set $\gamma=\beta\prod_{i=1}^l\prod_{j=1}^{d_i}\alpha_{n_i,j}^{n_i/2}$. 
Then $\gamma^2=\prod_{i=1}^k\prod_{j=1}^{c_i}\alpha_{\mbar_i,j}^{-m_i}$ (which is $\pm 1$) hence the choice of $\gamma$ doubles the number of components, recovering the same value as in the adjoint case.

\bigskip

\subsection{The sectors}\label{B_n sectors} Now we turn to the question of the number of components in the quotient.  For the adjoint case $\SO(2n+1)$ this number was computed by Solleveld \cite{Sol}. We repeat this here for completeness:  The action of the centraliser on the coordinates $\alpha_{\mbar_i,j}$, which parameterise the components of the fixed set, simply permutes $\alpha_{\mbar_i,1},\dots \alpha_{\mbar_i,c_i}$.  Hence what remains in the quotient is the number of $-1$ terms for each $i$ which ranges from $0$ to $c_i$.  The number of components in the quotient is therefore $\prod_{i=1}^k(c_i+1)$.

In the simply connected case we additionally have the constraint that the product $\prod_{j=1}^{c_i}\alpha_{\mbar_i,j}=1$ which means that while the number of $-1$ terms can again range from $0$ to $c_i$ for each $i$, the total number must be even. This yields a number of components arising from the $\alpha_{\mbar_i,j}$ choices which is the number of tuples $(a_1,\dots,a_k)$ with $a_i\in \{0,\dots,c_i\}$ and with $\sum_{i} a_i$ even:
\[
\left\lceil \frac 12 \prod_{i=1}^k (c_i+1)\right\rceil.
\]
This formula is obtained from the observation that if any $c_i$ is odd and hence the product is even, then the number of choices of $(a_1,\dots,a_k)$ yielding an odd sum or an even sum is the same, while if each $c_i$ is even then there is exactly one more even choice than odd choice.

In the case that there are any positive cycles of odd length this formula gives the total number of components since the choice of $\beta$ does not yield additional components.  In the case that there are only even length positive cycles the total number of components is
\[
\left\lceil \frac 12 \prod_{i=1}^k (c_i+1)\right\rceil+1
\]
because when any $\alpha_{\mbar_i,j}$ is $-1$ then the action of the corresponding cycle, which lies in the centraliser, has the effect of negating $\beta$ and thus the $\beta$ again does not give additional components in the quotient.  However when $\alpha_{\mbar_i,j}=1$ for all $i,j$ the centraliser does not change $\beta$ and thus there is one additional component arising from the choice of $\beta=\pm1$.

\bigskip

We now turn to examining the components of the quotient. In the adjoint case, as observed by Solleveld, the components of the quotient are all products of simplices, so we will focus on the simply connected case.

Given an element $w\in W(B_n)$, we split the indices $\{1,\dots,n\}$ into the sets
\begin{itemize}
    \item $I^+_{\ev}=\bigcup\limits_{n_i \text{ even}}\bigcup\limits_j I_{n_i,j}$
    
    \item $I^+_{\odd}=\bigcup\limits_{n_i \text{ odd}}\bigcup\limits_jI_{n_i,j}$

    \item $I^-=\bigcup\limits_{m_i}\bigcup\limits_j I_{\mbar_i,j}$.
\end{itemize}

For $X\subseteq \{1,\dots, n\}$, we let $W(B_X)$ denote the copy of $W(B_{|X|})$ in $W(B_n)$ corresponding to this subset, using the usual identification of these groups as signed permutations. The splitting of $\{1,\dots, n\}$ yields an inclusion
\[
W(B_{I^+_{\ev}})\times W(B_{I^+_{\odd}})\times W(B_{I^-}) \hookrightarrow W(B_n)
\]
such that $w$ is the image of an element on the left-hand side which we denote by $(w_{\ev},w_{\odd},w_\neg)$. The centraliser of $w$ in $W(B_n)$ is identified with the product \[
Z_{W(B_{I^+_{\ev}})}(w_{\ev})\times Z_{W(B_{I^+_{\odd}})}(w_{\odd})\times Z_{W(B_{I^-})}(w_\neg).
\]

It turns out that to decompose the torus it is better to use $I^+_{\ev}\cup J$ where $J=I^+_{\odd}\cup I^-$: the reason for this is that both $Z_{W(B_{I^+_{\odd}})}(w_{\odd})$ and $Z_{W(B_{I^-})}(w_\neg)$ contain elements which act on the $\beta$ coordinate by multiplication by odd powers of the corresponding $\alpha$ coordinate.

We write $(w_{\ev},w_{J})$ for the corresponding decomposition of $w$, and the centraliser splits as $Z_{W(B_{I^+_{\ev}})}(w_{\ev})\times Z_{W(B_{J})}(w_{J})$.  The fixed set $\coveringtorus^w$ can be identified with $T_\ev^{w_{\ev}}\times \coveringtorus^{w_J}_J$, where $T_\ev$ is the standard maximal torus in $\SO_{I^+_{\ev}}(2n+1)$ while $\coveringtorus_J$ is the standard maximal torus in $\Spin_J(2n+1)$ (here $\SO_X(2n+1),\Spin_X(2n+1)$ denotes the matrix or spin group for the corresponding subspace of $\R^{2n+1}$). The identification is given by the map
\[
(\alpha_1,\dots,\alpha_n,\beta) \mapsto \big((\alpha)_{\ev},((\alpha)_{\odd},(\alpha)_-,\gamma)\big)
\]
where
\[
\gamma=\beta\prod_{\substack{i \in \{1,\dots,l\}\\n_i\text{ even}}}\prod_{j=1}^{d_i}\alpha_{n_i,j}^{n_i/2}.
\]
We will show that this is equivariant.  It is easy to see that it is equivariant for the second factor $Z_{W(B_{J})}(w_J)$. The first factor $Z_{W(B_{I^+_{\ev}})}(w_{\ev})$ is generated by three types of elements:  we have cycles from $w_{\ev}$; the involution $\prod_{p\in I_{n_i,j}} \sigma_p$ for each $i,j$; and permutations of coordinates in cycles of a given length. It is again easy to see that the map is equivariant for elements of the first and third type, since these will fix both $\beta$ and $\gamma$. An element of the second type $\prod_{p\in I_{n_i,j}} \sigma_p$ acts to invert the coordinate $\alpha_{n_i,j}$ and to multiply $\beta$ by $\alpha_{n_i,j}^{n_i}$. Hence these elements also fix $\gamma$ and thus the map is indeed equivariant as claimed.

Looking at the first factor where we have $Z_{W(B_{I^+_{\ev}})}(w_{\ev})$ acting on $T_\ev^{w_{\ev}}$, the cycles of $w_\ev$ act trivially.  This leaves the action on the coordinates $(\alpha_{n_i,1},\dots, \alpha_{n_i,d_i})$ of the signed permutation group $W(B_{d_i})$, or equivalently $W(C_{d_i})$.  Since these coordinates correspond to the simply connected $C_{d_i}$ group the quotient must be a $d_i$ simplex.  It follows that $T_\ev^{w_{\ev}}/Z_{W(B_{I^+_{\ev}})}(w_{\ev})$ is the product over $i$ of $d_i$-simplices (for even $n_i$).

\bigskip

We now split into cases depending on whether or not $I^+_{\odd}$ is empty. When this is empty the fixed set $\coveringtorus_J^{w_J}$ is finite; this is the component group of $\coveringtorus^w$.  As discussed above the quotient has cardinality
    \[
\left\lceil \frac 12 \prod_{i=1}^k (c_i+1)\right\rceil+1.
\]

When $I^+_{\odd}$ is non-empty we begin with the identity component of the group $\coveringtorus_J^w$. This consists of those points of $\coveringtorus_J^w$ for which $\alpha_{\mbar_i,j}=1$ for all $i,j$, and we can therefore identify it with $\t_{\odd}^w / (\t_{\odd}^w\cap \Gamma_{\odd})$  where $\t_{\odd}$ and $\Gamma_{\odd}$ denote the Cartan subalgebra and nodal lattice for $\Spin_{I^+_\odd}(2n+1)$.  Here we will use the model of $\t_{\odd}$ as $\R^{I^+_\odd}$ and of $\Gamma_{\odd}$ as the even parity lattice in $\Z^{I^+_\odd}$.

For a vector in $\t_\odd$ to lie in the fixed set the values of $x_p$ must be constant over $p\in I_{n_i,j}$ for odd $n_i$. As in the earlier discussion for the coordinates $\alpha_p$, we introduce the notation $x_{n_i,j}$ for the common value of $x_p, p\in I_{n_i,j}$. An element of the fixed set lies in the lattice if each $x_{n_i,j}$ is an integer and the sum $\sum\limits_{n_i \text{ odd}} \sum\limits_j x_{n_i,j}$ is even. Hence the identity component of the fixed set is given by the quotient of $\t_\odd^w$ by this lattice.

For the identity component, the task is therefore to compute the quotient of $\t_\odd^w$ by the group $H:=(\t_\odd^w\cap \Gamma_\odd)\rtimes Z_{W(B_{I^+_\odd})}(w_\odd)$. To this end we introduce the following lemma.

\begin{lemma}\label{lemma0}
    Let $X$ be a topological space, let $G$ be a group $G$ acting on $X$, and let $H$ be a subgroup of $G$, not necessarily normal. Then there is a natural homeomorphism
\[
H\backslash X \cong G \backslash(G/H \times X).
\]
\end{lemma}

The point of this identity is that it will in certain circumstances allow us to replace the action of a group $H$ on $X$ which does not admit a \emph{strict} fundamental domain, with a more tractable action of a larger group $G$ whose action on $X$ does admit a strict fundamental domain.

\begin{proof}
    Consider two copies of $G\times X$ with the following commuting actions: On the first copy we let $H$ act on the right of $G$ and by the restriction of the given action on $X$ while $G$ acts on the left of $G$ and acts trivially on $X$.

    On the second copy let $H$ act on the right of $G$ but act trivially on $X$ while $G$ acts on the left of $G$ and by the given action on $X$.

    The map $(g,x)\mapsto (g,gx)$ gives a $G\times H$-equivariant identification of the two copies of $G\times X$. The quotient of the first copy by both groups yields $H\backslash X$ while the quotient of the second by definition gives $G\backslash(G/H\times X)$.
\end{proof}

To apply the lemma we introduce the larger group $G=(\t_\odd^w\cap \Z^{I^+_\odd})\rtimes Z_{W(B_{I^+_\odd})}(w_\odd)$, recalling that $\t_\odd^w\cap \Gamma_\odd$ is the even parity lattice in $\t_\odd^w\cap \Z^{I^+_\odd}$.

Let $N$ denote the normal subgroup of $G$ generated by the lattice along with the involutions inverting each $x_{n_i,j}$ (for $n_i$ odd). The action of $N$ on $\t_\odd^w$ admits a strict fundamental domain (unlike the action of $N\cap H$) which is the cube given by $x_{n_i,j}\in [0,\frac 12]$ for each $i,j$. Hence the product of $G/H$ with the cube gives a fundamental domain for the action of $N$ on $G/H\times \t_\odd^w$.  This is no longer strict; specifically, we consider those elements of $N$ which are reflections in the (codimension $1$) faces of the cube.  The reflections in faces containing the origin all lie in $N\cap H$ while reflections in faces containing $(\frac 12,\dots,\frac 12)$ do not.  Hence the two copies of the latter faces, indexed by $G/H$, will be identified. The resulting space is thus the cube doubled over the union of faces, $U$, containing $(\frac 12,\dots,\frac12)$, which is the suspension of $U$.

It remains to act with the group $G/N=\prod\limits_{n_i\text{ odd}}\prod_j S_{d_i}$. This again admits a strict fundamental domain in the doubled cube, which is the suspension of the corresponding (strict) fundamental domain in $U$.  We can identify $U$ with its image in $\t^w_\odd/(1,1,\dots,1)\R$ which is a convex subset of this space.  The fundamental domain for the action on $U$ is itself identified as a convex set with non-empty interior, and hence the quotient is a product of simplices (the fundamental domain for the action on one copy of the cube) doubled over this convex set.

\bigskip

 Turning to the non-identity components of  $\coveringtorus_J^w$, these are parameterised by the values $\alpha_{\mbar_i,j}$, with at least one coordinate such that $\alpha_{\mbar_i,j}=-1$.  The corresponding negative cycle in the centraliser has the effect of negating the value of $\gamma$ in the parametrisation of the torus, identifying the component of $\coveringtorus_J^w$ with a component of $T_J^w$.  The quotient of the component by its stabiliser is thus the same as for the corresponding component in the adjoint $B_n$ case, hence this yields a product of $d_i$-simplices.

Hence we have $\left\lceil \frac 12 \prod_{i=1}^k (c_i+1)\right\rceil-1$ components which are products of simplices along with one component which is a doubled product of simplices.

The whole quotient is given by the product of $T_\ev^{w_{\ev}}/Z_{W(B_{I^+_{\ev}})}(w_{\ev})$ with $\coveringtorus_J^w/Z_{W(B_{J})}(w_{J})$. 
Each component is a ball of dimension $\delta:=
\sum_i d_i$. When $I^+_{\odd}$ is empty the components are products of simplices
\[
\prod_{i=1}^l\Delta^{d_i}
\]
while when $I^+_{\odd}$ is non-empty, the non-identity components are again polysimplices while the identity component has twice the volume, being the union of two polysimplices over a `convex' region in their boundary.

To summarise, the structure of the sector associated to a conjugacy class with negative multiplicities $c_i$ for the simply connected case of $B_n$ is given in the following table. Note that, as in Solleveld's calculation for the adjoint form, the components are all contractible.

\begin{table}[H]
\[
\begin{array}{c|c}
   C:=\prod_{i=1}^k (c_i+1) &\\
   &\\
 \hline
 I^+_\odd \text{ empty}   & (\lceil \frac 12 C \rceil +1)\times V(\delta,0)   
 \\
 \\
  I^+_\odd \text{ non-empty}    & \lceil \frac 12 C \rceil\times V(\delta,0)
\end{array}
\]

\bigskip

    \caption{\label{simply_connected_B_n_sectors}The structure of a sector for the simply connected form of $B_n$}
\end{table}

\section{The $C_n$ case}\label{C_n case}

\subsection{Roots, Weyl group, and maximal tori}
Let $\t\cong \R^n$ be a Cartan subalgebra in the Lie algebra of type $C_n$. The Dynkin diagram for $C_n$ has $n-1$ short roots and one long root:
\[
r_1=(1,-1,0,\dots,0),\dots r_{n-1}=(0,\dots,0,1,-1), r_n=(0,\dots ,0,2)\in \t^*,
\]
where we identify $\t^*$ with $\t$ using the inner product.  With this identification the coroots are $r_i^\vee = r_i$ for $i=1,\dots n-1$ and $r_n^\vee=\frac 12 r_n=e_n$.  Since this root system is the dual of the $B_n$ root system the Weyl groups are canonically identified.

\bigskip

For $C_n$ the connection index is again $2$ and hence there are two compact connected semisimple Lie groups of type $C_n$.  Specifically there are $\Sp(n)$ (simply connected) and $\PSp(n)$ (adjoint).
\bigskip

The nodal lattice for a maximal torus in the simply connected $C_n$ group is the lattice  $\Z^n$.  On the other hand the nodal lattice for the adjoint case is given by elements having integer pairing with the root lattice. This is therefore the sum of the integer lattice and multiples of the element $(\frac 12,\dots,\frac 12)$.

    The torus in the simply connected case is therefore just the quotient $T=\R^n/\Z^n \cong \TT^n$. Moreover this is $W$-equivariantly identified with the torus in the adjoint $B_n$ case, where $W$ is the Weyl group of type $B_n$ or equivalently of type $C_n$.  Indeed in both cases the Weyl group acts by signed permutations on the coordinates of the Lie algebra (which is identified with $\R^n$) and the torus is obtained as the quotient by the integer lattice.

    In the adjoint case the torus is $T/\cZ$ where $\cZ$ is the centre of the simply connected Lie group of type $C_n$. The centre is given by plus or minus the identity matrix.  This acts on the torus $T$ by negating all coordinates, and on the Lie algebra this corresponds to shifting by the diagonal element $( \frac 12,\dots,\frac 12)$.
    
    \subsection{Fixed sets and actions of centralisers}
    Given the identification of the adjoint $B_n$ case with the simply connected $C_n$ case, the fixed sets for the action on the corresponding torus $T$ are as discussed in Section \ref{B_n case}.

Using the notation from Section \ref{signed permutations}
let $w$ be a representative of the conjugacy class of type
    \[
((n_1)^{d_1},\dots (n_l)^{d_l},(\mbar_1)^{c_1},\dots (\mbar_k)^{c_k}).
\]
As before we make the convention that positive cycles take basis vectors to basis vectors.
To examine the $w$-fixed set in $T/\cZ$ we consider its preimage in $T$.  This consists of points which are fixed by $w$ in $T$ as above, along with the set of points which are negated by $w$, which we denote by $A^w$. The fixed set $(T/\cZ)^w$ is the image in $T/\cZ$ of $T^w\cup A^w$.

The condition (defining $A^w$) that $w$ acts by negating coordinates of the torus means that for each positive cycle $(e_{p_1}\; \dots \; e_{p_{n_i}})(-e_{p_1}\; \dots \; -e_{p_{n_i}})$ of $w$ each point of $A^w$ must then satisfy
\[
(\alpha_{p_2},\dots,\alpha_{p_{n_i}},\alpha_{p_1})=-(\alpha_{p_1},\dots,\alpha_{p_{n_i}}).
\]
From this it follows that $\alpha_{p_1}=(-1)^{n_i}\alpha_{p_1}$ hence such points can exist only if $n_i$ is even for each $i$.

For a negative cycle of length $m_i$ in $w$, the condition that $w$ negates a point implies that $w^{m_i}$ changes each coordinate by $(-1)^{m_i}$. On the other hand $w^{m_i}$ acts by inverting all coordinates of the negative $m_i$ cycle. Thus each of these coordinates satisfies $(-1)^{m_i}\alpha=\alpha^{-1}$, so we require $\alpha=\pm 1$ if $m_i$ is even and $\alpha=\pm \sqrt{-1}$ if $m_i$ is odd.

Hence we deduce that if all $n_i$ are even then $A^w$ consists of points whose coordinates within each factor corresponding to $I_{n_i,j}$ or $I_{\mbar_i,j}$ are alternating $(\alpha,-\alpha,\dots)$, and with the additional constraint that within the $I_{\mbar_i,j}$ factors the values are either $\{\pm1\}$ or $\{\pm \sqrt{-1}\}$ depending on the parity of $m_i$. In the case that any $n_i$ is odd then $A^w$ is empty.

As in the $B_n$ case, given an element $w$, we can split the indices $\{1,\dots,n\}$ into the sets $I^+_{\ev},I^+_{\odd}, I^-$ corresponding respectively to the positive even cycles, positive odd cycles, and negative cycles.  The above discussion shows that $A^w$ is non-empty exactly when $I^+_{\odd}$ is empty. In this case the centraliser of $w$ factors as
\[
Z_{W(C_n)}=Z_{W(C_{I^+_{\ev}})}(w_{\ev})\times Z_{W(C_{I^-})}(w_\neg).
\]
Moreover there is a factorisation of $A^w$ as $A^{w_\ev}\times A^{w_\neg}$ corresponding to the splitting of indices and the factors of the centraliser act on the factors of $A^w$.

The element $w_\ev$ consists of positive even length cycles. For each cycle $(p_1\,p_2\,\dots\,p_{n_i})$ of $w_\ev$, we take alternate indices to form a cycle $(p_2\,p_4\,\dots\,p_{n_i})$ of length $n_i/2$, and we denote by $w_{ev,0}$ the product of these cycles.  Let $I^+_{\ev,0}$ denote the collection of all indices appearing in the chosen $n_i/2$ cycles.  The element $w_{ev,0}$ lies in the group $W(C_{I^+_{\ev,0}})$, and we define a map from $T_\ev$ to the standard maximal torus $T_{\ev,0}$ in $\Sp_{I^+_{\ev,0}}(n)$ as follows.  For each cycle  $(p_1\,p_2\,\dots\,p_{n_i})$ of $w_\ev$, we take the coordinates $(\alpha_{p_1},\dots,\alpha_{p_{n_i}})$ of a point in $T_\ev$, and define the $p_2,p_4,\dots,p_{n_i}$ coordinates of its image in $T_{\ev,0}$ to be the tuple $(\alpha_{p_1}\alpha_{p_2},\dots,\alpha_{p_{n_i-1}}\alpha_{p_{n_i}})$.

We claim the following:  The above map takes $A^{w_{\ev}}$ to $T_{\ev,0}^{w_{ev,0}}$.  Moreover $Z_{W(C_{I^+_{\ev}})}(w_{\ev})$ has the form $(N_1\times N_2)\rtimes Q$ where
\begin{itemize}
    \item $N_1$ is the group generated by the cycles of $w_\ev$;
    \item $N_2$ is the group generated by the $-1$ element on each cycle;
    \item $Q$ is the group of permutations of cycles of the same length.
\end{itemize}
This can be rewritten as $N_1\rtimes (N_2\rtimes Q)$, and the groups $N_2$, $Q$ are canonically identified with subgroups of the centraliser $Z_{W(C_{I^+_{\ev,0}})}(w_{\ev,0})$. We claim that the map $A^{w_{\ev}}/N_1$ to $T_{\ev,0}^{w_{ev,0}}$ is a bijection and moreover that this is $N_2\rtimes Q$ equivariant.

\subsection{The sectors} For the simply connected case, as discussed in Section \ref{B_n case}, for each $w$ the quotient of the fixed set $T^w$ by the centraliser yields $\prod_i(c_i+1)$ components, each of which is a product of simplices where the numbers $c_i$ are the multiplicities of the negative cycles of $w$.

For the adjoint form we consider the quotient of the fixed set $(T/\cZ)^w=T^w/\cZ \cup A^w/\cZ$ by the action of the centraliser.
We begin by considering the quotient of $T^w /\cZ$ by the centraliser $Z(w)$, which is the quotient of $T^w/Z(w)$ by the action of $\cZ$.  Recall from Section \ref{B_n case} that the components of $T^w/Z(w)$ are indexed by the values $\alpha_{\mbar_i,j}$, corresponding to the negative cycles, which may take the value $\pm1$, and that two such tuples of values give the same component in the quotient if they have the same number of $-1$ values amongst $\alpha_{\mbar_i,1},\dots,\alpha_{\mbar_i,c_i}$ for each $i$.

The action of $\cZ$ has the effect of exchanging the signs of each $\alpha_{\mbar_i,j}$ and hence changes the number $a_i$ of $-1$ values to $c_i-a_i$. Hence this action identifies components in pairs except possibly for a component having $a_i=c_i/2$ for each $i$.  This component exists precisely when all $c_i$ values are even.

Hence if any $c_i$ is odd, the quotient $(T^w/\cZ)/Z(w)$ has exactly \mbox{$\frac 12 C$} components, where, as in Section \ref{B_n sectors}, we set $C=\prod\limits_{i=1}^k(c_i+1)$.  Each of these coincides with a component of $T^w/Z(w)$ and hence is a product of simplices. If each $c_i$ is even then $C-1$ components in $T^w/Z(w)$ are identified in pairs while the remaining component is acted upon by $\cZ$. Thus in this case the total number of components is $\lceil \frac 12 C\rceil$.

The unpaired component in  $T^w/Z(w)$ is a product of simplices, with coordinates $x_{n_i,j}$ satisfying
\begin{itemize}
    \item $0\leq x_{n_i,j}\leq \frac 12$ for all $i,j$
    \item $x_{n_i,1}\leq \dots \leq x_{n_i,d_i}$ for each $i$.
\end{itemize}
The $-1$ element acts on these Lie-algebra coordinates as translation by $(\frac12,\dots,\frac 12)$. Subtracting this vector yields coordinates in $[-\frac 12,0]$ and acting with the centraliser to negate these gives values $\frac 12-x_{n_i,j}$ which lie in $[0,\frac 12]$ as required.  

This has the effect of reversing the second set of inequalities:
\[
\frac 12-x_{n_i,1}\geq \dots \geq \frac 12-x_{n_i,d_i}
\]
and we therefore reverse this list to return to the fundamental domain.  Hence the action of $-1$ on this set takes the coordinates $x_{n_i,1}, \dots, x_{n_i,d_i}$ to $\frac 12-x_{n_i,d_i},\dots,\frac 12-x_{n_i,1}$, giving an involution on the product of simplices.

Applying the results from the Section \ref{involutions} we see that this quotient is homeomorphic to $V(q+r-k,q)$ where $q=\sum_i \lceil d_i/2\rceil$ and $r$ is the number of $i$ such that $d_i$ is even, giving $q+r-k=\sum_i \lfloor d_i/2\rfloor = \delta-q$ where $\delta=\sum_i d_i$.

When there is at least one odd positive cycle, i.e.\ when $I^+_{\odd}$ is non-empty, this is the whole quotient since $A^w$ is empty.  It remains to identify the quotient of $A^w$ by the actions of $\cZ$ and $Z(w)$ in the case where $I^+_{\odd}$ is empty.

We begin by observing that as the coordinates of points in $A^w$ alternate over each (positive and negative) cycle, the effect of acting with the cycle, which lies in the centraliser, is to negate these coordinates.  We can do this for each $(n_i,j)$ and each $(\mbar_i,j)$ independently. The action of the element $w$ itself is to negate all coordinates, so the quotient of  $A^w$ by $\cZ$ and $Z(w)$ is given just by the quotient $A^w/Z(w)$.

Since the coordinates on negative cycles parametrise components, the action of the negative cycles is to identify all the components of $A^w$. The positive cycles act on the single remaining component, independently rotating each coordinate of the torus, to produce a new torus which has half the length in each direction. It remains to consider the actions of the factors $\{\pm 1\}\wr_{d_i} S_{d_i}$ of the centraliser. These act on the smaller torus exactly as they did on the original one by permutations and inversion of coordinates.

Hence the quotient of $A^w$ is a single polysimplex: geometrically this is the same as the other polysimplex components but scaled by a factor of $\frac 12$. Hence there are a total of $\lfloor \frac 12 C\rfloor+1$ components which are products of simplices. Again when all  $c_i$ are even there is an additional component as described above.

\bigskip

\begin{table}[h]
  \[
\begin{array}{c|c}
 C:=\prod_{i=1}^k (c_i+1) & \\
 q:=\sum_i \lceil d_i/2\rceil & \\
  \hline
 I^+_\odd\text{ empty} &  \left(\left\lfloor\frac 12C \right\rfloor+1\right)\times V(\delta,0)\,\sqcup\, (C\text{ mod } 2)\times V(\delta-q,q)\\
  &\\
I^+_\odd\text{ non-empty}   &  \left\lfloor\frac 12C \right\rfloor\times V(\delta,0)\,\sqcup\, (C\text{ mod } 2)\times V(\delta-q,q)\\
\end{array}
\]

    \bigskip
    
    \caption{\label{adjoint_C_n_sectors}The structure of a sector for the adjoint form of $C_n$}

\end{table}

\section{Centralisers, Fixed Sets and Quotients in the Weyl Group $D_{n}$}\label{D_n case}

\subsection{Centralisers in the Weyl group of type $D_n$}

Let $\t\cong \R^n$ be a Cartan subalgebra in the Lie algebra of type $D_n$. The simple roots of $D_n$ can be represented as
\[
r_1=(1,-1,0,\dots,0),\dots r_{n-1}=(0,\dots,0,1,-1), r_n=(-1,-1,0,\dots ,0)\in \t^*,
\]
where we identify $\t^*$ with $\t$ using the inner product.  As $D_n$ is simply-laced the coroots are identified with the roots.

Note that the roots of $D_n$ are exactly the long roots of $B_n$, and correspondingly the short roots of $C_n$. This identifies the Weyl group of $D_n$ as a subgroup of $W(B_n)=W(C_n)$.  The elements of $W(B_n)$ act as signed permutations on the coordinate basis, and therefore can be regarded as elements of the permutation group $S_{2n}$.  The group $W(D_n)$ is the index $2$ subgroup obtained as the intersection of $W(B_n)$ with the alternating group: by definition, all positive cycles and their compositions are elements of the alternating group since a positive $k$-cycle is the composition of $2$ $k$-cycles, while negative cycles are even length cycles of $S_{2n}$. The condition that an element lies in the alternating group thus translates as the requirement that it has an even number of negative cycle factors.

If two elements of $W(D_n)$ are conjugate then they must also be conjugate in $W(B_n)$; hence conjugate elements have the same cycle type (as described earlier, cf. Section \ref{signed permutations}). Since the centraliser of an element must either be the same as in $W(B_n)$ or have index $2$ it follows that each $W(B_n)$ conjugacy class consists of exactly one or two conjugacy classes of $W(D_n)$. The latter, which we will call the \emph{split case}, occurs exactly when the cycle type consists only of positive even length cycles (which in particular can occur only if $n$ is even), see \cite{carter}. To distinguish these classes we now select suitable conjugacy class representatives. For the first class we can choose any special positive element (in the sense of Section \ref{signed permutations}). For simplicity of notation we will choose an element, which we denote $\wplus$, each of whose positive cycles has the form $(p+1\; p+2\; \dots \;p+n_{i})({-(p+1)}\; {-(p+2)}\; \dots \;{-(p+n_{i})})$ for some $p$.

For the other conjugacy class we can take the conjugate of $\wplus$ by any element of $W(B_n)$ which is not in $W(D_n)$. In the case that $n/2$ is odd, there is a particularly nice choice, namely we take $\wminus=(\wplus)^{\sigma_{\odd}}$ where $\sigma_{\odd}=\prod_{p\text{ odd}}\sigma_p$. Our above convention that $\wplus$ is a product of contiguous cycles means that this $\wminus$ is the product of $\wplus$ with the central element $\prod_p\sigma_p$ of the Weyl group.

In the case that $n/2$ is even, the element $\sigma_\odd$ lies in $W(D_n)$ so we cannot use this to obtain a new conjugacy class.  In this case we shall simply take $\wminus=(\wplus)^{\sigma_1}$.

\subsection{Decomposing the centraliser}
\label{centraliser structure}
To analyse the sectors of type $D_n$ we will need to divide into various cases, corresponding to the decomposition of the centraliser that these induce.  There will be a total of four cases, depending on whether or not the cycle type has odd positive cycles (type I/II) and/or has negative cycles (type A/B). The cases IA and IIA (no negative cycles) are of particular interest.  Type IA occurs only when $n$ is even and here we see divergence between the geometry of the sectors for the $\SSpin^+$ and $\SSpin^-$ groups. Type IIA on the other hand is the case where we obtain the most interesting geometry from the sectors, including spheres, projective spaces and suspensions of projective spaces.

Given an element $w\in W(D_n)$, using the notation from Section \ref{B_n case} we split the indices $\{1,\dots,n\}$ into the sets
\begin{itemize}
    \item $I^+_{\ev}=\bigcup\limits_{n_i \text{ even}}\bigcup\limits_j I_{n_i,j}$
    
    \item $I^+_{\odd}=\bigcup\limits_{n_i \text{ odd}}\bigcup\limits_jI_{n_i,j}$

    \item $I^-=\bigcup\limits_{m_i}\bigcup\limits_j I_{\mbar_i,j}$.
\end{itemize}

For $X\subseteq \{1,\dots, n\}$, we let $W(D_X)$ denote the copy of $W(D_{|X|})$ in $W(D_n)$ corresponding to this subset, using the usual identification of these groups as signed permutations. As before we write $W(B_X)$ for the corresponding subgroup of $W(B_n)$.

The splitting of $\{1,\dots, n\}$ yields an inclusion
\[
W(D_{I^+_{\ev}})\times W(D_{I^+_{\odd}})\times W(D_{I^-}) \hookrightarrow W(D_n)
\]
with $w$ the image of an element on the left-hand side which we denote by $(w_{\ev},w_{\odd},w_\neg)$.

The image of $Z_{W(D_{I^+_{\ev}})}(w_{\ev})\times Z_{W(D_{I^+_{\odd}})}(w_{\odd})\times Z_{W(D_{I^-})}(w_\neg)$ lies in the centraliser of $w$, but in general as we will see this may be an index $2$ subgroup of the centraliser.
We can fix this by using the decomposition $I^+_{\ev}\cup J$ where $J=I^+_{\odd}\cup I^-$ and corresponding inclusion
\[
W(D_{I^+_{\ev}})\times W(D_{J}) \hookrightarrow W(D_n).
\]
Writing $(w_{\ev},w_{J})$ for the corresponding decomposition of $w$, the centraliser of $w$ splits as $Z_{W(D_{I^+_{\ev}})}(w_{\ev})\times Z_{W(D_{J})}(w_{J})$.  Indeed the factorisation works in the $B_n$ case and by construction $w_{\ev}$ lies in a split conjugacy class of $W(D_{I^+_{\ev}})$, whence the first centraliser factor is $Z_{W(B_{I^+_{\ev}})}(w_{\ev})$.

The first factor of the centraliser is a product of factors of the form
\[
(\langle \sigma_{I_{n_i,j}}\rangle \times (\Z/n_i\Z))\wr_{d_i}S_{d_i} \quad\text{for each even positive cycle}\\
\]

The second factor $Z_{W(D_{J})}(w_{J})$ is an index $2$ subgroup of $Z_{W(B_{J})}(w_{J})$ where the latter is the product of factors
\[
\begin{cases}
(\langle \sigma_{I_{n_i,j}}\rangle \times (\Z/n_i\Z))\wr_{d_i}S_{d_i} & \text{for each (odd) positive cycle}\\
(\Z/2m_i\Z)\wr_{c_i}S_{c_i}& \text{for each negative cycle.}
\end{cases}
\]
The index $2$ part required is given by elements having an even number of $\sigma_{I_{n_i,j}}$ and negative cycle terms.

We will divide the conjugacy classes into four types:
\begin{itemize}
    \item IA: All $n_i$ even, no negative cycles (the split case), i.e.\  $I^+_{\odd},I^-$ are empty. In this case the second factor of the above factorisation is trivial.
    
    \item IB: Negative cycles (at least two) with only even positives, i.e.\  $I^+_{\odd}$ is empty and $I^-$ non-empty. In this case the second factor of the centraliser is $Z_{W(D_{I^-})}(w_\neg)$.
    
    \item IIB: Negative cycles (at least two) with some odd positive, i.e.\  $I^+_{\odd},I^-$ are both non-empty.

    \item IIA: Some $n_i$ odd, no negative cycles, i.e.\ $I^+_{\odd}$ is non-empty and $I^-$ is empty. In this case the second factor of the centraliser is $Z_{W(D_{I^+_{\odd}})}(w_{\odd})$.

\end{itemize}

    In type IIB the second factor of the centraliser $Z_{W(D_{J})}(w_{J})$ contains the product
    \[
    Z_{W(D_{I^+_{\odd}})}(w_{\odd})\times Z_{W(D_{I^-})}(w_\neg)
    \]
    which has index $2$ in $Z_{W(D_{J})}(w_{J})$ since it has index $4$ in
    \[
    Z_{W(B_{I^+_{\odd}})}(w_{\odd})\times Z_{W(B_{I^-})}(w_\neg)\cong Z_{W(B_{J})}(w_{J}).
    \]

\subsection{Nodal lattices for groups of type $D_n$}

The nodal lattice for a maximal torus in the simply connected $D_n$ group (i.e.\ $\mathrm{Spin}(2n)$) is the even parity lattice in $\Z^n$.  Dual to this, the nodal lattice in the adjoint case (i.e.\ $\mathrm{PSO}(2n)$) is the integer lattice extended by the element $(\frac 12,\dots,\frac 12)$.

The group $D_n$ has connection index $4$ with the centre of the Lie group  $\mathrm{Spin}(2n)$ given by $\Z/4\Z$ when $n$ is odd, and by $\Z/2\Z\times \Z/2\Z$ when $n$ is even. In the former case the unique index $2$ subgroup of the centre gives $\mathrm{SO}(2n)$ as the quotient, with nodal lattice $\Z^n$. In the latter case there are three possible quotients. Explicitly the centre is given by $\{\pm 1, \pm \epsilon_1\epsilon_2\dots \epsilon_{2n}\}$ where $\epsilon_1,\dots \epsilon_{2n}$ are orthonormal basis vectors in $\R^{2n}$ viewed as elements of the Clifford algebra.  The quotient by $-1$ gives $\SO(2n)$, with corresponding nodal lattice $\Z^n$, which is self-dual.

The other two quotients have nodal lattice given by adding either $\langle (\frac 12, \dots ,\frac 12)\rangle$ or $\langle (-\frac 12, \dots ,\frac 12)\rangle$ to the even integer lattice. In the special case that $n=4$ these lattices are isometric to the standard integer lattice, and moreover the isometry preserves the set of roots/coroots. This however is unique to $n=4$. In general when $4|n$ these lattices are also self-dual, while for $n\equiv 2 \mod 4$ the lattices are dual to one another. Note however that these lattices are always isometric to one another, with isometry preserving the set of roots/coroots. The Lie group obtained as the quotient $\Spin(2n) / \langle \epsilon_1 \dots\epsilon_{2n}\rangle$ is called the semi-spin group $\SSpin^+(2n)$ while the isomorphic quotient by $\langle -\epsilon_1 \dots\epsilon_{2n}\rangle$ is denoted  $\SSpin^-(2n)$. We will denote by $E^\pm$ the elements $\pm\epsilon_1 \dots\epsilon_{2n}\in \Spin(2n)$.

Note that the lattice for the simply connected $D_n$ is isometric to that for the simply connected $B_n$ with the isomorphism identifying the roots of $D_n$ with the long roots of $B_n$, while the lattice for the adjoint $D_n$ is isometric to the lattice for the adjoint $C_n$, now identifying roots of $D_n$ with the short roots of $C_n$. This isometry is therefore equivariant with respect to the Weyl group $W(D_n)$ viewed as a subgroup of $W(B_n)=W(C_n)$.

The lattice for the $\SO(2n)$ case is isometric to the adjoint $B_n$ and the simply connected $C_n$ cases, and again this identification is $W(D_n)$-equivariant. When $n$ is even the semi-spin lattices have no counterpart in the $B_n$/$C_n$ cases. Indeed they are not preserved by the $W(B_n)=W(C_n)$ Weyl group, with elements of $W(B_n)$ which do not lie in $W(D_n)$ interchanging the two semi-spin lattices.

\subsection{The maximal tori and the action of the Weyl group}
The simply connected case is the Lie group $\Spin(2n)$.  For a subset $X$ of $\{1,2,\dots,n\}$ we denote by $\Spin_X(2n)$ the subgroup of $\Spin(2n)$ which is the image of $\Spin(2|X|)$ in $\Spin(2n)$ for the map induced by the inclusion of $\mathrm{span}\{\epsilon_{2j-1},\epsilon_{2j} : j \in X\}$ into $\R^{2n}$.

The maximal torus can be chosen to be the internal product in $\Spin(2n)$ of the subgroups $\Spin_{\{1\}}(2n),\dots,\Spin_{\{n\}}(2n)$. Note that each of these subgroups contains the $-1$ element of $\Spin(2n)$ in its image, so the maximal torus is a $\Z/2^{n-1}$ quotient of $\Spin(2)^n$. This torus is $W(D_n)$-equivariantly identified with the maximal torus $\coveringtorus$ in the simply connected $B_n$ case, indeed the group $\Spin(2n)$ is a subgroup of $\Spin(2n+1)$ and the maximal torus described above also gives a maximal torus of the latter.

For the Lie group $\SO(2n)$, the maximal torus is $\SO(2)^n$ which is
$W(D_n)$-equi\-variantly identified with the maximal torus $T$ in the adjoint $B_n$ case. Indeed the group $\SO(2n)$ is a subgroup of $\SO(2n+1)$ and the subgroup $\SO(2)^n$ is a maximal torus of the latter. We introduce the notation $\SO_X(2n)$ for the image of $\SO(2|X|)$ corresponding to the subspace $\mathrm{span}\{\epsilon_{2j-1},\epsilon_{2j} : j \in X\}$ of $\R^{2n}$.

We recall that $T$ is naturally identified with $\TT^n$, whose coordinates we denote $(\alpha_1,\dots, \alpha_n)$, while $\coveringtorus$ can be identified with the subtorus of $\TT^{n+1}$ of points $(\alpha_1,\dots, \alpha_n,\beta)$ such that $\beta^2\prod_i\alpha_i=1$.

\bigskip

In the semi-spin cases $\SSpin^\pm(2n)$ the maximal tori are given by the quotients $\coveringtorus / \langle E^\pm\rangle$ respectively.

The adjoint case is the Lie group $\PSO(2n)$ whose maximal torus is the quotient $\coveringtorus/\ZDSpin=T/\ZDSO$. We note that the identification of $T$ with the maximal torus in the simply connected $C_n$ group is equivariant with respect to the identifications of $\ZDSO$ and $\ZCSp$.

\bigskip

The above discussion shows that  the maximal tori in the classical cases are equivariantly isomorphic with counterparts in $B_n$ or $C_n$ groups. Hence for $w\in W(D_n)$ the $w$-fixed set in one of the classical $D_n$ maximal tori agrees with the $w$-fixed set in the corresponding $B_n$ or $C_n$ torus. The  discussion of the centralisers shows that, in the split case
the corresponding quotients of the fixed set by the centralisers are isomorphic.  On the other hand when the conjugacy classes do not split, the quotient of the fixed set in one of these cases by the $W(D_n)$ centraliser is  a $\Z/2\Z$ (orbifold) cover of the quotient by the $W(B_n)$ centraliser. 

\bigskip

To relate the semi-spin cases to these $B_n$ and $C_n$ cases we make the following observations.  As the groups $\SSpin^\pm(2n)$ are isomorphic as Lie groups we will consider in detail the case of $\SSpin^+(2n)$, whose maximal torus is $\coveringtorus / \langle E^+\rangle$. To examine the $w$-fixed set in this torus we consider its preimage in $\coveringtorus$.  This consists of points which are fixed by $w$ in $\coveringtorus$, along with those points $\tau$ for which $w\cdot \tau =E^+\tau$ which we denote by $\Atilde^w$. The fixed set $(\coveringtorus / \langle E^+\rangle)^w$ is the image in $\coveringtorus/\langle E^+\rangle$ of $\coveringtorus^w\cup \Atilde^w$.

To identify the element $E^+$ in the coordinates of $\coveringtorus$ we note that this element is the exponential of the vector $(\frac 12,\dots \frac 12,-\frac{n}2)$ in the Lie algebra.  Here we are using the coordinates of $\R^{n+1}_0$ for points in the Lie algebra. Hence in coordinates $E^+$ is identified as $(-1,\dots,-1,(-1)^{n/2})$: recall that the exponential map takes coordinates $(x_1,\dots,x_n,y)$ to $(e^{2\pi i x_1},\dots,e^{2\pi i x_n},e^{1\pi i y})$. Similarly in coordinates $E^-$ is identified as $(-1,\dots,-1,(-1)^{n/2+1})$.

Note that $\coveringtorus^w$ is exactly as in the simply connected $B_n$ case, just with $w$ restricted to the subgroup $W(D_n)$. In the case where there are any negative cycles, the corresponding coordinates parametrise components of $\coveringtorus^w$ and hence, since these coordinates of $E^+$ are $-1$, the element $E$ is not in the identity component. Thus $E^+$ acts on $\coveringtorus^w$ to identify components in pairs. On the other hand in the case with odd positive cycles and no negative cycles the set $\coveringtorus^w$ has only one component (see Section \ref{B_n fixed sets}) hence in these cases the quotient by $E^+$ reduces the volume of the component. In the split case $\coveringtorus^w$ has two components and we will see that for $w=\wplus$ the element $E^+$ lies in the identity component while for $w=\wminus$ the action of $E^+$ instead identifies the two components.

Let $\widetilde{A^w}$ denote the preimage in $\coveringtorus$ of the subset $A^w$ in $T$: recall from the adjoint $C_n$ case that $A^w$ is defined as the set of points in $T$ which are negated by $w$.  The set $\Atilde^w$ is a subset of $\widetilde{A^w}$, indeed a point of $\widetilde{A^w}$ lies in $\Atilde^w$ if and only if the $\beta$ coordinate is multiplied by $(-1)^{n/2}$ by the action of $w$.

We recall that when $w$ has any odd positive cycles $A^w$ is empty, hence $\Atilde^w$ is also empty. If all $n_i$ are even then $A^w$ consists of points whose coordinates within each factor corresponding to $I_{n_i,j}$ or $I_{\mbar_i,j}$ are alternating $(\alpha,-\alpha,\dots)$, and with the additional constraint that within the $I_{\mbar_i,j}$ factors the values are either $\{\pm1\}$ or $\{\pm \sqrt{-1}\}$ depending on the parity of $m_i$. With the exception of the conjugacy class of $\wminus$ in the split case, the positive cycles may be taken to be special positive elements in which case these cycles do not change $\beta$.  On the other hand the negative cycles multiply $\beta$ by $\pm1$ or by $\pm\sqrt{-1}$.  Since $n$ is even and by assumption there are no odd positive cycles, there must be an even number of odd negative cycles, and thus the effect of $w$ on $\beta$ is multiplication by $\pm1$. Hence as long as the set $I^-$ is non-empty the values of $\pm1$ or $\pm\sqrt{-1}$ can be chosen such that the required action on $\beta$ of multiplication by $(-1)^{n/2}$ is achieved. Therefore $\Atilde^w$ is exactly the preimage of half of the components of $A^w$. In taking the preimage the number of components is doubled (since changing the sign of $\beta$ gives a different component as all positive cycles are even), and hence $\Atilde^w$ has exactly the same number of components as $A^w$ in the non-split case.

Note that by construction the action of $w$ on $\Atilde^w$ is the same as the action of $E$, hence as $w$ lies in its own centraliser, the quotient of $\Atilde^w/\langle E^+\rangle /Z_{W(D_n)}(w)$ is the same as $\Atilde^w/Z_{W(D_n)}(w)$

\subsection{Type IA: The split cases}\label{section: split} In the split case the element $w$ has cycle type consisting only of even positive cycles. Note that the value of $n$ must therefore be even for split cases to exist.

Since $\wplus,\wminus$ are conjugate in $W(B_n)$ but not in $W(D_n)$ it follows that except in the semi-spin case the fixed sets and quotients occurring for $\wplus$ and $\wminus$ will agree.  Indeed, since here the centralisers are the same as in $W(B_n)$ and $W(C_n)$, we have already computed the quotients of the fixed sets by the centralisers in the classical cases: for the simply connected case $\Spin(2n)$ this is given by the corresponding quotient in the simply connected $B_n$ case; for $\SO(2n)$ it is given by the adjoint $B_n$ or the isomorphic dual simply connected $C_n$ case; for the adjoint case $\PSO(2n)$ we use the adjoint $C_n$ calculation.  Note that the total contribution to the extended quotient from these cycle types in each $D_n$ case is twice that from the corresponding $B_n/C_n$ case since these cycle types correspond to \emph{pairs} of conjugacy classes in $W(D_n)$.

\bigskip

Within the split case it remains to consider the semi-spin groups, which must exist in this case since $n$ is even.

The fixed set $\coveringtorus^{\wplus}$ consists of points where the values of $\alpha_p$ are constant over the coordinates $p\in I_{n_i,j}$ for each $i,j$ and satisfying
\[
\left(\beta\prod_{i,j} \alpha_{n_i,j}^{n_i/2}\right)^2=1.
\]
This has two components indexed by the choice of square root. Recall that in coordinates we have
\[
E^+=(-1,\dots,-1,(-1)^{n/2})
\]
and that $n=\sum_{i,j} n_i$. The identity component $\coveringtorus^{\wplus}_1$ is given by the equation $\beta\prod_{i,j} \alpha_{n_i,j}^{n_i/2}=1$ which is satisfied by the coordinates of $E$. Hence $E$ preserves each of the two components and, since it acts freely, the quotient $\coveringtorus^{\wplus}/\langle E^+\rangle$ consists of two copies of the torus $\coveringtorus^{\wplus}_1/\langle E^+\rangle$.  Indeed since the value of $\beta$ is a function of the $\alpha$ coordinates in $\coveringtorus^{\wplus}_1$ these components can be identified with $T^{\wplus}/\{\pm1\}$ which is the fixed set from the adjoint $C_n$ case.  This identification is equivariant with respect to the group $Z_{W(D_n)}(\wplus)=Z_{W(C_n)}(\wplus)$, so the quotient is as in this $C_n$ case.

\bigskip

For the double primed case we begin by considering the case when $n/2$ is even, and $\wminus=(\wplus)^{\sigma_1}$. The coordinates for the first cycle of the fixed set take the form $(\alpha_{n_1,1},\alpha_{n_1,1}^{-1},\dots,\alpha_{n_1,1}^{-1})$, while the other cycles satisfy the same conditions as before. The equation defining the $\beta$ for points in $\coveringtorus^{\wminus}$ becomes 
\[
\left(\beta(\alpha_{n_1,1})^{1-n_1}\prod_{i,j} \alpha_{n_i,j}^{n_i/2}\right)^2=1.
\]
Now the identity component $\coveringtorus^{\wminus}_1$ has equation $\beta(\alpha_{n_1,1})^{1-n_1}\prod_{i,j} \alpha_{n_i,j}^{n_i/2}=1$ which is false for the coordinates of $E$ since $(-1)^{1-n_1}=-1$, so $E$ acts to identify the two components and the quotient $\coveringtorus^{\wminus}/\langle E^+\rangle$ is the single torus $\coveringtorus^{\wminus}_1$, which may be equivariantly identified with  $T^{\wminus}$. From Section \ref{C_n case}, the quotient of $T^{\wminus}$ by the $W(C_n)$ centraliser (which is also the $W(D_n)$ centraliser) is a polysimplex.

When $n/2$ is odd $\wminus=(\wplus)^{\sigma_\odd}=\wplus\prod_p\sigma_p$ and so the coordinates for \emph{each} cycle of the fixed set are given alternately by $\alpha_{n_i,j}$ and $\alpha_{n_i,j}^{-1}$. Note that since $\prod_p\sigma_p$ is central, the centraliser of $\wminus$ is the same as that for $\wplus$.  Since the product of all $\alpha$ coordinates is $1$, the defining equation for $\beta$ becomes $\beta^2=1$.  Since $n/2$ is odd, all coordinates of $E$ are $-1$ hence the two components are identified in $\coveringtorus^{\wminus}/\langle E^+\rangle$ and this quotient can therefore be $Z_{W(D_n)}(\wminus)$-equivariantly identified with $T^{\wminus}$.  Hence the quotient is again a polysimplex.

\bigskip

The set $\Atilde^{\wplus}$ will equal $\widetilde{A^{\wplus}}$ when $n/2$ is even, and will be empty when $n/2$ is odd, since $\wplus$ acts trivially on the $\beta$ coordinate. On the other hand for $n/2$ even we have $\wminus=(\wplus)^{\sigma_1}$ and $\widetilde{A^{\wminus}}$ consists of points having the form $(\alpha_{n_1,1},-\alpha_{n_1,1}^{-1},\alpha_{n_1,1}^{-1},-\alpha_{n_1,1}^{-1},\dots)$ on the first cycle. The element $\wminus$ will multiply $\beta$ by the first and last coordinates of this cycle, that is by $\alpha_{n_1,1}(-\alpha_{n_1,1}^{-1})=-1$. Since the $\beta$-coordinate of $E$ is $1$ in this case $\Atilde^{\wminus}$ is empty.  When $n/2$ is odd we have $\wminus=(\wplus)^{\sigma_\odd}=\wplus\prod_p\sigma_p$. Points of the fixed set now alternate $\alpha_{n_i,j},-\alpha_{n_i,j}^{-1}$ and thus the product of coordinates gives $(-1)^{n/2}=-1$, hence $\Atilde^{\wminus}=\widetilde{A^{\wminus}}$ when $n/2$ is odd.  We remark that whether or not $n/2$ is even, the element $\sigma_\odd$ takes $\widetilde{A^{\wplus}}$ to $\widetilde{A^{(\wplus)^{\sigma_\odd}}}$.

We now consider the action of the centraliser on the set $\widetilde{A^{\wplus}}$, noting that when $n/2$ is even this gives the quotient of $\Atilde^{\wplus}$ as required. For simplicity of notation we will assume that each positive cycle of $\wplus$ has the form $(p+1\; p+2\; \dots \;p+n_{i})(-(p+1)\; -(p+2)\; \dots \;-(p+n_{i}))$ for some $p$. The  values of $p$ will all be even as each cycle is of even length.

Let $u$ be the product $(1\;2)(-1\;-2)\dots(n-1\;n)(-(n-1)\;-n)$. We note that since points of $\Atilde^{\wplus}$ must have alternating coordinates within cycles, that in particular  the $\alpha$ coordinates are all negated by $u$, that is $\widetilde{A^{\wplus}}\subseteq\widetilde{A^{u}}$. Indeed $\widetilde{A^{\wplus}}=\widetilde{A^{u}} \cap \coveringtorus^{(\wplus)^2}$.

We define a map $s$ from this `universal' $\widetilde{A^u}$ to the standard maximal torus of the group $\Spin(n)$ (which is a group of type $D_{n/2}$) which we denote by $\TSpinn$. We view the coordinates in $\TSpinn$ as being indexed by pairs $(1,2),(3,4)$ etc.\ and we define the map by
\[
s:(-\alpha_2,\alpha_2,-\alpha_4,\alpha_4,\dots,\beta) \mapsto (-\alpha_2^2,-\alpha_4^2,\dots,\beta).
\]
We think of the element $(\wplus)^2$ as acting on the set of pairs $(1,2),(3,4),\dots$ and hence giving a permutation action on the coordinates of $\TSpinn$, and as usual acting trivially on the $\beta$ coordinate.

The centraliser of $\wplus$ is generated by elements of three types: we have cycles from $\wplus$; the involution $\prod_{p\in I_{n_i,j}} \sigma_p$ for each $i,j$; and the permutation groups $S_{d_i}$ for each $i$.  We write the centraliser as $(N_1\times N_2)\rtimes Q$ where $N_1$ is generated by the cycles, $N_2$ by the involutions and $Q$ is the product of the permutation groups.

We now consider the centraliser of $(\wplus)^2$ viewed as above as an element of $W(B_{n/2})$. This again has the above form and we denote the corresponding subgroups by $\overline{N}_1,\overline{N}_2$ and $\overline{Q}$. The generators of $\overline{N}_1$ are cycles of length $n_i/2$.  The generators of $\overline{N}_2$ are products of $n_i/2$ negative $1$-cycles: these are not necessarily elements of $W(D_{n/2})$ depending on the parity of $n_i/2$, however they are in the larger group $W(B_{n/2})$.  These are canonically in bijection with the products of $n_i$ negative cycles generating $N_2$ and hence the groups $N_2$ and $\overline{N}_2$ are identified with one another.  Likewise $\overline{Q}$ is identified with $Q$.  We remark that the elements of $\overline{N}_1$ all act trivially on $\TSpinn^{(\wplus)^2}$\!\!.

\begin{lemma}\label{lemma1}
The map $s$ induces a homeomorphism $\widetilde{A^{\wplus}}/N_1 \cong \TSpinn^{(\wplus)^2}$\!\!. 
Moreover this identification is equivariant with respect to the actions of $N_2\rtimes Q$. Hence there is an identification 
\[
\widetilde{A^{\wplus}}/Z_{W(D_n)}(\wplus) \cong \TSpinn
    ^{(\wplus)^2}/Z_{W(B_{n/2})}((\wplus)^2).
    \]
\end{lemma}

\begin{proof}
The generators of $N_1$ act to negate the coordinates $\alpha_2,\alpha_4,\dots$, hence they leave the values $\alpha_p^2$ unchanged so the map $s$ induces a map from $\widetilde{A^{\wplus}}/N_1$ to $\TSpinn^{(\wplus)^2}$.  Conversely given two points with the same image in $\TSpinn^{(\wplus)^2}$, the values of their $\alpha$ coordinates must agree up to sign, and their $\beta$ coordinates must agree, hence they differ by the action of an element of $N_1$.  This makes the map injective, and it is easy to see that it is also surjective.

The action of a generator of $N_2$ is to invert the corresponding $-\alpha_p$ and $\alpha_p$ coordinates, and to multiply $\beta$ by their product.  The action of the generator viewed as an element of $\overline{N}_2$ is to invert the corresponding values of $-\alpha_p^2$ and to multiply $\beta$ by the product of these coordinates.  Hence $s$ is $N_2$ equivariant.

Similarly the permutations in $Q$ act to permute values of $\alpha_p$ and in the same way on the values of $-\alpha_p$, while $\overline{Q}$ permutes values of $-\alpha_p^2$.  Again this is equivariant.

The result now follows as
\begin{align*}
\widetilde{A^{\wplus}}/Z_{W(D_n)}(\wplus) &= \widetilde{A^{\wplus}}/(N_1\times N_2)\rtimes Q\\
&=(\widetilde{A^{\wplus}}/N_1) /N_2\rtimes Q\\
&\cong \TSpinn
    ^{(\wplus)^2} / \overline{N}_2\rtimes \overline{Q}\\
    &= \TSpinn
    ^{(\wplus)^2} /Z_{W(B_{n/2})}((\wplus)^2).
\end{align*}
\end{proof}

When $n/2$ is odd we make a similar argument for $\widetilde{A^{\wminus}}$.  Let $u''=u^{\sigma_\odd}$. As before we note that $\widetilde{A^{\wminus}}=\widetilde{A^{u''}} \cap \coveringtorus^{(\wminus)^2}$.

We define a map $p$ from $\Atilde^{u''}$ to $\TSpinn$ by $p(\tau)=s(\sigma_\odd\tau)$, noting that $\sigma_\odd$ takes $\widetilde{A^{u''}}$ to $\widetilde{A^{u}}$.

The following result is now a consequence of the previous lemma since $\wminus=(\wplus)^{\sigma_\odd}$ has the same centraliser as $\wplus$.

\begin{lemma}
    The map $p$ induces a homeomorphism $\widetilde{A^{\wminus}}/N_1 \cong \TSpinn
    ^{(\wminus)^2}$.  Moreover this identification is equivariant with respect to the actions of $N_2\rtimes Q$.
    Hence there is an identification
    \[
    \widetilde{A^{\wminus}}/Z_{W(D_n)}(\wminus)\cong \TSpinn
    ^{(\wminus)^2}/Z_{W(B_{n/2})}((\wminus)^2).
    \]
\end{lemma}

For the groups $\SSpin^-$ the roles of $\wplus,\wminus$ are switched.

Combining the calculations above with our computations in Sections \ref{involutions}, \ref{B_n sectors} and \ref{C_n case} we see that for cycle type $(n_1^{d_1}, \ldots, n_n^{d_k})$, in the semi-spin case the quotients are:

\bigskip

\begin{table}[h]
\[
\begin{array}{c|c|c|c|c}
 \multirow{2}{5em}{$\SSpin^{\pm}(2n)$}   & n/2& \multicolumn{2}{|c|}{w=\wplusminus} & w=\wminusplus\\
 &\text{mod } 2&4| \text{gcd}\{n_i\}&\text{gcd}\{n_i\}=2 \text{ mod } 4&\\
 \hline
{}^{\strut} (\coveringtorus^w/\langle E^\pm\rangle)/Z_W(w)  & 0,1 & 2 \times V(\delta-q,q)  & 2 \times V(\delta-q,q)  &V(\delta,0)   \\ [.5ex]
\hline
\multirow{2}{6em}{$\Atilde^w/Z_W(w)$}    & {}^{\strut}0&2 \times V(\delta,0) & V(\delta,0) &\emptyset \\ [.5ex]
\cline{2-5}
  &{}^{\strut}1& \text{n/a} & \emptyset  & V(\delta,0)
\end{array}
\]

\bigskip

\caption{The structure of the $\SSpin^\pm(2n)$ sector associated with a split cycle type.}
\end{table}
where $\delta=\sum_i d_i$ and $q=\sum\lceil d_i/2\rceil$.

\medskip

For comparison, in the three classical cases there is no dichotomy for the split conjugacy classes concerning the parity of $n/2$ and we obtain the following quotients:

\bigskip

\begin{table}[h]
\[
\begin{array}{c|c|c|c}
G & \Spin(2n)   &
 \SO(2n)   & 
\PSO(2n)    \\
 \hline
{}^{\strut} T(G)^w/Z_W(w)& 2 \times V(\delta,0) &  V(\delta,0)  &V(\delta-q,q)\\
\hline
{}^{\strut} A^w/Z_W(w) &\text{n/a} &\text{n/a}
& V(\delta,0)
\end{array}
\]

\bigskip

   \caption{The structure of the $D_n$ sectors for split cycle types in the classical cases}
   
\end{table}

\bigskip
\bigskip

\subsection{Type IB: Negative cycles (at least two) with only even positives}\label{D_n case 2}

As noted above the centraliser of $w$ in this case factorises as
    \[
    Z_{W(D_{I^+_{\ev}})}(w_{\ev})\times Z_{W(D_{I^-})}(w_\neg).
    \]
    
Beginning with the $\SO(2n)$ case, we have a corresponding equivariant factorisation of the fixed set as $T^w=T^{w_\ev}_{\ev}\times T^{w_\neg}_-$, and since $w_{\ev}$ is in a split conjugacy class, the first factor of the quotient is a product of simplices as above.

The second factor $T^{w_\neg}_-$ of the fixed set is a product $\prod_{i=1}^k\{\pm 1\}^{c_i}$. The group $Z_{W(D_{I^-})}(w_\neg)$ is generated by the even products of negative cycles appearing in $w_\neg$, which act trivially, along with permutations of factors corresponding to cycles of the same length. As a result the quotient of the second factor has $C:=\prod_{i=1}^k (c_i+1)$ points, where $c_i$ denotes the multiplicity of the negative cycles of length $\mbar_i$.

\bigskip

We now consider $\Spin(2n)$. In this case the fixed set $\coveringtorus^w$ can be identified with $T_\ev^{w_{\ev}}\times \coveringtorus^{w_\neg}_-$, where $T_\ev$ denotes the standard maximal torus in $\SO_{I^+_{\ev}}(2n)$ while $\coveringtorus_-$ denotes the standard maximal torus in $\Spin_{I^-}(2n)$. The identification is given by the map
\begin{align*}
(\alpha_1,\dots,\alpha_n,\beta) &\mapsto \big((\alpha)_{\ev},((\alpha)_-,\gamma)\big), \text{ where}\\
\gamma&=\beta\prod_{i=1}^l\prod_{j=1}^{d_i}\alpha_{n_i,j}^{n_i/2}.
\end{align*}
Note that the condition that $(\alpha_1,\dots, \alpha_n,\beta)$ lies in $\coveringtorus$ translates directly to show that $((\alpha)_-,\gamma)$ is in $\coveringtorus_-$. We will show that this map is equivariant.  It is easy to see that it is equivariant for the second factor $Z_{W(D_{I^-})}(w_\neg)$. The first factor $Z_{W(D_{I^+_{\ev}})}(w_{\ev})$ is a product of factors
\[
(\langle \sigma_{I_{n_i,j}}\rangle \times (\Z/n_i\Z))\wr_{d_i}S_{d_i} \quad\text{for each even positive cycle}\\
\]
It is again easy to see that the map is equivariant for elements of $\Z/n_i\Z$ and $S_{d_i}$. The element $\sigma_{I_{n_i,j}}$ acts to invert the coordinate $\alpha_{n_i,j}$ and to multiply $\beta$ by $\alpha_{n_i,j}^{n_i}$. Hence these elements also fix $\gamma$ and thus the map is equivariant as claimed.

The first factor of the fixed set, and the action on this, are exactly as in the $\SO(2n)$ case, yielding a product of simplices. For the second factor we note that the product $\prod_{i=1}^k\prod_{j=1}^{c_i}\alpha_{\mbar_i,j}$ must be $1$ so that the element $w_\neg$ preserves the value of $\gamma$. As in the simply connected $B_n$ case, the number of components arising from the choices of $\alpha_{\mbar_i,j}$ is $\left\lceil \frac 12 \prod_{i=1}^k (c_i+1)\right\rceil$. We additionally have the choice of the value of $\gamma$, however unless all $\alpha_{\mbar_i,j}$ are equal we can find a pair of negative cycles whose product acts to negate $\gamma$. Thus we have exactly two additional components arising from the choices of $\gamma$ in the cases where $\alpha_{\mbar_i,j}=1$ for all $i,j$ and where $\alpha_{\mbar_i,j}=-1$ for all $i,j$, giving a total of
\[
\left\lceil \frac 12 \prod_{i=1}^k (c_i+1)\right\rceil +2
\]
points in the quotient of the second factor.  Note that since $w_\neg$ lies in $W(D_n)$ the total number of cycles in $w_\neg$, and hence the total number of $\alpha_{\mbar_i,j}$ coordinates is even, hence the point with all coordinates $\alpha_{\mbar_i,j}=-1$ gives product of $1$ as required.

\bigskip

Turning now to the adjoint case ($\PSO(2n)$) the torus is $T/\{\pm I\}$ and the $w$ fixed set is the quotient $(T^w \sqcup A^w)/\{\pm I\}$. As before we write $T^w=T_{\ev}^{w_\ev}\times T_-^{w_\neg}$. The action of $\{\pm I\}$ commutes with the action of the centraliser.

We first consider the action of $\{\pm I \}$ on the discrete factor. We note that the action of $-I$ is to negate $\sum_{j=1}^{c_i}\alpha_{\mbar_i,j}$ for each $i$. If each multiplicity $c_i$ is even then there is a unique fixed point for the action, where all of these sums are zero. Otherwise the action is free on the discrete factor. Hence the action on $T^w/Z_{W(D_n)}(w)$ identifies components in pairs, with the exception of possibly one component, where $\{\pm I \}$ acts on the component, giving quotient $V(\delta-q,q)$ where $\delta=\sum_i d_i$ and $q=\sum_i\lceil d_i/2\rceil$ as in the adjoint $C_n$ group.

We recall that $A^w$ consists of points of the torus where on positive (even) cycles we have alternating values $\pm\alpha$ with $\alpha\in \TT$, while for negative cycles we have alternating $\pm\alpha$ where $\alpha$ is $\pm 1$ if $\mbar_i$ is even, and $\pm \sqrt{-1}$ if $\mbar_i$ is odd. As for $T^w$ we can factorise $A^w$ as $A_{\ev}^{w_\ev}\times A_-^{w_\neg}$. Noting that $Z_{W(D_{I_{\ev}^+})}(w_\ev)=Z_{W(C_{I_{\ev}^+})}(w_\ev)$ (since $w_\ev$ is in a split conjugacy class), the quotient of $A_{\ev}^{w_\ev}$ by the centraliser is, as in the adjoint $C_n$ case, a polysimplex.

For the action on $A_-^{w_\neg}$ we note that each negative cycle acts to negate the corresponding coordinate.  Since elements of $W(D_{I_-})$ must have an even number of negative cycles, we can negate any even number of coordinates, yielding exactly two components which remain distinct after the action of the permutation part of the centraliser.

Recalling that $\{\pm I\}$ acts trivially on $A^w/Z_{W(D_n)}(w)$, we conclude that this part of the quotient is $2\times V(\delta,0)$.

\bigskip

We conclude with the semi-spin case. We will focus on the case of $\SSpin^+(2n)$, for which the torus is $\coveringtorus/\langle E^+\rangle$: the $w$-fixed set in $\coveringtorus/\langle E^+\rangle$ can be identified with the $w^{\sigma_1}$ fixed set in $\coveringtorus/\langle E^-\rangle$ and likewise the quotients can also be identified. Here, unlike the split case, the elements $w$ and $w^{\sigma_1}$ lie in the same $W(D_n)$ conjugacy class, hence we do not need a separate calculation for $\SSpin^+(2n)$ and $\SSpin^-(2n)$.

The $w$ fixed set is the quotient $(\coveringtorus^w \sqcup \Atilde^w)/\langle E^+\rangle$. As in the $\Spin(2n)$ case we write $\coveringtorus^w=T_{\ev}^{w_\ev}\times \coveringtorus_-^{w_\neg}$ and write $(E^\ev,E^\neg)$ for the corresponding factorisation of $E^+$. Note that $E^\ev$ is $-1$ on all coordinates of $T_\ev$, while $E^\neg$ is $-1$ on all except (possibly) the $\gamma$ coordinate where its value is $(-1)^{n/2}$.

Since the action of $\langle E^+\rangle$ commutes with the action of the centraliser, the quotient of this part of the fixed set in $\coveringtorus/\langle E^+\rangle$ by the centraliser can be computed by taking the quotient from the $\Spin(2n)$ case, and dividing this by $\langle E^+\rangle$. Recall that $T_{\ev}/Z_{W(D_{I^+_{\ev}})}(w_{\ev})$ is identified as a polysimplex and $\coveringtorus_-^{w_\neg}/ Z_{W(D_{I^-})}(w_\neg)$ consists of $2+\left\lceil\frac 12\prod_{i=1}^k (c_i+1) \right\rceil$ points. These points are parametrised by values of $\alpha_{\mbar_i,j}=\pm 1$ up to permutation of the $j$ index and with the constraint that the product is $1$,  along with the values of $\gamma=\pm1$ when the $\alpha_{\mbar_i,j}$ are all equal.

We first consider the action of $E^\neg$ on the discrete factor. The action is to negate $\sum_{j=1}^{c_i}\alpha_{\mbar_i,j}$ for each $i$ and, when the $\alpha_{\mbar_i,j}$ are all equal, to multiply $\gamma$ by $(-1)^{n/2}$. If each multiplicity $c_i$ is even and additionally the sum $\sum_{i=1}^kc_i$ is a multiple of $4$, then, and only then, there is a point parametrised by $\alpha_{\mbar_i,j}=1$ for exactly $c_i/2$ coordinates and $-1$ for $c_i/2$ coordinates for each $i$. This point is fixed by $E^\neg$, while $E^\neg$ acts freely on all other points. We note that if each $c_i$ is even but the sum is not divisible by $4$ then $\prod_{i=1}^k (c_i+1)$ is congruent to $3$ mod $4$ and hence $2+\left\lceil\frac 12\prod_{i=1}^k (c_i+1) \right\rceil$ is even, with all points being identified in pairs. 
Likewise if any $c_i$ is odd (and hence there are at least two odd values of $c_i$) then $2+\left\lceil\frac 12\prod_{i=1}^k (c_i+1) \right\rceil$ is even and points are identified in pairs.

When there is a fixed point, we must take the quotient of the $T_{\ev}/Z_{W(D_{I^+_{\ev}})}(w_{\ev})$ factor for this single component by $\langle E^\ev\rangle$, obtaining a copy of $V(\delta-q,q)$.

\bigskip

We now consider the action of the centraliser on the set $\Atilde^{w}$. We cannot factorise the action of the centraliser $Z_{W(D_{I^+_{\ev}})}(w_\ev) \times Z_{W(D_{I^-})}(w_\neg)$ on $\Atilde^{w}$ in the same way that we factorised the action on $\coveringtorus^w=T_{\ev}^{w_\ev}\times \coveringtorus_-^{w_\neg}$. Positive cycles of $w_\ev$ of length $2$ mod $4$ are elements of the first factor of the centraliser which act non-trivially on the second factor of the torus: such a cycle negates the corresponding $\alpha_{n_i,j}$ coordinate while fixing all other coordinates, and hence negates the $\gamma$ coordinate, since this has a factor of $\alpha_{n_i,j}^{n_i/2}$.

Instead we work with the original $\alpha_{n_i,j},\alpha_{\mbar_i,j}$ and $\beta$ coordinates, and we begin by considering the action of the second factor $Z_{W(D_{I^-})}(w_\neg)$. The action of each (negative) cycle of $w_\neg$ is to negate the corresponding coordinate $\alpha_{\mbar_i,j}$ and to multiply $\beta$ by $\alpha_{\mbar_i,j}$.  If $m_i$ is even then $\alpha_{\mbar_i,j}$ is $\pm 1$ so applying the square of the negative cycle, which lies in  $Z_{W(D_{I^-})}(w_\neg)$, leaves $\alpha_{\mbar_i,j}$ unchanged and multiplies $\beta$ by $-\alpha_{\mbar_i,j}^2=-1$. We deduce that if some $m_i$ is even then the quotient of $\Atilde^w$ by $Z_{W(D_n)}(w)$ is the quotient of the image of $\Atilde^w$ in $A^w$ by $Z_{W(D_n)}(w)$. Since the image of $\Atilde^w$ is half of the components of $A^w$, the action of $Z_{W(D_{I^-})}(w_\neg)$ now reduces this image to give a single copy of the torus $A_\ev^{w_\ev}$, on which $Z_{W(D_{I^+_\ev})}(w_\ev)$ acts to give a polysimplex as quotient as in the $\PSO$ case.

We now consider the case where all $m_i$ are odd, and hence all $\alpha_{\mbar_i,j}$ are $\pm \sqrt{-1}$. Recall that $Z_{W(D_{I^-})}(w_\neg)$ is generated by products of pairs of negative cycles along with permutation groups $S_{c_i}$ where $c_i$ is the multiplicity of $m_i$. The permutation groups are generated by transpositions, and we note that a transposition exchanging $\alpha_{\mbar_i,j}$ and $\alpha_{\mbar_i,j'}$ either acts trivially if these values are equal or negates both values if $\alpha_{\mbar_i,j}=-\alpha_{\mbar_i,j'}$.  In the latter case the action of the transposition is the same as the action of the product of the two corresponding negative cycles, since this element negates $\alpha_{\mbar_i,j}$ and $\alpha_{\mbar_i,j'}$ and multiplies $\beta$ by their product which is $1$. Hence we have shown that the quotient $\Atilde^w/Z_{W(D_{I^-})}(w_\neg)$ is equal to the quotient of $\Atilde^w$ by the index $2$ subgroup of $\prod_i (\Z/m_i\Z)^{c_i}$. Using the pairs of negative cycles we can change the sign of all but one $\alpha_{\mbar_i,j}$ coordinate, so that these are all say $+\sqrt{-1}$, with the final coordinate determined by these: specifically the product of all the $\alpha_{\mbar_i,j}$ coordinates must be $(-1)^{n/2}$ and hence we require the final coordinate to take the value $(\sqrt{-1})^{1+n-\sum_ic_i}$. This leaves two components determined by the $\beta$ coordinate.

Noting that the factor $(\sqrt{-1})^{n-\sum_i c_i}$ is unchanged when raised to an odd power, we see that the product $\prod_{i,j}(\alpha_{\mbar_i,j})^{m_i}$ is given by
\[
(\sqrt{-1})^{\sum_i m_ic_i}(\sqrt{-1})^{n-\sum_i c_i}=(-1)^{t/2}
\]
where $t$ (which is even) is the total length of positive cycles minus the number of negative cycles.

Define $\beta'$ to be $(\sqrt{-1})^{t/2}\beta$ so that
\[
\beta^2\prod_{i,j}(\alpha_{n_i,j})^{n_i}\prod_{i,j}(\alpha_{\mbar_i,j})^{m_i}=(\beta')^2\prod_{i,j}(\alpha_{n_i,j})^{n_i}.
\]
The coordinates $\alpha_{n_i,j}$ and $\beta'$ define a point in $\widetilde{A_\ev^{w_\ev}}$, and hence we have an identification of $\Atilde^w/Z_{W(D_{I^-})}(w_\neg)$ with $\widetilde{A_\ev^{w_\ev}}$. It is easy to see that the actions of the elements of $Z_{W(D_{I^+_\ev})}(w_\ev)$ agree on the two spaces.  Hence $\Atilde^w/Z_{W(D_n)}(w)$ is isometric to $\widetilde{A_\ev^{w_\ev}}/Z_{W(D_{I^+_\ev})}(w_\ev)$, which by Lemma \ref{lemma1} is isometric to $\TSpinn^{w^2}/Z_{W(B_{n/2})}(w^2)$ where $w^2$ is viewed as a signed permutation of the set of pairs $\{\pm(1,2),\dots,\pm(n-1,n)\}$.  From Section \ref{B_n case} this is either $2\times V(\delta,0)$ or $V(\delta,0)$ when all $n_i$ are multiples of $4$ or some $n_i$ is $2$ mod $4$ respectively.

The results are summarised in the tables below.

\bigskip

\begin{table}[h]
\[
\begin{array}{c|c|c|c}
G & \Spin(2n)   &
 \SO(2n)   & 
\PSO(2n)    \\
 \hline
{}^{\strut} T(G)^w/Z_W(w)& \left( 2 +\left\lceil \frac 12 C\right\rceil\right) \times V(\delta,0) &  C\times V(\delta,0)  &\left\lfloor \frac 12 C\right\rfloor \times V(\delta,0) \;\sqcup\qquad\qquad\\
&&&\qquad(C\text{ mod } 2)\times V(\delta-q,q)\\[.5ex]
\hline
{}^{\strut} A^w/Z_W(w) &\text{n/a} &\text{n/a}
& 2\times V(\delta,0)
\end{array}
\]

\bigskip

    \caption{The structure of the sector in the classical cases for a conjugacy class with negative cycles but no odd positive cycles}

\end{table}

\bigskip

\begin{table}[h]
\[
\begin{array}{c|c|c|c|c}
 \multirow{2}{5em}{$\SSpin^{\pm}(2n)$}  &C & \prod_i m_i& & \\
 &\text{mod } 4&\text{mod } 2&4| \text{gcd}\{n_i\}&\text{gcd}\{n_i\}=2 \text{ mod } 4\\
 \hline
\multirow{2}{9em}{$(\coveringtorus^w/\langle E^\pm\rangle)/Z_W(w)$}  &{}^{\strut}0,3& \multirow{2}{1.5em}{$0,1$} & \multicolumn{2}{|c}{\left(1+\left\lceil\frac 14C \right\rceil\right)\times V(\delta,0)}    \\ [.5ex]
  &1&  & \multicolumn{2}{|c}{\left(1+\left\lfloor\frac 14C \right\rfloor\right)\times V(\delta,0)\,\sqcup\, V(\delta-q,q)  }   \\ [.5ex]
\hline
\multirow{2}{6em}{$\Atilde^w/Z_W(w)$}    &\multirow{2}{2.25em}{$0,1,3$}& {}^{\strut}0&V(\delta,0) & V(\delta,0) \\ [.5ex]
  &&{}^{\strut}1& 2 \times V(\delta,0) & V(\delta,0)  
\end{array}
\]

\bigskip

\caption{The structure of the sector in the semi-spin cases for a conjugacy class with negative cycles but no odd positive cycles}

\end{table}

\bigskip
where $C:=\prod_{i=1}^k (c_i+1)$, $\delta=\sum_id_i$ and $q=\sum_i\lceil d_i/2\rceil$.

\bigskip

\subsection{Type IIB: Negative cycles (at least two) and odd positive cycles}
We begin with the case of $\SO(2n)$. As usual the torus $T$ will be identified with the torus in the adjoint $B_n$ case and we may thus consider the action of elements of $W(B_n)$ as well as $W(D_n)$.

For an element $w$ in $W(D_n)$, a point of the torus $T$ will be fixed by $w$ if and only if it is fixed by each individual cycle of $w$ (regarded as an element of the larger group $W(B_n)$). In particular $w$ has at least one negative cycle, which we denote by $\eta$; this lies in $Z_{W(B_n)}(w)$ but not in $Z_{W(D_n)}(w)$.  Hence $Z_{W(B_n)}(w)=Z_{W(D_n)}(w) \cup \eta Z_{W(D_n)}(w)$ and since $\eta$ acts trivially the quotient of the fixed set by the action of $Z_{W(D_n)}(w)$ agrees with the quotient from the adjoint $B_n$ case.

For the adjoint case ($\PSO(2n)$), as in the adjoint $C_n$ case we can write
\[(T/\ZDSO)^w=T^w/\ZDSO \cup A^w/\ZDSO
\]
where $A^w$ denotes the set of points in $T$ which are negated by $w$.  However since we have odd positive cycles the $A^w$ part is empty. The quotient of $T^w/\ZDSO$ by the centraliser of $w$ is identified with the quotient of $T^w/Z_{W(D_n)}(w)$ by $\ZDSO$. As discussed in the $\SO(2n)$ case above, $T^w/Z_{W(D_n)}(w)$ is identified with the corresponding quotient in the adjoint $B_n$ or equivalently the simply connected $C_n$ case. It follows that the quotient of this by $\ZDSO$ is identified with the corresponding quotient in the adjoint $C_n$ case.

We now consider the simply connected ($\Spin(2n)$) case. Since we have odd positive cycles, the components are parametrised by the allowed values of $\alpha_{\mbar_i,j}=\pm 1$: these are required to have product equal to $1$. We consider the subgroup $N$ of $Z_{W(D_n)}(w)$ of elements which act trivially on $T^w$, and note that
\[
{\coveringtorus^w}/{Z_{W(D_n)}(w)} \cong \frac{\coveringtorus^w/N}{Z_{W(D_n)}(w)/N}.
\]
On points of $\coveringtorus^w$ each element of $N$ can either act trivially or negate the value of $\beta$, and by continuity this behaviour must be constant on each component. Hence $N$ must act trivially on the two components of $\coveringtorus^w$ meeting the centre $\ZDSpin$ of $\Spin(2n)$.

The components which are disjoint from the centre are those where the values of $\alpha_{\mbar_i,j}$ are not all equal. For each such component we can find an element of $N$ which is a product of two negative cycles of $w$ where one corresponds to a value $\alpha_{\mbar_i,j}=+1$ and the other to $-1$. The action of this element of $N$ is to multiply $\beta$ by the product $(+1)(-1)$, i.e.\ to negate this. It follows that the components of $\coveringtorus^w/N$ are exactly the components of $\coveringtorus^w$ meeting $\ZDSpin$ along with the components of $T^w$ disjoint from the centre $\ZDSO$ of $\SO(2n)$. Looking at the components of the latter type, it remains to act by $Z_{W(D_n)}(w)/N$, but since $N$ acts trivially on $T^w$ we obtain exactly the corresponding parts of the quotient $T^w/Z_{W(D_n)}(w)$ from the $\SO(2n)$ case. Note that while there are $C-2$ components of $T^w$ which are disjoint from $\ZDSO$ only those for which an even number of $\alpha_{\mbar_i,j}$ are equal to $-1$ appear in the image of $\coveringtorus^w/N$. This yields $\lceil\frac 12 C\rceil-2$ components of this form: note that the two components of $T^w$ meeting $\ZDSO$ are those where $\alpha_{\mbar_i,j}=+1$ for all $i,j$ (the identity component) and where $\alpha_{\mbar_i,j}=-1$ for all $i,j$ and both of these are in the image of $\coveringtorus^w/N$.

It remains to consider these two components. We note that the second component is the product of the identity component with an element of the centre of $\Spin(2n)$. It follows that these components are $Z_{W(D_n)}(w)$-equivariantly equivalent.  As in the $\SO(2n)$ case, letting $\eta$ be a negative cycle of $w$, we note that $\eta$ acts trivially on the identity component and so as $Z_{W(B_n)}(w)=Z_{W(D_n)}(w) \cup \eta Z_{W(D_n)}(w)$ these quotients are the same as the quotient of the identity component in the simply connected $B_n$ case.

For the semi-spin case $\SSpin^+(2n)$ we have $(\coveringtorus/\langle E^+\rangle)^w=\coveringtorus^w/\langle E^+\rangle \cup \Atilde^w/\langle E^+\rangle$, but similarly to the adjoint case the set $\Atilde^w$ is empty. We are therefore left with the quotient of $\coveringtorus^w/Z_{W(D_n)}(w)$ by $\langle E^+ \rangle$.   Recall that the $\lceil \frac 12C\rceil$ components of $\coveringtorus^w/Z_{W(D_n)}(w)$ arise from the choice of signs of $\alpha_{\mbar_i,j}$ and the requirement that these have product $+1$.

When at least one $c_i$ is odd then there must be two odd values since $\sum_i c_i$ is even for $w$ in $W(D_n)$. In this case $C$ is a multiple of $4$ and the action of $\langle E^+\rangle$ identifies the $\frac 12 C$ components in pairs yielding $\frac 14C$ components in the quotient. On the other hand when all $c_i$ values are even then the number of components in $\coveringtorus^w/Z_{W(D_n)}(w)$ may be even or odd depending on whether $C$ mod $4$ is $3$ or $1$ respectively. In the first of these cases the action of $\langle E^+\rangle$ again identifies components in pairs. In the second case there is exactly one component which is preserved by $\langle E^+\rangle$, namely the component where for each $i$ the number of $-1$ values for $\alpha_{\mbar_i,j}$ is exactly $c_i/2$.  Note that $C\equiv 1+\sum_i c_i$ mod $4$ when all $c_i$ are even hence $\sum_i c_i/2$ is even exactly when $C$ mod $4$ is $1$. In all cases the total number of components is $\left\lceil \frac 12\lceil\frac 12C\rceil\right\rceil=\lceil\frac 14C\rceil$.

The component of $\coveringtorus^w/Z_{W(D_n)}(w)$ which is preserved by $E^+$ is disjoint from the centre of $\Spin(2n)$ and hence we saw that this is identified with a component of $T^w/Z_{W(D_n)}(w)$. Hence the quotient of this component is the same as for the adjoint $D_n$ case which in turn was the same as the adjoint $C_n$ case, yielding $V(\delta-q,q)$ where $q=\sum_i\lceil d_i/2\rceil$.

\bigskip

\begin{table}[h]
\[
\begin{array}{c|c|c}
 C:=\prod_{i=1}^k (c_i+1) & C\text{ mod } 4 & \\
  &&\\
\hline
\Spin(2n)   & 0,1,3 & \lceil\frac 12 C\rceil \times V(\delta,0)  \\
&&\\
 \SO(2n)   & 0,1,3 & C \times V(\delta,0)\\
 &&\\
\SSpin^\pm(2n)   & 0,3 &  \lceil\frac 14 C\rceil \times V(\delta,0)\\
   & 1 & \lfloor\frac 14 C\rfloor \times V(\delta,0) \,\sqcup \,V(\delta-q,q)\\
  &&\\
\PSO(2n)   & 0 &  \left(\frac 12 C \times V(\delta,0)\right)\\
   & 1,3 &\left(\lfloor \frac 12 C\rfloor \times V(\delta,0) \right) \sqcup \,V(\delta-q,q)\\
\end{array}
\]

\bigskip

\caption{The structure of the $D_n$ sector associated to a conjugacy class which has negative cycles and odd positive cycles}

\end{table}

Note that the total number of components in the $\PSO(2n)$ cases is $\lceil \frac 12C\rceil$ agreeing with the number from the $\Spin(2n)$ case.

\bigskip

\subsection{Type IIA: No negative cycles \& at least one odd positive cycle}

\subsubsection{The case of $\SO(2n)$.}\label{no negs at least one odd pos}

The $w$-fixed set in the torus is then identified with the product of the $w_{\ev}$  and $w_{\odd}$ fixed sets in $T_{\ev}$, $T_{\odd}$ respectively, since a point is fixed by $w$ if and only if it is fixed by each cycle of $w$. The quotient of the $w$ fixed set by $Z_{W(D_n)}(w)$ is therefore the product of the quotients of the $w_{\ev}$ and $w_{\odd}$ fixed sets by their respective centralisers.  As in previous cases, the first factor is the product of $d_i$-simplices where $d_i$ ranges over the multiplicities of the even length cycles of $w$.

For the second factor the centraliser is an index $2$ subgroup of the $B_{I_{\odd}}$ centraliser which is given by the product
\[
\left(\prod_{n_i\text{ odd}}(\langle \sigma_{I_{n_i,j}}\rangle \times (\Z/n_i\Z))^{d_i}\right)\rtimes\prod_{n_i\text{ odd}} S_{d_i}.
\]
Specifically, since each $\sigma_{I_{n_i,j}}$ element is the product of an odd number ($n_i$) of negative $1$-cycles, the index $2$ subgroup is precisely the part with an even number of these  factors.

The $\Z/n_i\Z$ factors act trivially, leaving an action by the copy of $C_2^{{\delta_\odd}-1}$ generated by products of pairs of elements $\sigma_{I_{n_i,j}}$ and the permutation action of $\prod\limits_{n_i\text{ odd}} S_{d_i}$, where ${\delta_\odd}=\sum\limits_{n_i \text{ odd}} d_i$.

The $C_2^{{\delta_\odd}-1}$ factor acts on the fixed torus by inverting pairs of coordinates.  Thinking of the torus as given by the cube $[-\frac 12,\frac 12]^{\delta_\odd}$ with opposing faces identified, the actions of these elements are rotations through $\pi$ in coordinate planes.  A fundamental domain for this is given by $[-\frac 12,\frac 12] \times [0,\frac 12]^{{\delta_\odd}-1}$, with the following identifications.  The ``square'' faces $\{\pm \frac 12\}\times [0,\frac 12]^{{\delta_\odd}-1}$ are identified (as they were in the torus).  The ``rectangular'' faces given by setting $x_p=0$ or $x_p=\frac 12$ for some $p>1$ are preserved (up to the identifications in the torus) by the action negating the $1$ and $p$ coordinates.  Hence these faces are folded in the first coordinate.

This quotient can equivalently be described as the union of two copies of the cube $[0,\frac 12]^{{\delta_\odd}}$ glued on their boundary faces, and thus is the suspension of this boundary.  Hence this is a ${\delta_\odd}$-sphere.  The quotient of the fixed set by the centraliser is therefore
\[
S^{\delta_\odd}/\prod_{n_i \text{ odd}} S_{d_i}.
\]

Note that the action of the groups $S_{d_i}$  fixes the suspension points (the midpoints of the two cubes).  The identification $S^{\delta_\odd}=\Sigma(S^{{\delta_\odd}-1})$, where $S^{{\delta_\odd}-1}$ is the boundary of the cube, is thus equivariant where the permutations are taken to fix the suspension points.

\begin{lemma}\label{equator lemma}
    The boundary of $[0,\frac12]^{\delta_\odd}$ is $S_{\delta_\odd}$-equivariantly identified with $\Sigma W$ where $W$ is the union of the 
    ${\delta_\odd}-2$-dimensional subcubes of $[0,\frac12]^{\delta_\odd}$ containing neither $(0,\dots,0)$ nor $(\frac  12,\dots\frac 12)$.

    Moreover the `equator' $W$ can be $S_{\delta_\odd}$-equivariantly identified with the boundary $\partial\Delta^{{\delta_\odd}-1}$ of a ${\delta_\odd}-1$-simplex, specifically the link of $(0,\dots,0)$ in the boundary of $[0,\frac12]^{\delta_\odd}$.
\end{lemma}

\begin{proof}
    The first statement of the lemma is an immediate consequence of the fact that $S_{\delta_\odd}$ fixes the points $(0,\dots,0)$ and $(\frac  12,\dots\frac 12)$.

    The second statement follows by stereographic projection from $(0,\dots,0)$ of $W$ onto the link.
\end{proof}

\begin{lemma}\label{join lemma}
    Let $k_\odd>0$ denote the number of odd values $n_i$. The quotient of $\partial\Delta^{{\delta_\odd}-1}$ by the subgroup $\prod_{n_i \text{ odd}} S_{d_i}$ of $S_{\delta_\odd}$ is the join of a simplex $\Delta^{{\delta_\odd}-1-k_\odd}$  with the boundary $\partial\Delta^{k_\odd-1}$ of a simplex. This is a ${\delta_\odd}-2$ ball unless ${\delta_\odd}=k_\odd$ in which case we just obtain the ${\delta_\odd}-2$-sphere $\partial\Delta^{k_\odd-1}$ as the quotient.
\end{lemma}

\begin{proof}
    We begin by considering the quotient of $\Delta^{{\delta_\odd}-1}$. We will use barycentric coordinates $(x_1,\dots,x_{\delta_\odd})$ with $x_p\geq0$ and $\sum_px_p=1$.
    
    The permutation subgroup $S_{d_i}$ acts on $d_i$-tuples of coordinates, which we take to be $x_{p_i+1},\dots,x_{p_i+d_i}$ for some $p_i$. A fundamental domain for the action is then given by points whose coordinates satisfy $x_{p_i+1}\leq\dots\leq x_{p_i+d_i}$ for each $i$.

    For each $i$ and for $j=1,\dots,d_i$ let $v_{i,j}$ be the point with coordinates
    \[
    x_{p_i+j},\dots,x_{p_i+d_i}=\frac{1}{d_i+1-j}\]
    and all other coordinates equal to $0$. For  fixed  $i$ these points, which lie in the fundamental domain, span a $d_i-1$ simplex, and the full fundamental domain is the join of these simplices since a general point can be written as an affine combination of points in them.

    The quotient of $\partial \Delta^{{\delta_\odd}-1}$ is the set of points in the above fundamental domain such that at least one coordinate is zero. A ${\delta_\odd}-2$-dimensional face of the fundamental domain is thus in the desired quotient if and only if its vertices share a common zero coordinate.
    
    Since the $x_{p_i+1}$ coordinate of $v_{i,j}$ is zero if and only if $j>1$, and the $x_{p_i+1}$ coordinate of $v_{i',j}$ is zero for all $i'\neq i$ it follows that the set of all vertices except $v_{i,1}$ spans a face in $\partial \Delta^{{\delta_\odd}-1}$. The intersection of these over $i$ gives a common ${\delta_\odd}-1-k_\odd$-simplex. On the other hand the vertices $\{v_{i,1} : n_i \text{ odd}\}$ span a $k_\odd-1$-simplex in the quotient of $\Delta^{{\delta_\odd}-1}$, whose boundary is in the quotient of $\partial\Delta^{{\delta_\odd}-1}$.

    A general point of the quotient of $\partial\Delta^{{\delta_\odd}-1}$ is a convex combination of a point in the ${\delta_\odd}-1-k_\odd$-simplex with a point in the boundary of the $k_\odd-1$-simplex.
\end{proof}

Applying the above two lemmas the quotient $S^{\delta_\odd}/\prod_{n_i \text{ odd}} S_{d_i}$ is the double suspension of either a ${\delta_\odd}-2$-sphere, in the case that $k_\odd={\delta_\odd}$ i.e.\  all $d_i$ are $1$, or a ${\delta_\odd}-2$-ball.

Hence the full quotient of the $w$-fixed set by $Z_{W(D_n)}(w)$ is either a ball $V(\delta,0)$ or a product $V(\delta_\ev,0)\times S^{\delta_\odd}$ depending on whether $d_i=1$ for all odd $n_i$, where $\delta_\ev=\sum\limits_{n_i\text{ even}} d_i$.

\subsubsection{The adjoint case, $\PSO(2n)$}

Since we have odd positive cycles there is no $A^w$ part and thus the quotient of the fixed set in $T/\ZDSO$ by the action of
the centraliser is exactly the quotient by $\ZDSO$ of the quotient from the above $\SO(2n)$ case.

The $\ZDSO$ action translates by $\pm\frac12$ in each coordinate.  When ${\delta_\odd}$ is even, or equivalently when $n$ is even,
the signed permutation which negates all coordinates corresponding to the odd length cycles lies in $W(D_n)$, and hence the composition of this with the translation acts on the cube reversing each $[0,\frac 12]$ factor.  Note that this preserves the midpoint of this cube, and similarly it preserves the other suspension point $(\frac 34,\frac 14,\dots,\frac14)$. The action on the ${\delta_\odd}$-sphere is thus the suspension of the antipodal action on the ${\delta_\odd}-1$-sphere.

When ${\delta_\odd}$ is odd (or equivalently $n$ is odd) we think of the $\ZDSO$ action as translating by $+\frac12$ in the first coordinate and $-\frac 12$ in all the others. Since there are now an even number of coordinates translated by $-\frac12$ we can, as before, negate these using the action of the centraliser.  Thus the action translates the first coordinate of the cuboid and reverses the others. The action on the sphere therefore exchanges the suspension points in this case, while still acting antipodally on the boundary of the $[0,\frac 12]^{{\delta_\odd}}$-cube. Thus in this case the action on the entire ${\delta_\odd}$-sphere is antipodal.

In the case where $I^+_\ev$ is empty and $d_i=1$ for all $i$ this gives the complete answer: the quotient of the fixed set is either $\RP^{\delta_\odd}$ or $\Sigma \RP^{{\delta_\odd}-1}$ for ${\delta_\odd}$ respectively odd or even.

Turning to the general case where $d_i=1$ for all odd $n_i$ we examine the quotient of the product $V(\delta_\ev,0)\times S^{\delta_\odd}$ where $V(\delta_\ev,0)$ is the factor arising from the positive even cycles.

We first consider the case when $n$ (and hence ${\delta_\odd}$) is odd. Since the action of $\ZDSO$ on $S^{\delta_\odd}$ is fixed point free in this case, the quotient is a $V(\delta_\ev,0)$-bundle over $\RP^{\delta_\odd}$. To identify this explicitly we focus on the subset $V(q_\ev,0)\times S^{\delta_\odd}$ where $q_\ev=\sum\limits_{n_i \text{ even}} \lceil d_i/2\rceil$, noting that the action negates all directions of $V(q_\ev,0)$. We consider the natural identification of this as a subset of the sphere $\partial (V(q_\ev,0)\times CS^{\delta_\odd})$. The actions on $V(q_\ev,0), S^{\delta_\odd}$ yield the antipodal action on this sphere. The quotient of $V(q_\ev,0)\times S^{\delta_\odd}$ by the involution is thus identified with the unit ball bundle in the normal bundle of $\RP^{\delta_\odd}$ in $\RP^{q_\ev+\delta_\odd}$. The normal bundle is the sum of $q_\ev$ copies of the tautological line bundle over $\RP^{\delta_\odd}$. There is additionally a factor of dimension $\delta_\ev-q_\ev$ with trivial action.
We denote the vector bundle over $\RP^{\delta_\odd}$ which is the sum of $p$ copies of the trivial line bundle and $q$ copies of the tautological line bundle by $\bundle{\delta_\odd}pq$, and denote the unit ball bundle of this by $\unitbundle{\delta_\odd}pq$. The quotient in this case is therefore $\unitbundle{\delta_\odd}{\delta_\ev-q_\ev}{q_\ev}$.

On the other hand when $n$ (and hence ${\delta_\odd}$) is even we consider  $V(\delta_\ev,0)\times S^{\delta_\odd}$ as the union (over $V(\delta_\ev,0)\times S^{{\delta_\odd}-1}$) of two copies of $V(\delta_\ev,0)\times C S^{{\delta_\odd}-1}$ each of which is preserved by the $\ZDSO$ action. Applying the results of Section \ref{involutions} we can identify $V(\delta_\ev,0)\times C S^{{\delta_\odd}-1}$ with a cuboid on which $\ZDSO$ acts to fix $\delta_\ev-q_\ev$ coordinates and invert $q_\ev+{\delta_\odd}$ coordinates where $q_\ev=\sum\limits_{n_i \text{ even}} \lceil d_i/2\rceil$. This results in a quotient $V(\delta_\ev-q_\ev, q_\ev+\delta_\odd)$. The two copies of this are then glued together over the quotient of $V(\delta_\ev,0)\times S^{{\delta_\odd}-1}$, which as above is $\unitbundle{\delta_\odd-1}{\delta_\ev-q_\ev}{q_\ev}$.

Note that in both the even and odd cases the space $V(\delta_\ev,0)\times S^{{\delta_\odd}}$ is $\ZDSO$-equivariantly contractible onto a copy of $S^{\delta_\odd}$ and therefore at the level of homotopy these quotients are again either $\RP^{\delta_\odd}$ or $\Sigma\RP^{{\delta_\odd}-1}$.

\bigskip

We now turn to the case when at least one of the $d_i$, for $n_i$ odd, is greater than $1$. We take the fundamental domain $\Delta$ in $[0,\frac 12]^{\delta_\odd}$ for the permutation action such that for each odd $n_i$, the corresponding $d_i$ coordinates are in ascending order. We consider the point $x_0$ in this fundamental domain, which is also in the boundary of the cube $[0,\frac 12]^{\delta_\odd}$ whose coordinates are as follows.  For each odd $n_i$ with $d_i=1$, the corresponding coordinate in the cube is $\frac 14$. For each odd $n_i$ with $d_i>1$ the corresponding $d_i$ coordinates are $(0,\frac 14,\dots,\frac 14,\frac 12)$.

Let $\tau$ denote the map on $[0, \frac 12]^{\delta_\odd}$ reversing each coordinate.  Note that in the case where ${\delta_\odd}$ is even this is the action of the generator of $\ZDSO$, while in the case that ${\delta_\odd}$ is odd the generator takes the interior of the cube off itself, however it still acts in the same way as $\tau$ on the boundary. Let $\mu$ denote the product of transpositions which within each set of $d_i$ coordinates reverses the order of said coordinates. Note that the number of transpositions required to do this is $\delta_\odd-q_\odd$ where $q_\odd=\sum\limits_{n_i \text{ odd}} \lceil d_i/2\rceil$. By construction $\mu\circ\tau$ preserves the fundamental domain for the permutations, and moreover this fixes $x_0$.  Additionally, since $\tau$ agrees with the action of $\ZDSO$ on $\partial [0, \frac 12]^{\delta_\odd}$, the action of $\ZDSO$ on $\partial [0, \frac 12]^{\delta_\odd}$ descends to the action of $\mu\circ\tau$ on the quotient of $\partial [0, \frac 12]^{\delta_\odd}$ by the permutations.

Let $H$ denote the hyperplane bisector of $x_0,\tau x_0$. For $x\neq \tau x_0$ define $\rho(x)$ to be the stereographic projection of $x$ onto the hyperplane $H$ along the line from $x$ to  $\tau x_0$. Since $\Delta$ is convex (and does not contain $\tau x_0$) its image $\rho(\Delta)$ is also convex.  Moreover, each point in  $\rho(\Delta)$ is the image of a unique point in the intersection of $\Delta$ with the boundary of the cube thus establishing a homeomorphism between said intersection and the convex set $\rho(\Delta)$.

The sphere $S^{\delta_\odd}$ is the suspension of the boundary of the cube, and hence its quotient by the permutation group is identified with the suspension of $\rho(\Delta)$. Note that since $\mu\circ \tau$ fixes $\tau x_0$, the map $\rho$ commutes with the action of $\mu\circ\tau$ and hence the action of $\ZDSO$ on the quotient of $\partial [0, \frac 12]^{\delta_\odd}$ is identified with the action of $\mu\circ \tau$ on $\rho(\Delta)$. Hence the map from the quotient of $S^{\delta_\odd}$ by permutations, to the space  $\Sigma(\rho(\Delta))$, is equivariant for the action of $\ZDSO$ on $S^{\delta_\odd}$ and the action of $\mu\circ \tau$ on 
$\rho(\Delta)$, suspended either by exchanging or fixing suspension points depending on whether ${\delta_\odd}$ (or equivalently $n$) is odd or even.

The quotient of the fixed set in the torus $T/\ZDSO$ by the action of the centraliser is thus identified with the quotient of $V(\delta_\ev,0)\times\Sigma(\rho(\Delta))$ by the product of the corresponding involutions.  Recall that on the first factor (identified as usual with a cube) the involution preserves $\delta_\ev-q_\ev$ and reverses $q_\ev$ coordinates. The map $\mu\circ\tau$ reverses $q_\odd-1$ coordinates of $\rho(\Delta)$, so on the second factor the action reverses either $q_\odd$ or $q_\odd-1$ coordinates for ${\delta_\odd}$ odd or even respectively since this determines whether or not the suspension direction is also reversed.

Thus in the case that ${\delta_\odd}$ (and also $n$) is odd the resulting quotient is therefore  $V(\delta-q,q)$, while when ${\delta_\odd}$ is even the resulting quotient is $V(\delta-q+1,q-1)$. We remark that in the case when $n$ is odd this answer is homeomorphic to the corresponding quotient in the $C_n$ case however this is  not true for even $n$.

\subsubsection{The simply connected case, $\Spin(2n)$} 

In this case the fixed set $\coveringtorus^w$ can be identified with $T_\ev^{w_{\ev}}\times \coveringtorus^{w_\odd}_\odd$, where $T_\ev$ denotes the standard maximal torus in $\SO_{I^+_{\ev}}(2n)$ while $\coveringtorus_\odd$ denotes the standard maximal torus in $\Spin_{I^+_\odd}(2n)$. The identification is given by the map
\begin{align*}
(\alpha_1,\dots,\alpha_n,\beta) &\mapsto \big((\alpha)_{\ev},((\alpha)_\odd,\gamma)\big), \text{ where}\\
\gamma&=\beta\prod_{n_i\text{ even}}\prod_{j=1}^{d_i}\alpha_{n_i,j}^{n_i/2}.
\end{align*}
This map is equivariant: the argument is exactly as in Section \ref{D_n case 2} for the factorisation in terms of the even (positive) and negative parts.

The first factor of the fixed set, and the action on this, are exactly as in the $\SO(2n)$ case, yielding quotient $V(\delta_\ev,0)$. For the second factor we first note that there is a single component since traversing a loop in one of the $\alpha_{n_i,j}$ coordinates, for odd $n_i$, negates the value of $\gamma$.

As there is a single component we must simply take the quotient of the fixed set in the Lie algebra $\t_\odd^{w_\odd}$ by its intersection with the nodal lattice, and by the action of the centraliser $Z_{W(I^+_\odd)}(w_\odd)$. Identifying $\t_\odd^{w_\odd}$ with $\R^{\delta_\odd}$, where ${\delta_\odd}$ is the sum of $d_i$ for $n_i$ odd, the lattice is identified with the even parity lattice $\Z^{\delta_\odd}$. We denote this by $\Z^{\delta_\odd}_\ev$.

Note that the cycles of $w_{\odd}$ all act trivially on the fixed set: the part of the centraliser which acts nontrivially is the index $2$ subgroup of $\langle \sigma_{I_{n_i,j}} : n_i \text{ odd} \rangle$ along with permutation groups $S_{d_i}$. This is exactly the intersection of the product of Weyl groups $W(B_{d_i})$ (for $n_i$ odd) with $W(D_{\delta_\odd})$, and we denote the intersection of $\langle \sigma_{I_{n_i,j}} : n_i \text{ odd} \rangle$ with $W(D_{\delta_\odd})$ by $\langle \sigma_{I_{n_i,j}} : n_i \text{ odd} \rangle_{D_{{\delta_\odd}}}$. We define
\[H:=\Z^{\delta_\odd}_\ev \rtimes \langle \sigma_{I_{n_i,j}} : n_i \text{ odd} \rangle_{D_{{\delta_\odd}}}\rtimes \prod_{n_i\text{ odd}} S_{d_i}\]
and
\[G:=\Z^{\delta_\odd} \rtimes \langle \sigma_{I_{n_i,j}} : n_i \text{ odd} \rangle\rtimes \prod_{n_i\text{ odd}} S_{d_i}\]
which contains $H$ as an index $4$ subgroup. By Lemma \ref{lemma0} our intended quotient $Z_{W(D_{I^+_\odd})}(w_\odd)\backslash \coveringtorus_\odd^{w_\odd}=H\backslash \R^{\delta_\odd}$ is identified with $G\backslash (G/H\times \R^{\delta_\odd})$.

\smallskip

Let $N$ denote the normal subgroup $N=\Z^{\delta_\odd} \rtimes \langle \sigma_{I_{n_i,j}} : n_i \text{ odd} \rangle$ of $G$. We note that this is the product of infinite dihedral groups $\D_\infty^{\delta_\odd}$ acting on $\R^{\delta_\odd}$ in the standard way. It thus has strict fundamental domain $[0,\frac 12]^{\delta_\odd}$ in $\R^{\delta_\odd}$.

If $x_{n_i,j}=0$ for some $i,j$ then the point $x$ is stabilised by the corresponding reflection in $N$, and likewise if $x_{n_i,j}=\frac 12$ for some $i,j$. Moreover for a given point $x$, the reflections corresponding to coordinates which are $0$ or $\frac 12$ generate the stabiliser. We consider the actions of these reflections on $G/H$. The latter is a Klein $4$ group, which we may think of as the dihedral group $\D_2$ obtained as the abelianisation of $\D_\infty$. Each $\D_\infty$ factor acts by left translation on this abelianisation.

The reflections in hyperplanes $x_{n_i,j}=0$ all map to one generator of $\D_2$, which we denote $s_0$, and the reflections in $x_{n_i,j}=\frac 12$ all map to the other generator which we denote $s_{\!\frac12}$. We have $4$ copies of $[0,\frac 12]^{\delta_\odd}$ indexed by $\D_2=\{e,s_0,s_{\!\frac12},s_0s_{\!\frac 12}\}$.
Letting $U_0$ denote the union of codimension $1$ cubes containing $(0,\dots,0)$ and $U_{\!\frac 12}$ the union of cubes containing $(\frac 12,\dots,\frac 12)$, the copy of $U_0$ indexed by $s\in \D_2$ is identified with the $s_0s$ copy, and likewise the copy of $U_{\!\frac12}$ indexed by $s$ is identified with the $s_{\!\frac 12}s$ copy.

With these identifications, there are now two copies of $U_{\!\frac12}$, which are themselves identified over the set $W:=U_0\cap U_{\!\frac12}$ which is the union of ${\delta_\odd}-2$-subcubes containing neither  $(0,\dots,0)$ nor $(\frac12,\dots,\frac12)$. By Lemma \ref{equator lemma} this intersection is a ${\delta_\odd}-2$-sphere. Each of the copies of $U_{\!\frac12}$ can be identified with the cone on $W$ and hence the union of the two copies is $\Sigma (W)$. Moreover the cube $[0,\frac 12]^{\delta_\odd}$ is obtained from $U_{\!\frac12}$ by coning this to the point $(0,\dots,0)$, and hence the union of the four cubes, glued along $U_0$ and $U_{\!\frac12}$ is exactly the ${\delta_\odd}$-sphere $\Sigma\Sigma(W)$.

It remains to act with the group $G/N=\prod\limits_{n_i\text{ odd}} S_{d_i}$. Note that the above identification of the quotient by $N$ with the ${\delta_\odd}$-sphere $\Sigma\Sigma(W)$ gives a $G/N$-equivariant identification with the ${\delta_\odd}$-sphere appearing in the $\SO(2n)$ case.  Applying Lemmas \ref{equator lemma} and \ref{join lemma} this quotient is a ${\delta_\odd}$-sphere if all $d_i$ are $1$ and a ${\delta_\odd}$-ball otherwise.

The full quotient is thus either $V(\delta,0)$ or $V(\delta_\ev,0)\times S^{\delta_\odd}$.

\subsubsection{The semi-spin case, $\SSpin^\pm(2n)$}
For the semi-spin case $\SSpin^\pm(2n)$ we have $(\coveringtorus/\langle E^\pm\rangle)^w=\coveringtorus^w/\langle E^\pm\rangle \cup \Atilde^w/\langle E^\pm\rangle$, but similarly to the adjoint case the set $\Atilde^w$ is empty. We are therefore left with the quotient of $\coveringtorus^w/Z_{W(D_n)}(w)$ by $\langle E^\pm \rangle$. Since these quotients are isomorphic, for simplicity we consider the case of $E^+$, which on the Lie algebra is translation by $(\frac 12,\dots,\frac 12)$.

We recall from the preceding subsection that $\coveringtorus^w/Z_{W(D_n)}(w)$ is the product of the polysimplex $V(\delta_\ev,0)$ with the quotient of $\Sigma\Sigma(W)$ by the action of $\prod\limits_{n_i\text{ odd}} S_{d_i}$. Here $W$ is the union of the ${\delta_\odd}-2$-dimensional subcubes of $[0,\frac 12]^{{\delta_\odd}}$ containing neither of $(0,\dots,0),(\frac12,\dots,\frac12)$. Moreover as noted above the space $\Sigma\Sigma(W)$ is equivariantly identified with the ${\delta_\odd}$-sphere appearing in the $\SO(2n)$ case.

We now consider the action of $\langle E^+\rangle$ on $V(\delta_\ev,0)\times \Sigma\Sigma(W)$. As usual the action on the first (polysimplex) factor is the involution exchanging pairs of vertices with at most one fixed vertex in each simplex factor. The vertices of $W$ have coordinates which are $0$ or $\frac 12$ and translating gives coordinates of $\frac 12$ and $1$.  However since at least one coordinate is now $\frac 12$ we can apply rotations to change the $1$ coordinates to $0$.  We now consider the four suspension points.

The first suspension is built by suspending $W$ to the points $(e,(\frac 12,\dots,\frac 12))=(s_{\!\frac 12},(\frac 12,\dots,\frac 12))$ and $(s_0,(\frac 12,\dots,\frac 12))=(s_0s_{\!\frac 12},(\frac 12,\dots,\frac 12))$, and the second suspension is to the points $(e,(0,\dots,0))=(s_0,(0,\dots,0))$ and $(s_{\!\frac 12},(0,\dots,0))=(s_0s_{\!\frac 12},(0,\dots,0))$. The action of $E^+$ in these coordinates  is to add (or subtract) $E^+$ when the $\D_2$ index is $e$ or $s_0s_{\!\frac 12}$ and to add $E^-$ when the index is $s_0$ or $s_{\!\frac 12}$. Hence the effect of the element $E^+$ is to exchange the pair of suspension points for the first suspension with the pair for the second.

By the results of Section \ref{involutions} there is a homeomorphism $\Sigma\Sigma(W)\to (W\times Y)/\sim$, where the equivalence relation identifies $W$ to a point whenever the point of $Y$ is on the boundary. Under this identification the action of $E^+$ is the above action on $W$, and coincides with the reflection $s\leftrightarrow t$ on $Y$.

As noted above, the space $\Sigma\Sigma(W)$ is also identified with the ${\delta_\odd}$-sphere appearing in the $\SO(2n)$ case and therefore carries an action of $\ZDSO$. Now $W$ is coordinatised as before, while the suspension points for the first suspension are now the points $(0,\dots,0)$ and $(\frac 12,\dots,\frac 12)$ in the boundary of $[0,\frac 12]^{\delta_\odd}$. The second suspension is to two copies of the centre $(\frac 14,\dots,\frac 14)$ of this cube.  Recall that the action of $\ZDSO$ exchanges the points of the first suspension and fixes the points of the second. Hence the action of $\ZDSO$ on $(W\times Y)/\sim$ is the above action on $W$ and the reflection on $Y$ negating one of the two coordinates.

We now observe that the map $(t,s) \mapsto  (t+s,t-s)$ composed with the radial rescaling from $[-1,1]^2$ to $Y$ intertwines the actions of $\langle E^+\rangle$ and $\ZDSO$. We conclude that the quotients in the semi-spin case therefore are homeomorphic to the corresponding quotients from the $\PSO(2n)$ case (where $n$ is even).

\bigskip

\subsubsection{Type IIA summary}

For the case where the cycle type has no negative cycles and at least one odd positive cycle, the quotients in the $\SO(2n)$ and $\Spin(2n)$ cases are homeomorphic, and are given by the product of a polysimplex of dimension $\delta_\ev$ with either a ${\delta_\odd}$-ball or a ${\delta_\odd}$-sphere depending on whether or not $d_i=1$ for all odd $n_i$, see Table \ref{IIA Spin/SO}.

\smallskip
We recall that $\delta_\ev,\delta_\odd$ are respectively the number of positive cycles of even and odd length, while $q_\ev=\sum_{n_i\text{ even}} \lceil d_i/2\rceil$ and $q_\odd=\sum_{n_i\text{ odd}} \lceil d_i/2\rceil$ where $d_i$ is the multiplicity of the positive cycle of length $n_i$.

\smallskip

\begin{table}[h]

\[
\begin{array}{c|c}
 \SO(2n) \text{ and } \Spin(2n) \text{ cases }  \\
\hline
 d_i = 1 \text{ for all odd } n_i& V(\delta_\ev,0)\times S^{\delta_\odd}\\
 d_i > 1  \text{ for some  odd } n_i& V(\delta,0)\\
\end{array}
\]

\smallskip

\caption{\label{IIA Spin/SO}The structure of the sector associated to a conjugacy class with no negative cycles and at least one odd positive cycle in the $\SO(2n)$ and $\Spin(2n)$ cases.
}
\end{table}

For the $\PSO$ case the quotients depend on whether $n$ is odd (see Table \ref{IIA PSO(2n) n odd}) or even (see Table \ref{IIA PSO(2n)/SSpin(2n) n  even}).  The quotients for the semi-spin case agree with those for $\PSO(2n)$ with $n$ even. We recall that $\unitbundle{r}pq$ denotes the unit ball bundle in the sum of $p$ copies of the trivial line bundle and $q$ copies of the tautological line bundle over $\RP^{r}$.

\bigskip

\begin{table}[h]
\[
\begin{array}{c|c}
 \PSO(2n) \text{ for $n$ odd}  \\
\hline
 d_i = 1 \text{ for all odd } n_i& \unitbundle{\delta_\odd}{\delta_\ev-q_\ev}{q_\ev}\\
 d_i > 1  \text{ for some  odd } n_i& V(\delta-q,q)\\
\end{array}
\]

\smallskip

\caption{\label{IIA PSO(2n) n odd}The structure of the sector associated to a conjugacy class with no negative cycles and at least one odd positive cycle in the $\PSO(2n)$ case when $n$ is odd.
}
\end{table}

\begin{table}[h]
\[
\begin{array}{c|c}
 \PSO(2n) \text{ and } \SSpin^\pm(2n)\\ \text{ for $n$ even}  \\
\hline
 d_i = 1 \text{ for all odd } n_i&\mathcal{D}(V(\delta_\ev-q_\ev, q_\ev+\delta_\odd),\unitbundle{\delta_\odd-1}{\delta_\ev-q_\ev}{q_\ev}) \\
 d_i > 1  \text{ for some  odd } n_i& V(\delta-q+1,q-1)\\
\end{array}
\]
where $\mathcal{D}(X,Y)$ denotes the double of $X$ over $Y$.

\smallskip

\caption{\label{IIA PSO(2n)/SSpin(2n) n  even}The structure of the sector associated to a conjugacy class with no negative cycles and at least one odd positive cycle in the $\PSO(2n)$ and $\SSpin^\pm(2n)$ cases when $n$ is even.
}
\end{table}

\bigskip

\noindent By collapsing fibres, we note that the unit ball bundle $\unitbundle{\delta_\odd}{\delta_\ev-q_\ev}{q_\ev}$
retracts onto the base space $\RP^{\delta_\odd}$, while the double 
\[
\mathcal{D}(V(\delta_\ev-q_\ev, q_\ev+\delta_\odd),\unitbundle{\delta_\odd-1}{\delta_\ev-q_\ev}{q_\ev})
\]
retracts onto $\Sigma\RP^{\delta_\odd-1}$. If $\delta_\ev=0$ then there are no fibres and we have exactly the base space $\RP^{\delta_\odd}$ or $\Sigma\RP^{\delta_\odd-1}$.

\FloatBarrier
\section{Examples and applications}
\subsection{The Lie groups of type $D_4$}\label{D4 Lie groups}
For $n=4$, the lattice of Lie groups arising is somewhat simplified by the classical fact that the groups $\SO(8),\SSpin^\pm(8)$ are all isomorphic. That is a consequence of the identification of the three non-trivial elements of the centre of $\Spin(8)$ by the triality automorphism. The relationship at the level of the inertial action of the Weyl group is subtle. Partly to illustrate this point and partly to provide a concrete example of the calculations provided above we will now summarise the computations for the Lie groups $\Spin(8),\SO(8),\SSpin^\pm(8)$ and $\PSO(8)$.

\begin{table}[h]
\centering
\renewcommand{\arraystretch}{1.35}
\setlength{\tabcolsep}{6pt}

\begin{tabular}{|c|c|c|c|}
\hline
\textbf{Type} &
\textbf{Cycle type} &
$\mathbf{SO(8)}$ &
$\mathbf{SSpin^{+}(8)}$ \\
\hline

\multirow{4}{*}{IA}
&
$(4)^{+}$
&
$V(1,0)$
&
$2V(0,1)\sqcup 2V(1,0)$
\\
\cline{2-4}

&
$(4)^{-}$
&
$V(1,0)$
&
$V(1,0)$
\\
\cline{2-4}

&
$(2^2)^{+}$
&
$V(2,0)$
&
$2V(1,1)\sqcup V(2,0)$
\\
\cline{2-4}

&
$(2^2)^{-}$
&
$V(2,0)$
&
$V(2,0)$
\\
\hline

\multirow{3}{*}{IIA}
&
$(3,1)$
&
$S^{2}$
&
$\Sigma \RP^1$
\\
\cline{2-4}

&
$(2,1^2)$
&
$V(3,0)$
&
$V(2,1)$
\\
\cline{2-4}

&
$(1^4)$
&
$V(4,0)$
&
$V(3,1)$
\\
\hline

\multirow{4}{*}{IB}
&
$(2, \bar{1}^2)$&$3V(1,0)$&$3V(1,0)$
\\
\cline{2-4}
&\cellcolor[gray]{0.9}$(\bar{3},\bar{1})$&\cellcolor[gray]{0.9}$4V(0,0)$&\cellcolor[gray]{0.9}$4V(0,0)$
\\
\cline{2-4}
&\cellcolor[gray]{0.9}$(\bar{2}^2)$&\cellcolor[gray]{0.9}$3V(0,0)$&\cellcolor[gray]{0.9}$3V(0,0)$
\\
\cline{2-4}
&\cellcolor[gray]{0.9}$(\bar{1}^4)$&\cellcolor[gray]{0.9}$5V(0,0)$&\cellcolor[gray]{0.9}$5V(0,0)$
\\
\hline

\multirow{2}{*}{IIB}
&
$(1, \bar{2},\bar{1})$
&
$4V(1,0)$
&
$V(1,0)$
\\
\cline{2-4}

&
$(1^2, \bar{1}^2)$
&
$3V(2,0)$
&
$V(2,0)$
\\
\hline

\end{tabular}

\bigskip

\caption{\label{table:D4}
Comparison of the $\SO(8)$ and $\SSpin^{+}(8)$ entries. Elliptic cases are in shaded cells. }
\end{table}

Recall that entries of the form $V(r,0)$ are balls of dimension $r$, while those denoted $V(s,1)$ are the product of a ball of dimension $s$ with the cone on $\RP^0$, so they too are homeomorphic to balls. With that remark it is clear that in most cases the $\SO(8)$ entries are homeomorphic to the corresponding entries for $\SSpin^+(8)$. The exceptions are precisely those arising in the type IIB and the positive type IA rows. To explain what is really going on we need to consider the effect of the equivariant identification on cycle types.

Identifying the Cartan subalgebra with $\R^4$ as usual the lattices can be explicitly described as follows:

\bigskip

\begin{table}[h]

\centering
\renewcommand{\arraystretch}{1.35}
\setlength{\tabcolsep}{6pt}

\begin{tabular}{|c|c|}
\hline
$G$ & Lattice basis\\
\hline
$\Spin(8)$ & $(-1,1,0,0),(1,1,0,0),(0,-1,1,0),(0,0,-1,1)$
\\
\hline
$\SO(8)$ &$(1,0,0,0),(0,1,0,0),(0,0,1,0),(0,0,0,1)$
\\
\hline
$\SSpin^+(8)$ &
${\begin{aligned}
  &\textstyle u_1^+=(\frac12,\frac12,\frac12,\frac12),
&\textstyle u_2^+=(\frac12,-\frac12,\frac12,-\frac12),\\
&\textstyle u_3^+=(\frac12,-\frac12,-\frac12,\frac12),
&\textstyle u_4^+=(\frac12,\frac12,-\frac12,-\frac12)\phantom{,}
\end{aligned}}^{\strut}_{\strut}$
\\
\hline
$\SSpin^-(8)$ &
${\begin{aligned}
  &\textstyle u_1^-=(-\frac12,\frac12,\frac12,\frac12),
&\textstyle u_2^-=(-\frac12,-\frac12,\frac12,-\frac12),\\
&\textstyle u_3^-=(-\frac12,-\frac12,-\frac12,\frac12),
&\textstyle u_4^-=(-\frac12,\frac12,-\frac12,-\frac12)
\phantom{,}
\end{aligned}}^{\strut}_{\strut}$
\\
\hline
$\PSO(8)$&$(1,0,0,0),(0,1,0,0),(0,1,0,0),(\frac12,\frac12,\frac12,\frac12)$\\
\hline
\end{tabular}

\bigskip

\caption{The lattices for the  real Lie groups of type $D_4$}
\end{table}

We observe that the lattices for $\SSpin^\pm(8)$ are isometric to the standard integer lattice appearing for $\SO(8)$, with this isometry preserving the set of roots. Indeed the set of roots is exactly the set of vectors of the form $\pm u\pm v$ where 
\[u,v\in \{(1,0,0,0),(0,1,0,0),(0,0,1,0),(0,0,0,1)\}
\]
but this is equal to the set of vectors $\pm u\pm v$ where 
\[u,v\in \left\{\left(\frac12,\frac12,\frac12,\frac12\right),
\left(\frac12,-\frac12,\frac12,-\frac12\right),\left(\frac12,-\frac12,-\frac12,\frac12\right),
\left(\frac12,\frac12,-\frac12,-\frac12\right)\right\}.
\]
Since this gives an isomorphism of the root data it follows that the Lie groups $\SO(8)$ and $\SSpin^+(8)$ are isomorphic.  The same argument shows $\SO(8)$ is also isomorphic to $\SSpin^-(8)$.

Recalling that the cycle structure for an element of the Weyl group depends on the choice of signed basis, we observe that the above isomorphisms will sometimes change the cycle structure.  For instance consider the $4$-cycle $(1\,2\,3\,4)(-1\,-2\,-3\,-4)$ acting in the usual way on the standard basis.  It acts on the $\SSpin^+(8)$ basis as follows:
\[
u_1^+\mapsto u_1^+,\quad u_2^+\mapsto-u_2^+,\quad u_3^+\mapsto u_4^+,\quad u_4^+\mapsto -u_3^+.
\]
Hence under the isomorphism the cycle type changes from $(4)^+$ to $(1,\overline{2},\overline{1})$. Conversely if we begin with the element $(2\,-2),(3\,4\,-3\,-4)$ of cycle type $(1,\overline{2},\overline{1})$ then its action on the $\SSpin^+(8)$ basis is:
\[
u_1^+\mapsto u_3^+,\quad u_2^+\mapsto u_1^+,\quad u_3^+\mapsto u_4^+,\quad u_4^+\mapsto u_2^+.
\]
giving an element of cycle type $(4)^+$. Similarly elements of cycle types $(2,2)^+, (1,1,\overline 1, \overline 1)$ are switched by the identification, and the corresponding sectors in both cases are homeomorphic.

\bigskip

The entries in Table \ref{table:D4} allow us to compute the cohomology $H^*(X\q W, \mathbb C)$ which, via the standard equivariant Chern character argument \cite{BaumConneschern}, computes the $W$-equivariant $K$-theory of the corresponding tori for the groups $\SO(8), \SSpin^+(8)$ after tensoring with $\mathbb C$. This in turn (by \cite{ABPS2}) computes the $K$-theory for the tempered Iwahori algebra $\fA(\G)$.

For $\SO(8)$ the cohomology has rank $29$ in dimension $0$, rank $0$ in dimension $1$ and rank $1$ in dimension $2$, with the latter arising from the sphere $S^2$ in the  first sector in the Type IIA case. In $K$-theory this yields rank $30$ in even dimensions and $0$ in the odd dimensions. Since the Lie groups $\SO(8)$ and $\SSpin^+(8)$ are isomorphic we expect the same answer in the latter case, and this is witnessed by the homeomorphisms between the corresponding sectors discussed above.

We therefore obtain 

\begin{align*}
K_0(\fA(\G))\otimes \mathbb C &\cong K_0^W(C(\bT_G))\otimes \mathbb C\cong\mathbb C^{30}\\
K_1(\fA(\G))\otimes \mathbb C &\cong K_1^W(C(\bT_G))\otimes \mathbb C= 0,
\end{align*}
where $\G$ denotes the $p$-adic group $\SO(8,F)$ or $\HSpin^\pm(8,F)$ and $\bT_G$ is a maximal torus in the real Lie group $G=\SO(8,\R)$ or $\SSpin^\pm(8,\R)$.

\medskip

\bigskip

We now consider the dual cases of $\Spin(8)$ and $\PSO(8)$. By the results of \cite{NPWKdual} the equivariant $K$-theory groups for the maximal tori in these cases are naturally dual to one another and hence must be rationally isomorphic by the Universal Coefficient Theorem. Moreover in \cite{NPWAn} we strengthened this to show that this duality respects the stratification into sectors by conjugacy classes of the Weyl group (after tensoring with $\C$).

Table \ref{Spin(8), PSO(8)} shows that in this case the duality is mediated by \emph{homeomorphisms} at the level of sectors, which are in turn parameterised by the cycle structures, recalling that we need to take a little care with the split conjugacy classes which appear as Type IA cycle structures.

\bigskip

\begin{table}[h]
\centering
\renewcommand{\arraystretch}{1.35}
\setlength{\tabcolsep}{6pt}

\begin{tabular}{|c|c|c|c|}
\hline
\textbf{Type} &
\textbf{Bipartition} &
$\mathbf{Spin(8)}$ &
$\mathbf{PSO(8)}$ \\
\hline

\multirow{4}{*}{IA}
&
$(4)^{+}$
&
2$V(1,0)$
&
$V(0,1)\sqcup V(1,0)$
\\
\cline{2-4}

&
$(4)^{-}$
&
2$V(1,0)$
&
$V(0,1)\sqcup V(1,0)$
\\
\cline{2-4}

&
$(2^2)^{+}$
&
2$V(2,0)$
&
$V(1,1)\sqcup V(2,0)$
\\
\cline{2-4}

&
$(2^2)^{-}$
&
2$V(2,0)$
&
$V(1,1)\sqcup V(2,0)$
\\
\hline

\multirow{3}{*}{IIA}
&
$(3,1)$
&
$S^2$
&
$\Sigma \RP ^1$
\\
\cline{2-4}

&
$(2,1^2)$
&
$V(3,0)$
&
$V(2,1)$
\\
\cline{2-4}

&
$(1^4)$
&
$V(4,0)$
&
$V(3,1)$
\\
\hline

\multirow{3}{*}{IB}&
$(2, \bar{1}^2)$&$4V(1,0)$&$V(0,1)\sqcup 3V(1,0)$
\\
\cline{2-4}
&\cellcolor[gray]{0.9}$(\bar{3},\bar{1})$&\cellcolor[gray]{0.9}$4V(0,0)$&\cellcolor[gray]{0.9}$4V(0,0)$
\\
\cline{2-4}
&\cellcolor[gray]{0.9}$(\bar{2}^2)$&\cellcolor[gray]{0.9}4V(0,0)&\cellcolor[gray]{0.9}$4V(0,0)$
\\
\cline{2-4}
&\cellcolor[gray]{0.9}$(\bar{1}^4)$&\cellcolor[gray]{0.9}5V(0,0)&\cellcolor[gray]{0.9}$5V(0,0)$
\\
\hline

\multirow{2}{*}{IIB}
&
$(1, \bar{2},\bar{1})$
&
$2V(1,0)$
&
$2V(1,0)$
\\
\cline{2-4}

&
$(1^2, \bar{1}^2)$
&
$2V(2,0)$
&
$V(2,0)\sqcup V(1,1)$
\\
\hline

\end{tabular}

\bigskip

\caption{\label{Spin(8), PSO(8)}Comparison of the $\Spin(8)$ and $\PSO(8)$ entries. Elliptic cases are in shaded cells.}
\end{table}

This yields:
\begin{align*}
K_0(\fA(\G))\otimes \mathbb C &\cong K_0^W(C(\bT_G))\otimes \mathbb C\cong \mathbb C^{33}\\
K_1(\fA(\G))\otimes \mathbb C &\cong K_1^W(C(\bT_G))\otimes \mathbb C= 0,
\end{align*}
where $\G$ is either of the $p$-adic groups $\Spin(8,F)$, $\PSO(8,F)$ and Langlands dually $G$ is either $\PSO(8,\R)$ or $\Spin(8,\R)$.

As shown in \cite{NPWKdual} there is an isomorphism between the $C^*$-algebra $C(\bT_G)\rtimes W$ and the reduced group $C^*$-algebra of the extended affine Weyl group of the Langlands dual of $G$.  Hence the above calculation also computes the $K$-theory for the extended affine Weyl groups of both $\Spin(8,\R)$ and of $\PSO(8,\R)$, while the earlier calculations determine the $K$-theory of the extended affine Weyl groups of both $\SO(8,\R)$ and $\SSpin^\pm(8,\R)$.

\subsection{Ranks of the equivariant $K$-theory groups $K_*^W(C(\bT_G))$}\label{K-theory ranks}

In \cite{Sol} Solleveld computed the ranks of the $K$-theory groups $K_*^W(C(\bT_G))$ when $G=\SO(2n)$ by counting suitable bipartitions and partitions. Our computations above naturally reduce to the following which can easily be seen to be equivalent to Solleveld's result.

For given $n$ we start by defining $b_n^\ev$ to be the sum of the terms 
\[
C((n_j)^{d_j};(m_i)^{c_i}):=\prod_i(c_i+1).
\]
over bipartitions $((n_j)^{d_j};(m_i)^{c_i})$ of $n$ where $\sum_i c_i$ is even. In the case that there are no cycles in the second term of the bipartition we view the corresponding empty product as $1$. We now define $q_n$ to be the number of partitions of $n$ with distinct odd parts (allowing the case of no odd parts).

We can now give the formulae for the ranks of $K$-theory for the case of $\SO(2n)$. 
\nopagebreak

\noindent For $n$ even:
\[
\rank(K_0)=b_n^\ev+q_n,\quad \rank(K_1)=0.
\]
For $n$ odd:
\[
\rank(K_0)=b_n^\ev,\quad \rank(K_1)=q_n.
\]

Turning now to the cases of $\Spin(2n)$ and $\PSO(2n)$ (which have the same $K$-theory ranks) we define $\tilde b_n^\ev$ to be the sum of $\lceil\frac 12 C((n_j)^{d_j};(m_i)^{c_i})\rceil$, again over bipartitions with $\sum_ic_i$ even. Let $p_{n/2}$ denote the number of partitions of $n/2$, interpreted as zero when $n$ is odd. For the cases of $\Spin(2n)$ and $\PSO(2n)$ the $K$-theory groups have the following ranks.

\smallskip

\noindent For $n$ even:
\[
\rank(K_0)=\tilde b_n^\ev+2p_{n/2} + q_n,\quad \rank(K_1)=0.
\]
For $n$ odd:
\[
\rank(K_0)=\tilde b_n^\ev+2p_{n/2},\quad \rank(K_1)=q_n.
\]

\smallskip

In particular we obtain positive $K_1$ rank only if $n$ is odd. We note that the $K_1$ ranks are the same for all three forms: $\Spin, \SO$ and $\PSO$. For $\Spin,\PSO$ this is expected by our duality results however we would not necessarily expect $\SO$ to also agree with these and indeed the rank of $K_0$ for $\SO$ is typically different to the rank for $\Spin$ or $\PSO$.

\bigskip
We now turn to the groups of type $B_n$ and $C_n$. For the groups $\SO(2n+1)$ and $\Sp(n)$ the K-theory ranks were computed by Solleveld and our answers below agree with that computation. We define $b_n$ to be the sum of $C((n_j)^{d_j};(m_i)^{c_i})$ over all bipartitions. Then:

\smallskip

\[
\rank(K_0)= b_n, \quad \rank(K_1)=0,
\]

\smallskip

\noindent For the groups $\Spin(2n+1), \PSp(n)$ we define 
$\tilde b_n$ to be the sum of the integers $\lceil \frac 12 C((n_j)^{d_j};(m_i)^{c_i})\rceil$ over all bipartitions, and define  $r_n$ to be the number of bipartitions of $n$ with no positive parts of odd length. From Tables \ref{simply_connected_B_n_sectors} and \ref{adjoint_C_n_sectors} we obtain:
\[
\rank(K_0)= \tilde b_n + r_n, \quad \rank(K_1)=0.
\]
We note that while the $K$-theory ranks for $\Spin(2n+1)$ and $\PSp(n)$ are the same, the orbifold geometry of the sectors in these cases may differ even though the homotopy type is the same. A ball arising in a $\Spin(2n+1)$ sector may be replaced by (a bundle over) the cone on a real projective space in the corresponding $\PSp(n)$ sector. A similar remark applies to $\Spin(2n)$ and $\PSO(2n)$.

\subsection{The homotopy sector conjecture and rational homotopy}\label{sector conjecture}

In general there is no equivariant identification of the action of a Weyl group on the maximal tori of its various Lie groups, however as discussed above, in \cite{NPWKdual} we showed that there is an equivariant duality in $K$-theory between Langlands dual tori. Hence there is an isomorphism at the level of complex cohomology for the corresponding extended quotients, which by \cite{NPWAn} is stratified by the sectors.

The fact that the sectors in the $\Spin(8)$ and $\PSO(8)$ cases considered in Section \ref{D4 Lie groups} are actually equivalent up to homeomorphism, not just cohomological rank, is unusual as shown by our previous calculations for Weyl groups of type $A_n$ and $E_6$ in \cite{NPWAn, NPWE6}. On the other hand, in every case computed to this point the cohomological equivalence was carried by a homotopy equivalence at the level of sectors, leading us to conjecture that this would always be the case. In his thesis, \cite{Majda} the first author showed that this conjecture was also true for the Lie group of type $E_7$. It is trivial for the Lie groups $E_8, F_4$ and $G_2$ since these are self-dual, and the earlier results in the current paper also establish it for Lie groups of type $B_n, C_n$ and $D_4$. 

With that context we now give a counterexample to our conjecture. We will consider the extent to which it fails and observe that it can be rescued by replacing the conjecture that Langlands dual sectors are homotopy equivalent by the observation that this is always rationally true. 

For the Lie groups of type $D_9$ and  other cases discussed below, the conjectured homotopy equivalence fails precisely because of the appearance of sectors homeomorphic to a sphere in the $\Spin(18)$ case and a projective space in the corresponding $\PSO(18)$ case. These are not homotopy equivalent, moreover even their integral homologies are different.

Specifically, consider the sector of cycle type $(5,3,1)$ for the Weyl group $D_9$ acting on the maximal torus for $\Spin(18)$. The multiplicities are all $1$ with $\delta_{\odd} = 3$, so according to our recipe the sector is a sphere $S^3$. On the other hand for the Langlands dual Lie group $\PSO(18)$, the sector is $\RP^3$, and is therefore not homotopy equivalent to $S^3$, although these sectors do have the same rational cohomology.

\medskip

For $n=16$ consider the cycle type $(7,5,3,1)$. The corresponding $\Spin(32)$ and $\SO(32)$ sectors are homeomorphic to the $4$-sphere, or equivalently the suspension of $S^3$, while the corresponding $\PSO(32)$ and $\SSpin^\pm(32)$ sectors are the suspension of $\RP^{3}$. Specifically, the latter quotient is the double of $V(0, 4)$, which is the cone on $\RP^3$, along the copy of $\RP^3$ in its boundary. While these spaces are not homotopy equivalent they are \emph{rationally} homotopy equivalent. Note that the group $\SSpin^+(32)$ arising here is the locus for the $\SO(32)$ heterotic string in physics.

We can make this construction for any conjugacy class of type IIA whose odd positive cycles are of distinct lengths. However, if the number of odd parts is less than three the projective space that appears in the $\PSO(2n)$ sector is homeomorphic to  the corresponding sphere in the $\Spin(2n)$ sector. It follows that we can construct non-homotopic dual sectors in this way when $n$ is odd and at least $9$ or even and at least $16$.

If the cycle type satisfies the above conditions but has even positive parts as well then the $\Spin(2n)$ sector has an additional direct factor homeomorphic to a ball, while the corresponding $\PSO(2n)$ sector  has a fibre isomorphic to a unit ball  of the same dimension. This of course does not change the homotopy type. This allows us to construct examples for  $n=11,13,15$ where the $\Spin(2n)$ sectors are not homotopy equivalent to the corresponding $\PSO(2n)$ sectors, for example using the cycle types $(5,3,2,1), (5,4,3,1)$ or $(6,5,3,1)$. In these cases the corresponding $\Spin(2n)$ sectors are homeomorphic to products $[0,1]\times S^3$ while the $\PSO(2n)$ sectors are unit interval bundles over $\RP^3$.

In summary, the sector homotopy conjecture holds except in the following specific sectors: 
 sectors of type IIA for the dual Lie groups $\Spin(2n)$ and $\PSO(2n)$ corresponding to cycle-types with distinct odd cycle lengths and at least three such odd cycles.  In particular counterexamples exist for these groups if and only if $n\geq 9, \not=10,12,14$. This establishes Theorem \ref{theorem}.

\bigskip

A continuous map between \emph{simply connected} spaces $f:X\to Y$ is a \emph{rational homotopy equivalence} if it induces isomorphisms
\[
f_* :\pi_i(X)\otimes_\Z \Q\to \pi_i(Y)\otimes_\Z \Q
\]
for all $i>1$. More generally we will say that a continuous map $f:X\to Y$ is a rational homotopy equivalence if it induces a bijection $f_*:\pi_0(X)\to \pi_0(Y)$ and for each $x_0$ in $X$:
\begin{itemize}
    \item the groups $\pi_1(X,x_0)$ and $\pi_1(Y,f(x_0))$ are abelian, and
    \item $f_*:\pi_i(X,x_0)\otimes_\Z \Q\to \pi_i(Y,f(x_0))\otimes_\Z \Q$ is an isomorphism for all $i>0$.
\end{itemize}

An example of this is furnished by the quotient map $f: S^\delta\to \RP^\delta$. The homotopy exact sequence from the associated fibration shows immediately that except for $i=1$ the $f_*:\pi_i(S^\delta) \to \pi_i(\RP^\delta)$ is an isomorphism. For $\delta>1$ we have
\[
\pi_1(S^\delta)=0=\pi_1(\RP^\delta)\otimes_\Z \Q
\]
while for $\delta=1$ the function $f$ induces the times $2$ map
\[
\pi_1(S^\delta)\otimes_\Z\Q\cong \Q\to\pi_1(\RP^\delta)\otimes_\Z \Q\cong \Q.
\]
Indeed in this example, we do not need to use $\Q$ as coefficients, it is sufficient to invert $2$, and we say that $f$ is a \emph{dyadic homotopy equivalence}.

We now consider the \emph{suspension} of the double cover of an odd-dimensional real projective space by a sphere. To show that this is also a rational (indeed dyadic) homotopy equivalence we will use the following theorem.

\begin{theorem}[Whitehead-Serre Theorem, {\cite[Theorem 8.6]{felix}}]
    If $\mathbb K$ is a subring of $\Q$ and if $f : X\rightarrow Y$ is a continuous map between simply connected spaces then $f_*:\pi_*(X, x_0) \otimes_\Z \mathbb K\rightarrow \pi_*(Y, f(x_0)) \otimes_\Z \mathbb K$ is an isomorphism if and only if $f_*:H_*(X, \mathbb K)\rightarrow H_*(Y, \mathbb K)$  is an isomorphism. 
\end{theorem}

We apply this with $\mathbb K=\Z[1/2]$,  $X= S^\delta=\Sigma S^{\delta-1}$, $Y=\Sigma \RP^{\delta-1}$, where $\delta$ is even. The map $f$ is the suspension of the double cover $S^{\delta-1}\rightarrow \RP^{\delta-1}$. Since $\delta-1$ is odd the double cover induces an isomorphism $H_*(S^{\delta-1},\mathbb K)\cong
H_*(\RP^{\delta-1},\mathbb K)$, indeed both homologies are $\mathbb K$ in dimensions $0$ and $\delta-1$ and are $0$ otherwise. It follows that $f$ induces an isomorphism on  the homologies of the suspensions. Hence by the Whitehead-Serre Theorem the map $f:S^\delta\to\Sigma\RP^{\delta-1}$ is a dyadic homotopy equivalence.  We note  that this is \emph{false} for $\delta$ odd.

\bigskip

By Theorem \ref{theorem} when the sectors of dual groups are not homotopy equivalent the only possibility is that they are of the form $S^\delta$ and $\RP^\delta$ with $\delta$ odd, or of the form $S^\delta$ and $\Sigma\RP^{\delta-1}$ with $\delta$ even. This completes the proof of Corollary \ref{corollary}.

\subsection{Injectivity of sectors}\label{injectivity of sectors}

There is of course a map from an extended quotient $X\q W$ to the ordinary quotient $X/W$. We say that a sector $X^w/Z_W(w)$ is \emph{injective} if the restriction of the above map to $X^w/Z_W(w)\to X/W$ is injective.

Sectors need not in general be injective since there may be points $x\in X^w$ and elements $g\in W$ such that $gx\in X^w$ but there is no element of $Z_W(w)$ mapping $x$ to $gx$. The sector $X^w/Z_W(w)$ is injective if and only if for each $x\in X^w$ the intersection of the $W$-orbit $W\!x$ with $X^w$ is exactly the $Z_W(w)$-orbit of $x$.

If a sector of $\bT_G\q W$ is injective then this has implications for the geometry of its image in the Iwahori-spherical block. Since the cuspidal support map $\widehat{\fA(\G)}\to \bT_G/W$ is continuous the following commutative diagram shows that the restriction to a single injective sector of the map $\bT_G\q W\to\widehat{\fA(\G)}$ has continuous inverse:
\[
\begin{tikzcd}
\bT_G\q W \arrow[rr,
"\mu"]\arrow[dr]\arrow[ddr]& &\widehat{\fA(\G)}\arrow[ddl] \\
&(\bT_G\q W)_2 \arrow[ur] \arrow[d]&\\
  &\bT_G/W
\end{tikzcd}
\]

We therefore deduce the following result.

\begin{theorem}
    Let $\mu:\bT_G\q W\to \widehat{\fA(\G)}$ be an ABPS map induced by a choice of $c-Irr$ system. For $w\in W$, if the sector $\bT_G^w/Z_W(w)$ is injective then $\mu(\bT_G^w/Z_W(w))$ is Hausdorff and the restriction of $\mu^{-1}$ gives a continuous bijection
    \[
    \mu^{-1}:\mu(\bT_G^w/Z_W(w))\to\bT_G^w/Z_W(w).
    \]

    If $\G$ is split and has connected centre then $\mu$ can be constructed to be continuous so we have a homeomorphic copy of the sector in $\widehat{\fA(\G)}$.
\end{theorem}

We remark that in general for an injective sector we have a homeomorphism if and only if $\mu$ can be chosen such that the image $\mu(\bT_G^w/Z_W(w))$ is compact.

We now turn to the question of showing that certain specific sectors are injective.   In order to prepare for the next result, let $F$ denote a finite extension of $\Q_p$ with $p > 2$.   The split $18$-dimensional quadratic form is 
\[
q = x_1x_{18} + x_2x_{17} + \cdots + x_9x_{10}.
\]
Then
\[
\SO(18,F): = \SO(q)
\]
is $F$-split with split rank $9$.   The groups $\Spin(18,F)$ and $\PSO(18,F)$ are also split, and   $\PSO(18,F)$ is the split adjoint form of $D_9$.

\begin{theorem}
If $\G=\PSO(18,F)$  then there is a subspace of $\widehat{\fA(\G)}$, corresponding to the sector of type $(5,3,1)$, which is homeomorphic to $S^3$. If $\G=\Spin(18,F)$ then there is a subspace of $\widehat{\fA(\G)}$, corresponding to the sector of type $(5,3,1)$, which maps continuously bijectively to $\RP^3$.
\end{theorem}

To establish this we consider $D_9$ and the cycle type $(5,3,1)$. For the sake of definiteness we take $w$ to have cycles permuting the first $5$ and the next $3$ coordinates.

We note that as $w$ is of type IIA (odd positive cycles and no negative cycles) it follows that there are surjections from the fixed set  $\coveringtorus^w$ in the $\Spin$ torus to the fixed set $T^w$ in $\SO$, and from $T^w$ to the fixed set $(T/\ZDSO)^w$ in the $\PSO$ torus. Hence it suffices to prove the injectivity result for the case of the fixed set $\coveringtorus ^w$. This fixed set is
\[
\{(\alpha_5,\alpha_5,\alpha_5,\alpha_5,\alpha_5,\alpha_3,\alpha_3,\alpha_3,\alpha_1,\beta) : \alpha_5^5\alpha_3^3\alpha_1\beta^2=1\}
\]

Let $x\in \coveringtorus^w$. As noted above injectivity amounts to the requirement that $W\!x\cap\coveringtorus^w=Z_W(w)x$. It is  therefore sufficient to consider a single $x$ from each $Z_W(w)$ orbit.

If $\alpha_1,\alpha_3,\alpha_5$ agree up to the inverse map $\alpha\mapsto \alpha^{-1}$ then, since $Z_W(w)$ can act to invert any two of $\alpha_1,\alpha_3,\alpha_5$, we can assume without loss of generality that $\alpha_1,\alpha_3,\alpha_5$ are all equal.

Similarly if two coordinates agree up to the inverse, with the third value different then by inverting the third coordinate and one of the other two we may assume that two of the three values agree exactly.

This gives three cases:
\begin{enumerate}
    \item $\alpha_1=\alpha_3=\alpha_5=\alpha$ for some $\alpha\in \TT$;
    \item $\alpha_1,\alpha_3,\alpha_5$ take the values $\alpha,\alpha,\alpha'$ (up to permutation) for some $\alpha,\alpha'\in \TT$ with $\alpha'\neq \alpha,\alpha^{-1}$;
    \item $\alpha_i\neq \alpha_j,\alpha_j^{-1}$ for any $i\neq j$.
\end{enumerate}

A general element of the Weyl group has the form $\sigma_Js$ where $s$ is a (positive) permutation and $J$ is an even cardinality subset of $\{1,\dots,9\}$. Since distinct coordinate values remain distinct after applying the inversion $\alpha\mapsto \alpha^{-1}$ it follows that if $\sigma_Jsx$ is in the fixed set $\coveringtorus^w$ then $s$ can only permute coordinates within the blocks where $x$ is constant, i.e.\ $s$ must fix the point $x$. This is because the constant blocks must have distinct sizes and therefore cannot be interchanged.

We are left with the action of $\sigma_J$ on $x$. If $\alpha_5$ is $\pm1$ then replacing $J$ by a set $J'$ which is either $J\cup\{1,2,3,4,5\}$ or $J\setminus \{1,2,3,4,5\}$, chosen so that $J'$ again has even size, will not change the action on $x$, i.e.\ $\sigma_J x = \sigma_{J'} x$. 
It is clear that the $\alpha$ coordinates are unchanged and moreover the $\beta$ coordinate is also unchanged since $J'$ differs evenly from $J$. A similar argument applies to $\alpha_3$ hence we may assume that on any part of the partition where the value taken by the coordinates is $\pm1$ the set $J$ contains either all or none of these coordinates.

This leaves us to consider the action of $\sigma_J$ on those coordinates that are not equal to $\pm1$. Again, either all or none of the coordinates of each part must lie in $J$ otherwise $\sigma_Jx$ is not in the fixed set.

Thus in all cases we deduce that the action of $\sigma_Js$ on $x$ agrees with the action by some element of the centraliser.

\bigskip

We conclude by observing that the sector corresponding to cycle type $(7,5,3,1)$ is \emph{not} injective since in the Weyl group we have an involution exchanging the $7+1$ coordinates of the first and last cycle with the $5+3$ coordinates of the other two cycles.  This identifies points of the fixed set which are not identified by the action of the centraliser.

\newpage

\begin{appendix}
\section{Sectors, cohomology and $K$-theory for $\Spin(18) $ and $\PSO(18)$}
\label{appendixA}

As described in Section \ref{sector conjecture}, the most interesting topology across all the sectors for $p$-adic groups appears for those with Weyl group of type $D_n$. For all other cases, the sectors are homotopy equivalent (often homeomorphic) to balls, spheres or tori. The first truly interesting case appears for the $p$-adic groups $\PSO(18, F), \Spin(18,F)$  with Langlands duals $\Spin(18, \mathbb C), \PSO(18, \mathbb C)$ respectively. These provide counterexamples to our original homotopy sector conjecture, and also the first examples of sectors with $\RP^n$ topology for $n\geq 3$. Since these Lie groups also play an important role in mathematical physics, and in order to provide a full calculation of the $K$-theory for these groups, we tabulate the sectors below.

\input{spin18-table}

\end{appendix}

\clearpage

\listoftables

\bigskip

\noindent\textsc{School of Mathematical Sciences, University of Southampton, Highfield, Southampton SO17 1BJ, United Kingdom}

\smallskip
\noindent\url{dm7g15@soton.ac.uk}\\
\url{G.A.Niblo@soton.ac.uk}\\
\url{r.j.plymen@soton.ac.uk}\\
\url{wright@soton.ac.uk}

\vskip 1in

\noindent {\bf Keywords:} semisimple p-adic groups, Langlands dual, tempered spectrum, Iwahori-spherical block, extended quotient, $K$-theory, ABPS conjecture, $\Spin(18), \SO(18), \PSO(18)$

\medskip

\noindent {\bf MSC2020:} 22E50, 19L47, 55P62

\end{document}

%% file: spin18-table.tex
\def\tableheaderspinXVIII{Sectors, cohomology and $K$-theory for $\Spin(18)$ and $\PSO(18)$}

\begin{sidewaystable}[p]
\centering
\scriptsize
\caption{\tableheaderspinXVIII\@. Part 1 of 5.}
\label{tab:bipartitions-9-dn}
\begin{tabular}{C{2.65cm}C{0.85cm}C{3.0cm}C{3.0cm}C{2.25cm}C{2.25cm}C{1.55cm}C{1.55cm}}
\toprule
\makecell{\textbf{Bi-}\\\textbf{partition}} & \makecell{\textbf{Type}} & \makecell{\textbf{Spin}\\\textbf{sector}} & \makecell{\textbf{PSO}\\\textbf{sector}} & \makecell{\textbf{Spin}\\\textbf{cohom.}} & \makecell{\textbf{PSO}\\\textbf{cohom.}} & \makecell{\textbf{Spin}\\\textbf{K-theory}} & \makecell{\textbf{PSO}\\\textbf{K-theory}} \\
\midrule
\((\bar{8}, \bar{1})\) & \(\mathrm{IB}\) & \(4 \times V(0,0)\) & \(4 \times V(0,0)\) & \(\mathbb{Z}^{4}\) & \(\mathbb{Z}^{4}\) & \(\mathbb{C}^{4} \oplus \mathbb{C}^{0}\) & \(\mathbb{C}^{4} \oplus \mathbb{C}^{0}\) \\
\((\bar{7}, \bar{2})\) & \(\mathrm{IB}\) & \(4 \times V(0,0)\) & \(4 \times V(0,0)\) & \(\mathbb{Z}^{4}\) & \(\mathbb{Z}^{4}\) & \(\mathbb{C}^{4} \oplus \mathbb{C}^{0}\) & \(\mathbb{C}^{4} \oplus \mathbb{C}^{0}\) \\
\((\bar{6}, \bar{3})\) & \(\mathrm{IB}\) & \(4 \times V(0,0)\) & \(4 \times V(0,0)\) & \(\mathbb{Z}^{4}\) & \(\mathbb{Z}^{4}\) & \(\mathbb{C}^{4} \oplus \mathbb{C}^{0}\) & \(\mathbb{C}^{4} \oplus \mathbb{C}^{0}\) \\
\((\bar{6}, \bar{1}^{3})\) & \(\mathrm{IB}\) & \(6 \times V(0,0)\) & \(6 \times V(0,0)\) & \(\mathbb{Z}^{6}\) & \(\mathbb{Z}^{6}\) & \(\mathbb{C}^{6} \oplus \mathbb{C}^{0}\) & \(\mathbb{C}^{6} \oplus \mathbb{C}^{0}\) \\
\((\bar{5}, \bar{4})\) & \(\mathrm{IB}\) & \(4 \times V(0,0)\) & \(4 \times V(0,0)\) & \(\mathbb{Z}^{4}\) & \(\mathbb{Z}^{4}\) & \(\mathbb{C}^{4} \oplus \mathbb{C}^{0}\) & \(\mathbb{C}^{4} \oplus \mathbb{C}^{0}\) \\
\((\bar{5}, \bar{2}, \bar{1}^{2})\) & \(\mathrm{IB}\) & \(8 \times V(0,0)\) & \(8 \times V(0,0)\) & \(\mathbb{Z}^{8}\) & \(\mathbb{Z}^{8}\) & \(\mathbb{C}^{8} \oplus \mathbb{C}^{0}\) & \(\mathbb{C}^{8} \oplus \mathbb{C}^{0}\) \\
\((\bar{4}, \bar{3}, \bar{1}^{2})\) & \(\mathrm{IB}\) & \(8 \times V(0,0)\) & \(8 \times V(0,0)\) & \(\mathbb{Z}^{8}\) & \(\mathbb{Z}^{8}\) & \(\mathbb{C}^{8} \oplus \mathbb{C}^{0}\) & \(\mathbb{C}^{8} \oplus \mathbb{C}^{0}\) \\
\((\bar{4}, \bar{2}^{2}, \bar{1})\) & \(\mathrm{IB}\) & \(8 \times V(0,0)\) & \(8 \times V(0,0)\) & \(\mathbb{Z}^{8}\) & \(\mathbb{Z}^{8}\) & \(\mathbb{C}^{8} \oplus \mathbb{C}^{0}\) & \(\mathbb{C}^{8} \oplus \mathbb{C}^{0}\) \\
\((\bar{4}, \bar{1}^{5})\) & \(\mathrm{IB}\) & \(8 \times V(0,0)\) & \(8 \times V(0,0)\) & \(\mathbb{Z}^{8}\) & \(\mathbb{Z}^{8}\) & \(\mathbb{C}^{8} \oplus \mathbb{C}^{0}\) & \(\mathbb{C}^{8} \oplus \mathbb{C}^{0}\) \\
\((\bar{3}^{2}, \bar{2}, \bar{1})\) & \(\mathrm{IB}\) & \(8 \times V(0,0)\) & \(8 \times V(0,0)\) & \(\mathbb{Z}^{8}\) & \(\mathbb{Z}^{8}\) & \(\mathbb{C}^{8} \oplus \mathbb{C}^{0}\) & \(\mathbb{C}^{8} \oplus \mathbb{C}^{0}\) \\
\((\bar{3}, \bar{2}^{3})\) & \(\mathrm{IB}\) & \(6 \times V(0,0)\) & \(6 \times V(0,0)\) & \(\mathbb{Z}^{6}\) & \(\mathbb{Z}^{6}\) & \(\mathbb{C}^{6} \oplus \mathbb{C}^{0}\) & \(\mathbb{C}^{6} \oplus \mathbb{C}^{0}\) \\
\((\bar{3}, \bar{2}, \bar{1}^{4})\) & \(\mathrm{IB}\) & \(12 \times V(0,0)\) & \(12 \times V(0,0)\) & \(\mathbb{Z}^{12}\) & \(\mathbb{Z}^{12}\) & \(\mathbb{C}^{12} \oplus \mathbb{C}^{0}\) & \(\mathbb{C}^{12} \oplus \mathbb{C}^{0}\) \\
\((\bar{2}^{3}, \bar{1}^{3})\) & \(\mathrm{IB}\) & \(10 \times V(0,0)\) & \(10 \times V(0,0)\) & \(\mathbb{Z}^{10}\) & \(\mathbb{Z}^{10}\) & \(\mathbb{C}^{10} \oplus \mathbb{C}^{0}\) & \(\mathbb{C}^{10} \oplus \mathbb{C}^{0}\) \\
\((\bar{2}, \bar{1}^{7})\) & \(\mathrm{IB}\) & \(10 \times V(0,0)\) & \(10 \times V(0,0)\) & \(\mathbb{Z}^{10}\) & \(\mathbb{Z}^{10}\) & \(\mathbb{C}^{10} \oplus \mathbb{C}^{0}\) & \(\mathbb{C}^{10} \oplus \mathbb{C}^{0}\) \\
\((2, \bar{6}, \bar{1})\) & \(\mathrm{IB}\) & \(4 \times V(1,0)\) & \(4 \times V(1,0)\) & \(\mathbb{Z}^{4}\) & \(\mathbb{Z}^{4}\) & \(\mathbb{C}^{4} \oplus \mathbb{C}^{0}\) & \(\mathbb{C}^{4} \oplus \mathbb{C}^{0}\) \\
\((2, \bar{5}, \bar{2})\) & \(\mathrm{IB}\) & \(4 \times V(1,0)\) & \(4 \times V(1,0)\) & \(\mathbb{Z}^{4}\) & \(\mathbb{Z}^{4}\) & \(\mathbb{C}^{4} \oplus \mathbb{C}^{0}\) & \(\mathbb{C}^{4} \oplus \mathbb{C}^{0}\) \\
\((2, \bar{4}, \bar{3})\) & \(\mathrm{IB}\) & \(4 \times V(1,0)\) & \(4 \times V(1,0)\) & \(\mathbb{Z}^{4}\) & \(\mathbb{Z}^{4}\) & \(\mathbb{C}^{4} \oplus \mathbb{C}^{0}\) & \(\mathbb{C}^{4} \oplus \mathbb{C}^{0}\) \\
\((2, \bar{4}, \bar{1}^{3})\) & \(\mathrm{IB}\) & \(6 \times V(1,0)\) & \(6 \times V(1,0)\) & \(\mathbb{Z}^{6}\) & \(\mathbb{Z}^{6}\) & \(\mathbb{C}^{6} \oplus \mathbb{C}^{0}\) & \(\mathbb{C}^{6} \oplus \mathbb{C}^{0}\) \\
\((2, \bar{3}, \bar{2}, \bar{1}^{2})\) & \(\mathrm{IB}\) & \(8 \times V(1,0)\) & \(8 \times V(1,0)\) & \(\mathbb{Z}^{8}\) & \(\mathbb{Z}^{8}\) & \(\mathbb{C}^{8} \oplus \mathbb{C}^{0}\) & \(\mathbb{C}^{8} \oplus \mathbb{C}^{0}\) \\
\((2, \bar{2}^{3}, \bar{1})\) & \(\mathrm{IB}\) & \(6 \times V(1,0)\) & \(6 \times V(1,0)\) & \(\mathbb{Z}^{6}\) & \(\mathbb{Z}^{6}\) & \(\mathbb{C}^{6} \oplus \mathbb{C}^{0}\) & \(\mathbb{C}^{6} \oplus \mathbb{C}^{0}\) \\
\((2, \bar{2}, \bar{1}^{5})\) & \(\mathrm{IB}\) & \(8 \times V(1,0)\) & \(8 \times V(1,0)\) & \(\mathbb{Z}^{8}\) & \(\mathbb{Z}^{8}\) & \(\mathbb{C}^{8} \oplus \mathbb{C}^{0}\) & \(\mathbb{C}^{8} \oplus \mathbb{C}^{0}\) \\
\((4, \bar{4}, \bar{1})\) & \(\mathrm{IB}\) & \(4 \times V(1,0)\) & \(4 \times V(1,0)\) & \(\mathbb{Z}^{4}\) & \(\mathbb{Z}^{4}\) & \(\mathbb{C}^{4} \oplus \mathbb{C}^{0}\) & \(\mathbb{C}^{4} \oplus \mathbb{C}^{0}\) \\
\((4, \bar{3}, \bar{2})\) & \(\mathrm{IB}\) & \(4 \times V(1,0)\) & \(4 \times V(1,0)\) & \(\mathbb{Z}^{4}\) & \(\mathbb{Z}^{4}\) & \(\mathbb{C}^{4} \oplus \mathbb{C}^{0}\) & \(\mathbb{C}^{4} \oplus \mathbb{C}^{0}\) \\
\((4, \bar{2}, \bar{1}^{3})\) & \(\mathrm{IB}\) & \(6 \times V(1,0)\) & \(6 \times V(1,0)\) & \(\mathbb{Z}^{6}\) & \(\mathbb{Z}^{6}\) & \(\mathbb{C}^{6} \oplus \mathbb{C}^{0}\) & \(\mathbb{C}^{6} \oplus \mathbb{C}^{0}\) \\
\((2^{2}, \bar{4}, \bar{1})\) & \(\mathrm{IB}\) & \(4 \times V(2,0)\) & \(4 \times V(2,0)\) & \(\mathbb{Z}^{4}\) & \(\mathbb{Z}^{4}\) & \(\mathbb{C}^{4} \oplus \mathbb{C}^{0}\) & \(\mathbb{C}^{4} \oplus \mathbb{C}^{0}\) \\
\((2^{2}, \bar{3}, \bar{2})\) & \(\mathrm{IB}\) & \(4 \times V(2,0)\) & \(4 \times V(2,0)\) & \(\mathbb{Z}^{4}\) & \(\mathbb{Z}^{4}\) & \(\mathbb{C}^{4} \oplus \mathbb{C}^{0}\) & \(\mathbb{C}^{4} \oplus \mathbb{C}^{0}\) \\
\((2^{2}, \bar{2}, \bar{1}^{3})\) & \(\mathrm{IB}\) & \(6 \times V(2,0)\) & \(6 \times V(2,0)\) & \(\mathbb{Z}^{6}\) & \(\mathbb{Z}^{6}\) & \(\mathbb{C}^{6} \oplus \mathbb{C}^{0}\) & \(\mathbb{C}^{6} \oplus \mathbb{C}^{0}\) \\
\((6, \bar{2}, \bar{1})\) & \(\mathrm{IB}\) & \(4 \times V(1,0)\) & \(4 \times V(1,0)\) & \(\mathbb{Z}^{4}\) & \(\mathbb{Z}^{4}\) & \(\mathbb{C}^{4} \oplus \mathbb{C}^{0}\) & \(\mathbb{C}^{4} \oplus \mathbb{C}^{0}\) \\
\((4, 2, \bar{2}, \bar{1})\) & \(\mathrm{IB}\) & \(4 \times V(2,0)\) & \(4 \times V(2,0)\) & \(\mathbb{Z}^{4}\) & \(\mathbb{Z}^{4}\) & \(\mathbb{C}^{4} \oplus \mathbb{C}^{0}\) & \(\mathbb{C}^{4} \oplus \mathbb{C}^{0}\) \\
\((2^{3}, \bar{2}, \bar{1})\) & \(\mathrm{IB}\) & \(4 \times V(3,0)\) & \(4 \times V(3,0)\) & \(\mathbb{Z}^{4}\) & \(\mathbb{Z}^{4}\) & \(\mathbb{C}^{4} \oplus \mathbb{C}^{0}\) & \(\mathbb{C}^{4} \oplus \mathbb{C}^{0}\) \\
\((9)\) & \(\mathrm{IIA}\) & \(V(0,0) \times S^{1}\) & \(\unitbundle{1}{0}{0}\) & \(\mathbb{Z} \oplus \mathbb{Z}\) & \(\mathbb{Z} \oplus \mathbb{Z}\) & \(\mathbb{C}^{1} \oplus \mathbb{C}^{1}\) & \(\mathbb{C}^{1} \oplus \mathbb{C}^{1}\) \\
\((8, 1)\) & \(\mathrm{IIA}\) & \(V(1,0) \times S^{1}\) & \(\unitbundle{1}{0}{1}\) & \(\mathbb{Z} \oplus \mathbb{Z}\) & \(\mathbb{Z} \oplus \mathbb{Z}\) & \(\mathbb{C}^{1} \oplus \mathbb{C}^{1}\) & \(\mathbb{C}^{1} \oplus \mathbb{C}^{1}\) \\
\bottomrule
\end{tabular}
\end{sidewaystable}

\begin{sidewaystable}[p]
\centering
\scriptsize
\caption{\tableheaderspinXVIII\@. Part 2 of 5.}
\label{tab:bipartitions-9-dn-2}
\begin{tabular}{C{2.65cm}C{0.85cm}C{3.0cm}C{3.0cm}C{2.25cm}C{2.25cm}C{1.55cm}C{1.55cm}}
\toprule
\makecell{\textbf{Bi-}\\\textbf{partition}} & \makecell{\textbf{Type}} & \makecell{\textbf{Spin}\\\textbf{sector}} & \makecell{\textbf{PSO}\\\textbf{sector}} & \makecell{\textbf{Spin}\\\textbf{cohom.}} & \makecell{\textbf{PSO}\\\textbf{cohom.}} & \makecell{\textbf{Spin}\\\textbf{K-theory}} & \makecell{\textbf{PSO}\\\textbf{K-theory}} \\
\midrule
\((7, 2)\) & \(\mathrm{IIA}\) & \(V(1,0) \times S^{1}\) & \(\unitbundle{1}{0}{1}\) & \(\mathbb{Z} \oplus \mathbb{Z}\) & \(\mathbb{Z} \oplus \mathbb{Z}\) & \(\mathbb{C}^{1} \oplus \mathbb{C}^{1}\) & \(\mathbb{C}^{1} \oplus \mathbb{C}^{1}\) \\
\((6, 3)\) & \(\mathrm{IIA}\) & \(V(1,0) \times S^{1}\) & \(\unitbundle{1}{0}{1}\) & \(\mathbb{Z} \oplus \mathbb{Z}\) & \(\mathbb{Z} \oplus \mathbb{Z}\) & \(\mathbb{C}^{1} \oplus \mathbb{C}^{1}\) & \(\mathbb{C}^{1} \oplus \mathbb{C}^{1}\) \\
\((6, 2, 1)\) & \(\mathrm{IIA}\) & \(V(2,0) \times S^{1}\) & \(\unitbundle{1}{0}{2}\) & \(\mathbb{Z} \oplus \mathbb{Z}\) & \(\mathbb{Z} \oplus \mathbb{Z}\) & \(\mathbb{C}^{1} \oplus \mathbb{C}^{1}\) & \(\mathbb{C}^{1} \oplus \mathbb{C}^{1}\) \\
\((5, 4)\) & \(\mathrm{IIA}\) & \(V(1,0) \times S^{1}\) & \(\unitbundle{1}{0}{1}\) & \(\mathbb{Z} \oplus \mathbb{Z}\) & \(\mathbb{Z} \oplus \mathbb{Z}\) & \(\mathbb{C}^{1} \oplus \mathbb{C}^{1}\) & \(\mathbb{C}^{1} \oplus \mathbb{C}^{1}\) \\
\((5, 3, 1)\) & \(\mathrm{IIA}\) & \(V(0,0) \times S^{3}\) & \(\unitbundle{3}{0}{0}\) & \(\mathbb{Z} \oplus 0 \oplus 0 \oplus \mathbb{Z}\) & \(\mathbb{Z} \oplus 0 \oplus \mathbb{Z}/2 \oplus \mathbb{Z}\) & \(\mathbb{C}^{1} \oplus \mathbb{C}^{1}\) & \(\mathbb{C}^{1} \oplus \mathbb{C}^{1}\) \\
\((5, 2^{2})\) & \(\mathrm{IIA}\) & \(V(2,0) \times S^{1}\) & \(\unitbundle{1}{1}{1}\) & \(\mathbb{Z} \oplus \mathbb{Z}\) & \(\mathbb{Z} \oplus \mathbb{Z}\) & \(\mathbb{C}^{1} \oplus \mathbb{C}^{1}\) & \(\mathbb{C}^{1} \oplus \mathbb{C}^{1}\) \\
\((4^{2}, 1)\) & \(\mathrm{IIA}\) & \(V(2,0) \times S^{1}\) & \(\unitbundle{1}{1}{1}\) & \(\mathbb{Z} \oplus \mathbb{Z}\) & \(\mathbb{Z} \oplus \mathbb{Z}\) & \(\mathbb{C}^{1} \oplus \mathbb{C}^{1}\) & \(\mathbb{C}^{1} \oplus \mathbb{C}^{1}\) \\
\((4, 3, 2)\) & \(\mathrm{IIA}\) & \(V(2,0) \times S^{1}\) & \(\unitbundle{1}{0}{2}\) & \(\mathbb{Z} \oplus \mathbb{Z}\) & \(\mathbb{Z} \oplus \mathbb{Z}\) & \(\mathbb{C}^{1} \oplus \mathbb{C}^{1}\) & \(\mathbb{C}^{1} \oplus \mathbb{C}^{1}\) \\
\((4, 2^{2}, 1)\) & \(\mathrm{IIA}\) & \(V(3,0) \times S^{1}\) & \(\unitbundle{1}{1}{2}\) & \(\mathbb{Z} \oplus \mathbb{Z}\) & \(\mathbb{Z} \oplus \mathbb{Z}\) & \(\mathbb{C}^{1} \oplus \mathbb{C}^{1}\) & \(\mathbb{C}^{1} \oplus \mathbb{C}^{1}\) \\
\((3, 2^{3})\) & \(\mathrm{IIA}\) & \(V(3,0) \times S^{1}\) & \(\unitbundle{1}{1}{2}\) & \(\mathbb{Z} \oplus \mathbb{Z}\) & \(\mathbb{Z} \oplus \mathbb{Z}\) & \(\mathbb{C}^{1} \oplus \mathbb{C}^{1}\) & \(\mathbb{C}^{1} \oplus \mathbb{C}^{1}\) \\
\((2^{4}, 1)\) & \(\mathrm{IIA}\) & \(V(4,0) \times S^{1}\) & \(\unitbundle{1}{2}{2}\) & \(\mathbb{Z} \oplus \mathbb{Z}\) & \(\mathbb{Z} \oplus \mathbb{Z}\) & \(\mathbb{C}^{1} \oplus \mathbb{C}^{1}\) & \(\mathbb{C}^{1} \oplus \mathbb{C}^{1}\) \\
\((7, 1^{2})\) & \(\mathrm{IIA}'\) & \(V(3,0)\) & \(V(1,2)\) & \(\mathbb{Z}\) & \(\mathbb{Z}\) & \(\mathbb{C}^{1} \oplus \mathbb{C}^{0}\) & \(\mathbb{C}^{1} \oplus \mathbb{C}^{0}\) \\
\((6, 1^{3})\) & \(\mathrm{IIA}'\) & \(V(4,0)\) & \(V(1,3)\) & \(\mathbb{Z}\) & \(\mathbb{Z}\) & \(\mathbb{C}^{1} \oplus \mathbb{C}^{0}\) & \(\mathbb{C}^{1} \oplus \mathbb{C}^{0}\) \\
\((5, 2, 1^{2})\) & \(\mathrm{IIA}'\) & \(V(4,0)\) & \(V(1,3)\) & \(\mathbb{Z}\) & \(\mathbb{Z}\) & \(\mathbb{C}^{1} \oplus \mathbb{C}^{0}\) & \(\mathbb{C}^{1} \oplus \mathbb{C}^{0}\) \\
\((5, 1^{4})\) & \(\mathrm{IIA}'\) & \(V(5,0)\) & \(V(2,3)\) & \(\mathbb{Z}\) & \(\mathbb{Z}\) & \(\mathbb{C}^{1} \oplus \mathbb{C}^{0}\) & \(\mathbb{C}^{1} \oplus \mathbb{C}^{0}\) \\
\((4, 3, 1^{2})\) & \(\mathrm{IIA}'\) & \(V(4,0)\) & \(V(1,3)\) & \(\mathbb{Z}\) & \(\mathbb{Z}\) & \(\mathbb{C}^{1} \oplus \mathbb{C}^{0}\) & \(\mathbb{C}^{1} \oplus \mathbb{C}^{0}\) \\
\((4, 2, 1^{3})\) & \(\mathrm{IIA}'\) & \(V(5,0)\) & \(V(1,4)\) & \(\mathbb{Z}\) & \(\mathbb{Z}\) & \(\mathbb{C}^{1} \oplus \mathbb{C}^{0}\) & \(\mathbb{C}^{1} \oplus \mathbb{C}^{0}\) \\
\((4, 1^{5})\) & \(\mathrm{IIA}'\) & \(V(6,0)\) & \(V(2,4)\) & \(\mathbb{Z}\) & \(\mathbb{Z}\) & \(\mathbb{C}^{1} \oplus \mathbb{C}^{0}\) & \(\mathbb{C}^{1} \oplus \mathbb{C}^{0}\) \\
\((3^{3})\) & \(\mathrm{IIA}'\) & \(V(3,0)\) & \(V(1,2)\) & \(\mathbb{Z}\) & \(\mathbb{Z}\) & \(\mathbb{C}^{1} \oplus \mathbb{C}^{0}\) & \(\mathbb{C}^{1} \oplus \mathbb{C}^{0}\) \\
\((3^{2}, 2, 1)\) & \(\mathrm{IIA}'\) & \(V(4,0)\) & \(V(1,3)\) & \(\mathbb{Z}\) & \(\mathbb{Z}\) & \(\mathbb{C}^{1} \oplus \mathbb{C}^{0}\) & \(\mathbb{C}^{1} \oplus \mathbb{C}^{0}\) \\
\((3^{2}, 1^{3})\) & \(\mathrm{IIA}'\) & \(V(5,0)\) & \(V(2,3)\) & \(\mathbb{Z}\) & \(\mathbb{Z}\) & \(\mathbb{C}^{1} \oplus \mathbb{C}^{0}\) & \(\mathbb{C}^{1} \oplus \mathbb{C}^{0}\) \\
\((3, 2^{2}, 1^{2})\) & \(\mathrm{IIA}'\) & \(V(5,0)\) & \(V(2,3)\) & \(\mathbb{Z}\) & \(\mathbb{Z}\) & \(\mathbb{C}^{1} \oplus \mathbb{C}^{0}\) & \(\mathbb{C}^{1} \oplus \mathbb{C}^{0}\) \\
\((3, 2, 1^{4})\) & \(\mathrm{IIA}'\) & \(V(6,0)\) & \(V(2,4)\) & \(\mathbb{Z}\) & \(\mathbb{Z}\) & \(\mathbb{C}^{1} \oplus \mathbb{C}^{0}\) & \(\mathbb{C}^{1} \oplus \mathbb{C}^{0}\) \\
\((3, 1^{6})\) & \(\mathrm{IIA}'\) & \(V(7,0)\) & \(V(3,4)\) & \(\mathbb{Z}\) & \(\mathbb{Z}\) & \(\mathbb{C}^{1} \oplus \mathbb{C}^{0}\) & \(\mathbb{C}^{1} \oplus \mathbb{C}^{0}\) \\
\((2^{3}, 1^{3})\) & \(\mathrm{IIA}'\) & \(V(6,0)\) & \(V(2,4)\) & \(\mathbb{Z}\) & \(\mathbb{Z}\) & \(\mathbb{C}^{1} \oplus \mathbb{C}^{0}\) & \(\mathbb{C}^{1} \oplus \mathbb{C}^{0}\) \\
\((2^{2}, 1^{5})\) & \(\mathrm{IIA}'\) & \(V(7,0)\) & \(V(3,4)\) & \(\mathbb{Z}\) & \(\mathbb{Z}\) & \(\mathbb{C}^{1} \oplus \mathbb{C}^{0}\) & \(\mathbb{C}^{1} \oplus \mathbb{C}^{0}\) \\
\((2, 1^{7})\) & \(\mathrm{IIA}'\) & \(V(8,0)\) & \(V(3,5)\) & \(\mathbb{Z}\) & \(\mathbb{Z}\) & \(\mathbb{C}^{1} \oplus \mathbb{C}^{0}\) & \(\mathbb{C}^{1} \oplus \mathbb{C}^{0}\) \\
\((1^{9})\) & \(\mathrm{IIA}'\) & \(V(9,0)\) & \(V(4,5)\) & \(\mathbb{Z}\) & \(\mathbb{Z}\) & \(\mathbb{C}^{1} \oplus \mathbb{C}^{0}\) & \(\mathbb{C}^{1} \oplus \mathbb{C}^{0}\) \\
\((1, \bar{7}, \bar{1})\) & \(\mathrm{IIB}\) & \(2 \times V(1,0)\) & \(2 \times V(1,0)\) & \(\mathbb{Z}^{2}\) & \(\mathbb{Z}^{2}\) & \(\mathbb{C}^{2} \oplus \mathbb{C}^{0}\) & \(\mathbb{C}^{2} \oplus \mathbb{C}^{0}\) \\
\((1, \bar{6}, \bar{2})\) & \(\mathrm{IIB}\) & \(2 \times V(1,0)\) & \(2 \times V(1,0)\) & \(\mathbb{Z}^{2}\) & \(\mathbb{Z}^{2}\) & \(\mathbb{C}^{2} \oplus \mathbb{C}^{0}\) & \(\mathbb{C}^{2} \oplus \mathbb{C}^{0}\) \\
\((1, \bar{5}, \bar{3})\) & \(\mathrm{IIB}\) & \(2 \times V(1,0)\) & \(2 \times V(1,0)\) & \(\mathbb{Z}^{2}\) & \(\mathbb{Z}^{2}\) & \(\mathbb{C}^{2} \oplus \mathbb{C}^{0}\) & \(\mathbb{C}^{2} \oplus \mathbb{C}^{0}\) \\
\((1, \bar{5}, \bar{1}^{3})\) & \(\mathrm{IIB}\) & \(4 \times V(1,0)\) & \(4 \times V(1,0)\) & \(\mathbb{Z}^{4}\) & \(\mathbb{Z}^{4}\) & \(\mathbb{C}^{4} \oplus \mathbb{C}^{0}\) & \(\mathbb{C}^{4} \oplus \mathbb{C}^{0}\) \\
\bottomrule
\end{tabular}
\smallskip

\begin{minipage}{0.8\textwidth}
{\strut Sectors labelled $\mathrm{IIA}'$ are of type IIA with the additional property that there is some odd positive cycle with multiplicity greater than $1$.}
\end{minipage}
\end{sidewaystable}

\begin{sidewaystable}[p]
\centering
\scriptsize
\caption{\tableheaderspinXVIII\@. Part 3 of 5.}
\label{tab:bipartitions-9-dn-3}
\begin{tabular}{C{2.65cm}C{0.85cm}C{3.0cm}C{3.0cm}C{2.25cm}C{2.25cm}C{1.55cm}C{1.55cm}}
\toprule
\makecell{\textbf{Bi-}\\\textbf{partition}} & \makecell{\textbf{Type}} & \makecell{\textbf{Spin}\\\textbf{sector}} & \makecell{\textbf{PSO}\\\textbf{sector}} & \makecell{\textbf{Spin}\\\textbf{cohom.}} & \makecell{\textbf{PSO}\\\textbf{cohom.}} & \makecell{\textbf{Spin}\\\textbf{K-theory}} & \makecell{\textbf{PSO}\\\textbf{K-theory}} \\
\midrule
\((1, \bar{4}^{2})\) & \(\mathrm{IIB}\) & \(2 \times V(1,0)\) & \(V(1,0) \sqcup V(0,1)\) & \(\mathbb{Z}^{2}\) & \(\mathbb{Z}^{2}\) & \(\mathbb{C}^{2} \oplus \mathbb{C}^{0}\) & \(\mathbb{C}^{2} \oplus \mathbb{C}^{0}\) \\
\((1, \bar{4}, \bar{2}, \bar{1}^{2})\) & \(\mathrm{IIB}\) & \(6 \times V(1,0)\) & \(6 \times V(1,0)\) & \(\mathbb{Z}^{6}\) & \(\mathbb{Z}^{6}\) & \(\mathbb{C}^{6} \oplus \mathbb{C}^{0}\) & \(\mathbb{C}^{6} \oplus \mathbb{C}^{0}\) \\
\((1, \bar{3}^{2}, \bar{1}^{2})\) & \(\mathrm{IIB}\) & \(5 \times V(1,0)\) & \(4 \times V(1,0) \sqcup V(0,1)\) & \(\mathbb{Z}^{5}\) & \(\mathbb{Z}^{5}\) & \(\mathbb{C}^{5} \oplus \mathbb{C}^{0}\) & \(\mathbb{C}^{5} \oplus \mathbb{C}^{0}\) \\
\((1, \bar{3}, \bar{2}^{2}, \bar{1})\) & \(\mathrm{IIB}\) & \(6 \times V(1,0)\) & \(6 \times V(1,0)\) & \(\mathbb{Z}^{6}\) & \(\mathbb{Z}^{6}\) & \(\mathbb{C}^{6} \oplus \mathbb{C}^{0}\) & \(\mathbb{C}^{6} \oplus \mathbb{C}^{0}\) \\
\((1, \bar{3}, \bar{1}^{5})\) & \(\mathrm{IIB}\) & \(6 \times V(1,0)\) & \(6 \times V(1,0)\) & \(\mathbb{Z}^{6}\) & \(\mathbb{Z}^{6}\) & \(\mathbb{C}^{6} \oplus \mathbb{C}^{0}\) & \(\mathbb{C}^{6} \oplus \mathbb{C}^{0}\) \\
\((1, \bar{2}^{4})\) & \(\mathrm{IIB}\) & \(3 \times V(1,0)\) & \(2 \times V(1,0) \sqcup V(0,1)\) & \(\mathbb{Z}^{3}\) & \(\mathbb{Z}^{3}\) & \(\mathbb{C}^{3} \oplus \mathbb{C}^{0}\) & \(\mathbb{C}^{3} \oplus \mathbb{C}^{0}\) \\
\((1, \bar{2}^{2}, \bar{1}^{4})\) & \(\mathrm{IIB}\) & \(8 \times V(1,0)\) & \(7 \times V(1,0) \sqcup V(0,1)\) & \(\mathbb{Z}^{8}\) & \(\mathbb{Z}^{8}\) & \(\mathbb{C}^{8} \oplus \mathbb{C}^{0}\) & \(\mathbb{C}^{8} \oplus \mathbb{C}^{0}\) \\
\((1, \bar{1}^{8})\) & \(\mathrm{IIB}\) & \(5 \times V(1,0)\) & \(4 \times V(1,0) \sqcup V(0,1)\) & \(\mathbb{Z}^{5}\) & \(\mathbb{Z}^{5}\) & \(\mathbb{C}^{5} \oplus \mathbb{C}^{0}\) & \(\mathbb{C}^{5} \oplus \mathbb{C}^{0}\) \\
\((1^{2}, \bar{6}, \bar{1})\) & \(\mathrm{IIB}\) & \(2 \times V(2,0)\) & \(2 \times V(2,0)\) & \(\mathbb{Z}^{2}\) & \(\mathbb{Z}^{2}\) & \(\mathbb{C}^{2} \oplus \mathbb{C}^{0}\) & \(\mathbb{C}^{2} \oplus \mathbb{C}^{0}\) \\
\((1^{2}, \bar{5}, \bar{2})\) & \(\mathrm{IIB}\) & \(2 \times V(2,0)\) & \(2 \times V(2,0)\) & \(\mathbb{Z}^{2}\) & \(\mathbb{Z}^{2}\) & \(\mathbb{C}^{2} \oplus \mathbb{C}^{0}\) & \(\mathbb{C}^{2} \oplus \mathbb{C}^{0}\) \\
\((1^{2}, \bar{4}, \bar{3})\) & \(\mathrm{IIB}\) & \(2 \times V(2,0)\) & \(2 \times V(2,0)\) & \(\mathbb{Z}^{2}\) & \(\mathbb{Z}^{2}\) & \(\mathbb{C}^{2} \oplus \mathbb{C}^{0}\) & \(\mathbb{C}^{2} \oplus \mathbb{C}^{0}\) \\
\((1^{2}, \bar{4}, \bar{1}^{3})\) & \(\mathrm{IIB}\) & \(4 \times V(2,0)\) & \(4 \times V(2,0)\) & \(\mathbb{Z}^{4}\) & \(\mathbb{Z}^{4}\) & \(\mathbb{C}^{4} \oplus \mathbb{C}^{0}\) & \(\mathbb{C}^{4} \oplus \mathbb{C}^{0}\) \\
\((1^{2}, \bar{3}, \bar{2}, \bar{1}^{2})\) & \(\mathrm{IIB}\) & \(6 \times V(2,0)\) & \(6 \times V(2,0)\) & \(\mathbb{Z}^{6}\) & \(\mathbb{Z}^{6}\) & \(\mathbb{C}^{6} \oplus \mathbb{C}^{0}\) & \(\mathbb{C}^{6} \oplus \mathbb{C}^{0}\) \\
\((1^{2}, \bar{2}^{3}, \bar{1})\) & \(\mathrm{IIB}\) & \(4 \times V(2,0)\) & \(4 \times V(2,0)\) & \(\mathbb{Z}^{4}\) & \(\mathbb{Z}^{4}\) & \(\mathbb{C}^{4} \oplus \mathbb{C}^{0}\) & \(\mathbb{C}^{4} \oplus \mathbb{C}^{0}\) \\
\((1^{2}, \bar{2}, \bar{1}^{5})\) & \(\mathrm{IIB}\) & \(6 \times V(2,0)\) & \(6 \times V(2,0)\) & \(\mathbb{Z}^{6}\) & \(\mathbb{Z}^{6}\) & \(\mathbb{C}^{6} \oplus \mathbb{C}^{0}\) & \(\mathbb{C}^{6} \oplus \mathbb{C}^{0}\) \\
\((3, \bar{5}, \bar{1})\) & \(\mathrm{IIB}\) & \(2 \times V(1,0)\) & \(2 \times V(1,0)\) & \(\mathbb{Z}^{2}\) & \(\mathbb{Z}^{2}\) & \(\mathbb{C}^{2} \oplus \mathbb{C}^{0}\) & \(\mathbb{C}^{2} \oplus \mathbb{C}^{0}\) \\
\((3, \bar{4}, \bar{2})\) & \(\mathrm{IIB}\) & \(2 \times V(1,0)\) & \(2 \times V(1,0)\) & \(\mathbb{Z}^{2}\) & \(\mathbb{Z}^{2}\) & \(\mathbb{C}^{2} \oplus \mathbb{C}^{0}\) & \(\mathbb{C}^{2} \oplus \mathbb{C}^{0}\) \\
\((3, \bar{3}^{2})\) & \(\mathrm{IIB}\) & \(2 \times V(1,0)\) & \(V(1,0) \sqcup V(0,1)\) & \(\mathbb{Z}^{2}\) & \(\mathbb{Z}^{2}\) & \(\mathbb{C}^{2} \oplus \mathbb{C}^{0}\) & \(\mathbb{C}^{2} \oplus \mathbb{C}^{0}\) \\
\((3, \bar{3}, \bar{1}^{3})\) & \(\mathrm{IIB}\) & \(4 \times V(1,0)\) & \(4 \times V(1,0)\) & \(\mathbb{Z}^{4}\) & \(\mathbb{Z}^{4}\) & \(\mathbb{C}^{4} \oplus \mathbb{C}^{0}\) & \(\mathbb{C}^{4} \oplus \mathbb{C}^{0}\) \\
\((3, \bar{2}^{2}, \bar{1}^{2})\) & \(\mathrm{IIB}\) & \(5 \times V(1,0)\) & \(4 \times V(1,0) \sqcup V(0,1)\) & \(\mathbb{Z}^{5}\) & \(\mathbb{Z}^{5}\) & \(\mathbb{C}^{5} \oplus \mathbb{C}^{0}\) & \(\mathbb{C}^{5} \oplus \mathbb{C}^{0}\) \\
\((3, \bar{1}^{6})\) & \(\mathrm{IIB}\) & \(4 \times V(1,0)\) & \(3 \times V(1,0) \sqcup V(0,1)\) & \(\mathbb{Z}^{4}\) & \(\mathbb{Z}^{4}\) & \(\mathbb{C}^{4} \oplus \mathbb{C}^{0}\) & \(\mathbb{C}^{4} \oplus \mathbb{C}^{0}\) \\
\((2, 1, \bar{5}, \bar{1})\) & \(\mathrm{IIB}\) & \(2 \times V(2,0)\) & \(2 \times V(2,0)\) & \(\mathbb{Z}^{2}\) & \(\mathbb{Z}^{2}\) & \(\mathbb{C}^{2} \oplus \mathbb{C}^{0}\) & \(\mathbb{C}^{2} \oplus \mathbb{C}^{0}\) \\
\((2, 1, \bar{4}, \bar{2})\) & \(\mathrm{IIB}\) & \(2 \times V(2,0)\) & \(2 \times V(2,0)\) & \(\mathbb{Z}^{2}\) & \(\mathbb{Z}^{2}\) & \(\mathbb{C}^{2} \oplus \mathbb{C}^{0}\) & \(\mathbb{C}^{2} \oplus \mathbb{C}^{0}\) \\
\((2, 1, \bar{3}^{2})\) & \(\mathrm{IIB}\) & \(2 \times V(2,0)\) & \(V(2,0) \sqcup V(0,2)\) & \(\mathbb{Z}^{2}\) & \(\mathbb{Z}^{2}\) & \(\mathbb{C}^{2} \oplus \mathbb{C}^{0}\) & \(\mathbb{C}^{2} \oplus \mathbb{C}^{0}\) \\
\((2, 1, \bar{3}, \bar{1}^{3})\) & \(\mathrm{IIB}\) & \(4 \times V(2,0)\) & \(4 \times V(2,0)\) & \(\mathbb{Z}^{4}\) & \(\mathbb{Z}^{4}\) & \(\mathbb{C}^{4} \oplus \mathbb{C}^{0}\) & \(\mathbb{C}^{4} \oplus \mathbb{C}^{0}\) \\
\((2, 1, \bar{2}^{2}, \bar{1}^{2})\) & \(\mathrm{IIB}\) & \(5 \times V(2,0)\) & \(4 \times V(2,0) \sqcup V(0,2)\) & \(\mathbb{Z}^{5}\) & \(\mathbb{Z}^{5}\) & \(\mathbb{C}^{5} \oplus \mathbb{C}^{0}\) & \(\mathbb{C}^{5} \oplus \mathbb{C}^{0}\) \\
\((2, 1, \bar{1}^{6})\) & \(\mathrm{IIB}\) & \(4 \times V(2,0)\) & \(3 \times V(2,0) \sqcup V(0,2)\) & \(\mathbb{Z}^{4}\) & \(\mathbb{Z}^{4}\) & \(\mathbb{C}^{4} \oplus \mathbb{C}^{0}\) & \(\mathbb{C}^{4} \oplus \mathbb{C}^{0}\) \\
\((1^{3}, \bar{5}, \bar{1})\) & \(\mathrm{IIB}\) & \(2 \times V(3,0)\) & \(2 \times V(3,0)\) & \(\mathbb{Z}^{2}\) & \(\mathbb{Z}^{2}\) & \(\mathbb{C}^{2} \oplus \mathbb{C}^{0}\) & \(\mathbb{C}^{2} \oplus \mathbb{C}^{0}\) \\
\((1^{3}, \bar{4}, \bar{2})\) & \(\mathrm{IIB}\) & \(2 \times V(3,0)\) & \(2 \times V(3,0)\) & \(\mathbb{Z}^{2}\) & \(\mathbb{Z}^{2}\) & \(\mathbb{C}^{2} \oplus \mathbb{C}^{0}\) & \(\mathbb{C}^{2} \oplus \mathbb{C}^{0}\) \\
\((1^{3}, \bar{3}^{2})\) & \(\mathrm{IIB}\) & \(2 \times V(3,0)\) & \(V(3,0) \sqcup V(1,2)\) & \(\mathbb{Z}^{2}\) & \(\mathbb{Z}^{2}\) & \(\mathbb{C}^{2} \oplus \mathbb{C}^{0}\) & \(\mathbb{C}^{2} \oplus \mathbb{C}^{0}\) \\
\((1^{3}, \bar{3}, \bar{1}^{3})\) & \(\mathrm{IIB}\) & \(4 \times V(3,0)\) & \(4 \times V(3,0)\) & \(\mathbb{Z}^{4}\) & \(\mathbb{Z}^{4}\) & \(\mathbb{C}^{4} \oplus \mathbb{C}^{0}\) & \(\mathbb{C}^{4} \oplus \mathbb{C}^{0}\) \\
\((1^{3}, \bar{2}^{2}, \bar{1}^{2})\) & \(\mathrm{IIB}\) & \(5 \times V(3,0)\) & \(4 \times V(3,0) \sqcup V(1,2)\) & \(\mathbb{Z}^{5}\) & \(\mathbb{Z}^{5}\) & \(\mathbb{C}^{5} \oplus \mathbb{C}^{0}\) & \(\mathbb{C}^{5} \oplus \mathbb{C}^{0}\) \\
\bottomrule
\end{tabular}
\end{sidewaystable}

\begin{sidewaystable}[p]
\centering
\scriptsize
\caption{\tableheaderspinXVIII\@. Part 4 of 5.}
\label{tab:bipartitions-9-dn-4}
\begin{tabular}{C{2.65cm}C{0.85cm}C{3.0cm}C{3.0cm}C{2.25cm}C{2.25cm}C{1.55cm}C{1.55cm}}
\toprule
\makecell{\textbf{Bi-}\\\textbf{partition}} & \makecell{\textbf{Type}} & \makecell{\textbf{Spin}\\\textbf{sector}} & \makecell{\textbf{PSO}\\\textbf{sector}} & \makecell{\textbf{Spin}\\\textbf{cohom.}} & \makecell{\textbf{PSO}\\\textbf{cohom.}} & \makecell{\textbf{Spin}\\\textbf{K-theory}} & \makecell{\textbf{PSO}\\\textbf{K-theory}} \\
\midrule
\((1^{3}, \bar{1}^{6})\) & \(\mathrm{IIB}\) & \(4 \times V(3,0)\) & \(3 \times V(3,0) \sqcup V(1,2)\) & \(\mathbb{Z}^{4}\) & \(\mathbb{Z}^{4}\) & \(\mathbb{C}^{4} \oplus \mathbb{C}^{0}\) & \(\mathbb{C}^{4} \oplus \mathbb{C}^{0}\) \\
\((3, 1, \bar{4}, \bar{1})\) & \(\mathrm{IIB}\) & \(2 \times V(2,0)\) & \(2 \times V(2,0)\) & \(\mathbb{Z}^{2}\) & \(\mathbb{Z}^{2}\) & \(\mathbb{C}^{2} \oplus \mathbb{C}^{0}\) & \(\mathbb{C}^{2} \oplus \mathbb{C}^{0}\) \\
\((3, 1, \bar{3}, \bar{2})\) & \(\mathrm{IIB}\) & \(2 \times V(2,0)\) & \(2 \times V(2,0)\) & \(\mathbb{Z}^{2}\) & \(\mathbb{Z}^{2}\) & \(\mathbb{C}^{2} \oplus \mathbb{C}^{0}\) & \(\mathbb{C}^{2} \oplus \mathbb{C}^{0}\) \\
\((3, 1, \bar{2}, \bar{1}^{3})\) & \(\mathrm{IIB}\) & \(4 \times V(2,0)\) & \(4 \times V(2,0)\) & \(\mathbb{Z}^{4}\) & \(\mathbb{Z}^{4}\) & \(\mathbb{C}^{4} \oplus \mathbb{C}^{0}\) & \(\mathbb{C}^{4} \oplus \mathbb{C}^{0}\) \\
\((2, 1^{2}, \bar{4}, \bar{1})\) & \(\mathrm{IIB}\) & \(2 \times V(3,0)\) & \(2 \times V(3,0)\) & \(\mathbb{Z}^{2}\) & \(\mathbb{Z}^{2}\) & \(\mathbb{C}^{2} \oplus \mathbb{C}^{0}\) & \(\mathbb{C}^{2} \oplus \mathbb{C}^{0}\) \\
\((2, 1^{2}, \bar{3}, \bar{2})\) & \(\mathrm{IIB}\) & \(2 \times V(3,0)\) & \(2 \times V(3,0)\) & \(\mathbb{Z}^{2}\) & \(\mathbb{Z}^{2}\) & \(\mathbb{C}^{2} \oplus \mathbb{C}^{0}\) & \(\mathbb{C}^{2} \oplus \mathbb{C}^{0}\) \\
\((2, 1^{2}, \bar{2}, \bar{1}^{3})\) & \(\mathrm{IIB}\) & \(4 \times V(3,0)\) & \(4 \times V(3,0)\) & \(\mathbb{Z}^{4}\) & \(\mathbb{Z}^{4}\) & \(\mathbb{C}^{4} \oplus \mathbb{C}^{0}\) & \(\mathbb{C}^{4} \oplus \mathbb{C}^{0}\) \\
\((1^{4}, \bar{4}, \bar{1})\) & \(\mathrm{IIB}\) & \(2 \times V(4,0)\) & \(2 \times V(4,0)\) & \(\mathbb{Z}^{2}\) & \(\mathbb{Z}^{2}\) & \(\mathbb{C}^{2} \oplus \mathbb{C}^{0}\) & \(\mathbb{C}^{2} \oplus \mathbb{C}^{0}\) \\
\((1^{4}, \bar{3}, \bar{2})\) & \(\mathrm{IIB}\) & \(2 \times V(4,0)\) & \(2 \times V(4,0)\) & \(\mathbb{Z}^{2}\) & \(\mathbb{Z}^{2}\) & \(\mathbb{C}^{2} \oplus \mathbb{C}^{0}\) & \(\mathbb{C}^{2} \oplus \mathbb{C}^{0}\) \\
\((1^{4}, \bar{2}, \bar{1}^{3})\) & \(\mathrm{IIB}\) & \(4 \times V(4,0)\) & \(4 \times V(4,0)\) & \(\mathbb{Z}^{4}\) & \(\mathbb{Z}^{4}\) & \(\mathbb{C}^{4} \oplus \mathbb{C}^{0}\) & \(\mathbb{C}^{4} \oplus \mathbb{C}^{0}\) \\
\((5, \bar{3}, \bar{1})\) & \(\mathrm{IIB}\) & \(2 \times V(1,0)\) & \(2 \times V(1,0)\) & \(\mathbb{Z}^{2}\) & \(\mathbb{Z}^{2}\) & \(\mathbb{C}^{2} \oplus \mathbb{C}^{0}\) & \(\mathbb{C}^{2} \oplus \mathbb{C}^{0}\) \\
\((5, \bar{2}^{2})\) & \(\mathrm{IIB}\) & \(2 \times V(1,0)\) & \(V(1,0) \sqcup V(0,1)\) & \(\mathbb{Z}^{2}\) & \(\mathbb{Z}^{2}\) & \(\mathbb{C}^{2} \oplus \mathbb{C}^{0}\) & \(\mathbb{C}^{2} \oplus \mathbb{C}^{0}\) \\
\((5, \bar{1}^{4})\) & \(\mathrm{IIB}\) & \(3 \times V(1,0)\) & \(2 \times V(1,0) \sqcup V(0,1)\) & \(\mathbb{Z}^{3}\) & \(\mathbb{Z}^{3}\) & \(\mathbb{C}^{3} \oplus \mathbb{C}^{0}\) & \(\mathbb{C}^{3} \oplus \mathbb{C}^{0}\) \\
\((4, 1, \bar{3}, \bar{1})\) & \(\mathrm{IIB}\) & \(2 \times V(2,0)\) & \(2 \times V(2,0)\) & \(\mathbb{Z}^{2}\) & \(\mathbb{Z}^{2}\) & \(\mathbb{C}^{2} \oplus \mathbb{C}^{0}\) & \(\mathbb{C}^{2} \oplus \mathbb{C}^{0}\) \\
\((4, 1, \bar{2}^{2})\) & \(\mathrm{IIB}\) & \(2 \times V(2,0)\) & \(V(2,0) \sqcup V(0,2)\) & \(\mathbb{Z}^{2}\) & \(\mathbb{Z}^{2}\) & \(\mathbb{C}^{2} \oplus \mathbb{C}^{0}\) & \(\mathbb{C}^{2} \oplus \mathbb{C}^{0}\) \\
\((4, 1, \bar{1}^{4})\) & \(\mathrm{IIB}\) & \(3 \times V(2,0)\) & \(2 \times V(2,0) \sqcup V(0,2)\) & \(\mathbb{Z}^{3}\) & \(\mathbb{Z}^{3}\) & \(\mathbb{C}^{3} \oplus \mathbb{C}^{0}\) & \(\mathbb{C}^{3} \oplus \mathbb{C}^{0}\) \\
\((3, 2, \bar{3}, \bar{1})\) & \(\mathrm{IIB}\) & \(2 \times V(2,0)\) & \(2 \times V(2,0)\) & \(\mathbb{Z}^{2}\) & \(\mathbb{Z}^{2}\) & \(\mathbb{C}^{2} \oplus \mathbb{C}^{0}\) & \(\mathbb{C}^{2} \oplus \mathbb{C}^{0}\) \\
\((3, 2, \bar{2}^{2})\) & \(\mathrm{IIB}\) & \(2 \times V(2,0)\) & \(V(2,0) \sqcup V(0,2)\) & \(\mathbb{Z}^{2}\) & \(\mathbb{Z}^{2}\) & \(\mathbb{C}^{2} \oplus \mathbb{C}^{0}\) & \(\mathbb{C}^{2} \oplus \mathbb{C}^{0}\) \\
\((3, 2, \bar{1}^{4})\) & \(\mathrm{IIB}\) & \(3 \times V(2,0)\) & \(2 \times V(2,0) \sqcup V(0,2)\) & \(\mathbb{Z}^{3}\) & \(\mathbb{Z}^{3}\) & \(\mathbb{C}^{3} \oplus \mathbb{C}^{0}\) & \(\mathbb{C}^{3} \oplus \mathbb{C}^{0}\) \\
\((3, 1^{2}, \bar{3}, \bar{1})\) & \(\mathrm{IIB}\) & \(2 \times V(3,0)\) & \(2 \times V(3,0)\) & \(\mathbb{Z}^{2}\) & \(\mathbb{Z}^{2}\) & \(\mathbb{C}^{2} \oplus \mathbb{C}^{0}\) & \(\mathbb{C}^{2} \oplus \mathbb{C}^{0}\) \\
\((3, 1^{2}, \bar{2}^{2})\) & \(\mathrm{IIB}\) & \(2 \times V(3,0)\) & \(V(3,0) \sqcup V(1,2)\) & \(\mathbb{Z}^{2}\) & \(\mathbb{Z}^{2}\) & \(\mathbb{C}^{2} \oplus \mathbb{C}^{0}\) & \(\mathbb{C}^{2} \oplus \mathbb{C}^{0}\) \\
\((3, 1^{2}, \bar{1}^{4})\) & \(\mathrm{IIB}\) & \(3 \times V(3,0)\) & \(2 \times V(3,0) \sqcup V(1,2)\) & \(\mathbb{Z}^{3}\) & \(\mathbb{Z}^{3}\) & \(\mathbb{C}^{3} \oplus \mathbb{C}^{0}\) & \(\mathbb{C}^{3} \oplus \mathbb{C}^{0}\) \\
\((2^{2}, 1, \bar{3}, \bar{1})\) & \(\mathrm{IIB}\) & \(2 \times V(3,0)\) & \(2 \times V(3,0)\) & \(\mathbb{Z}^{2}\) & \(\mathbb{Z}^{2}\) & \(\mathbb{C}^{2} \oplus \mathbb{C}^{0}\) & \(\mathbb{C}^{2} \oplus \mathbb{C}^{0}\) \\
\((2^{2}, 1, \bar{2}^{2})\) & \(\mathrm{IIB}\) & \(2 \times V(3,0)\) & \(V(3,0) \sqcup V(1,2)\) & \(\mathbb{Z}^{2}\) & \(\mathbb{Z}^{2}\) & \(\mathbb{C}^{2} \oplus \mathbb{C}^{0}\) & \(\mathbb{C}^{2} \oplus \mathbb{C}^{0}\) \\
\((2^{2}, 1, \bar{1}^{4})\) & \(\mathrm{IIB}\) & \(3 \times V(3,0)\) & \(2 \times V(3,0) \sqcup V(1,2)\) & \(\mathbb{Z}^{3}\) & \(\mathbb{Z}^{3}\) & \(\mathbb{C}^{3} \oplus \mathbb{C}^{0}\) & \(\mathbb{C}^{3} \oplus \mathbb{C}^{0}\) \\
\((2, 1^{3}, \bar{3}, \bar{1})\) & \(\mathrm{IIB}\) & \(2 \times V(4,0)\) & \(2 \times V(4,0)\) & \(\mathbb{Z}^{2}\) & \(\mathbb{Z}^{2}\) & \(\mathbb{C}^{2} \oplus \mathbb{C}^{0}\) & \(\mathbb{C}^{2} \oplus \mathbb{C}^{0}\) \\
\((2, 1^{3}, \bar{2}^{2})\) & \(\mathrm{IIB}\) & \(2 \times V(4,0)\) & \(V(4,0) \sqcup V(1,3)\) & \(\mathbb{Z}^{2}\) & \(\mathbb{Z}^{2}\) & \(\mathbb{C}^{2} \oplus \mathbb{C}^{0}\) & \(\mathbb{C}^{2} \oplus \mathbb{C}^{0}\) \\
\((2, 1^{3}, \bar{1}^{4})\) & \(\mathrm{IIB}\) & \(3 \times V(4,0)\) & \(2 \times V(4,0) \sqcup V(1,3)\) & \(\mathbb{Z}^{3}\) & \(\mathbb{Z}^{3}\) & \(\mathbb{C}^{3} \oplus \mathbb{C}^{0}\) & \(\mathbb{C}^{3} \oplus \mathbb{C}^{0}\) \\
\((1^{5}, \bar{3}, \bar{1})\) & \(\mathrm{IIB}\) & \(2 \times V(5,0)\) & \(2 \times V(5,0)\) & \(\mathbb{Z}^{2}\) & \(\mathbb{Z}^{2}\) & \(\mathbb{C}^{2} \oplus \mathbb{C}^{0}\) & \(\mathbb{C}^{2} \oplus \mathbb{C}^{0}\) \\
\((1^{5}, \bar{2}^{2})\) & \(\mathrm{IIB}\) & \(2 \times V(5,0)\) & \(V(5,0) \sqcup V(2,3)\) & \(\mathbb{Z}^{2}\) & \(\mathbb{Z}^{2}\) & \(\mathbb{C}^{2} \oplus \mathbb{C}^{0}\) & \(\mathbb{C}^{2} \oplus \mathbb{C}^{0}\) \\
\((1^{5}, \bar{1}^{4})\) & \(\mathrm{IIB}\) & \(3 \times V(5,0)\) & \(2 \times V(5,0) \sqcup V(2,3)\) & \(\mathbb{Z}^{3}\) & \(\mathbb{Z}^{3}\) & \(\mathbb{C}^{3} \oplus \mathbb{C}^{0}\) & \(\mathbb{C}^{3} \oplus \mathbb{C}^{0}\) \\
\((5, 1, \bar{2}, \bar{1})\) & \(\mathrm{IIB}\) & \(2 \times V(2,0)\) & \(2 \times V(2,0)\) & \(\mathbb{Z}^{2}\) & \(\mathbb{Z}^{2}\) & \(\mathbb{C}^{2} \oplus \mathbb{C}^{0}\) & \(\mathbb{C}^{2} \oplus \mathbb{C}^{0}\) \\
\bottomrule
\end{tabular}
\end{sidewaystable}

\begin{sidewaystable}[p]
\centering
\scriptsize
\caption{\tableheaderspinXVIII\@. Part 5 of 5.}
\label{tab:bipartitions-9-dn-5}
\begin{tabular}{C{2.65cm}C{0.85cm}C{3.0cm}C{3.0cm}C{2.25cm}C{2.25cm}C{1.55cm}C{1.55cm}}
\toprule
\makecell{\textbf{Bi-}\\\textbf{partition}} & \makecell{\textbf{Type}} & \makecell{\textbf{Spin}\\\textbf{sector}} & \makecell{\textbf{PSO}\\\textbf{sector}} & \makecell{\textbf{Spin}\\\textbf{cohom.}} & \makecell{\textbf{PSO}\\\textbf{cohom.}} & \makecell{\textbf{Spin}\\\textbf{K-theory}} & \makecell{\textbf{PSO}\\\textbf{K-theory}} \\
\midrule
\((4, 1^{2}, \bar{2}, \bar{1})\) & \(\mathrm{IIB}\) & \(2 \times V(3,0)\) & \(2 \times V(3,0)\) & \(\mathbb{Z}^{2}\) & \(\mathbb{Z}^{2}\) & \(\mathbb{C}^{2} \oplus \mathbb{C}^{0}\) & \(\mathbb{C}^{2} \oplus \mathbb{C}^{0}\) \\
\((3^{2}, \bar{2}, \bar{1})\) & \(\mathrm{IIB}\) & \(2 \times V(2,0)\) & \(2 \times V(2,0)\) & \(\mathbb{Z}^{2}\) & \(\mathbb{Z}^{2}\) & \(\mathbb{C}^{2} \oplus \mathbb{C}^{0}\) & \(\mathbb{C}^{2} \oplus \mathbb{C}^{0}\) \\
\((3, 2, 1, \bar{2}, \bar{1})\) & \(\mathrm{IIB}\) & \(2 \times V(3,0)\) & \(2 \times V(3,0)\) & \(\mathbb{Z}^{2}\) & \(\mathbb{Z}^{2}\) & \(\mathbb{C}^{2} \oplus \mathbb{C}^{0}\) & \(\mathbb{C}^{2} \oplus \mathbb{C}^{0}\) \\
\((3, 1^{3}, \bar{2}, \bar{1})\) & \(\mathrm{IIB}\) & \(2 \times V(4,0)\) & \(2 \times V(4,0)\) & \(\mathbb{Z}^{2}\) & \(\mathbb{Z}^{2}\) & \(\mathbb{C}^{2} \oplus \mathbb{C}^{0}\) & \(\mathbb{C}^{2} \oplus \mathbb{C}^{0}\) \\
\((2^{2}, 1^{2}, \bar{2}, \bar{1})\) & \(\mathrm{IIB}\) & \(2 \times V(4,0)\) & \(2 \times V(4,0)\) & \(\mathbb{Z}^{2}\) & \(\mathbb{Z}^{2}\) & \(\mathbb{C}^{2} \oplus \mathbb{C}^{0}\) & \(\mathbb{C}^{2} \oplus \mathbb{C}^{0}\) \\
\((2, 1^{4}, \bar{2}, \bar{1})\) & \(\mathrm{IIB}\) & \(2 \times V(5,0)\) & \(2 \times V(5,0)\) & \(\mathbb{Z}^{2}\) & \(\mathbb{Z}^{2}\) & \(\mathbb{C}^{2} \oplus \mathbb{C}^{0}\) & \(\mathbb{C}^{2} \oplus \mathbb{C}^{0}\) \\
\((1^{6}, \bar{2}, \bar{1})\) & \(\mathrm{IIB}\) & \(2 \times V(6,0)\) & \(2 \times V(6,0)\) & \(\mathbb{Z}^{2}\) & \(\mathbb{Z}^{2}\) & \(\mathbb{C}^{2} \oplus \mathbb{C}^{0}\) & \(\mathbb{C}^{2} \oplus \mathbb{C}^{0}\) \\
\((7, \bar{1}^{2})\) & \(\mathrm{IIB}\) & \(2 \times V(1,0)\) & \(V(1,0) \sqcup V(0,1)\) & \(\mathbb{Z}^{2}\) & \(\mathbb{Z}^{2}\) & \(\mathbb{C}^{2} \oplus \mathbb{C}^{0}\) & \(\mathbb{C}^{2} \oplus \mathbb{C}^{0}\) \\
\((6, 1, \bar{1}^{2})\) & \(\mathrm{IIB}\) & \(2 \times V(2,0)\) & \(V(2,0) \sqcup V(0,2)\) & \(\mathbb{Z}^{2}\) & \(\mathbb{Z}^{2}\) & \(\mathbb{C}^{2} \oplus \mathbb{C}^{0}\) & \(\mathbb{C}^{2} \oplus \mathbb{C}^{0}\) \\
\((5, 2, \bar{1}^{2})\) & \(\mathrm{IIB}\) & \(2 \times V(2,0)\) & \(V(2,0) \sqcup V(0,2)\) & \(\mathbb{Z}^{2}\) & \(\mathbb{Z}^{2}\) & \(\mathbb{C}^{2} \oplus \mathbb{C}^{0}\) & \(\mathbb{C}^{2} \oplus \mathbb{C}^{0}\) \\
\((5, 1^{2}, \bar{1}^{2})\) & \(\mathrm{IIB}\) & \(2 \times V(3,0)\) & \(V(3,0) \sqcup V(1,2)\) & \(\mathbb{Z}^{2}\) & \(\mathbb{Z}^{2}\) & \(\mathbb{C}^{2} \oplus \mathbb{C}^{0}\) & \(\mathbb{C}^{2} \oplus \mathbb{C}^{0}\) \\
\((4, 3, \bar{1}^{2})\) & \(\mathrm{IIB}\) & \(2 \times V(2,0)\) & \(V(2,0) \sqcup V(0,2)\) & \(\mathbb{Z}^{2}\) & \(\mathbb{Z}^{2}\) & \(\mathbb{C}^{2} \oplus \mathbb{C}^{0}\) & \(\mathbb{C}^{2} \oplus \mathbb{C}^{0}\) \\
\((4, 2, 1, \bar{1}^{2})\) & \(\mathrm{IIB}\) & \(2 \times V(3,0)\) & \(V(3,0) \sqcup V(0,3)\) & \(\mathbb{Z}^{2}\) & \(\mathbb{Z}^{2}\) & \(\mathbb{C}^{2} \oplus \mathbb{C}^{0}\) & \(\mathbb{C}^{2} \oplus \mathbb{C}^{0}\) \\
\((4, 1^{3}, \bar{1}^{2})\) & \(\mathrm{IIB}\) & \(2 \times V(4,0)\) & \(V(4,0) \sqcup V(1,3)\) & \(\mathbb{Z}^{2}\) & \(\mathbb{Z}^{2}\) & \(\mathbb{C}^{2} \oplus \mathbb{C}^{0}\) & \(\mathbb{C}^{2} \oplus \mathbb{C}^{0}\) \\
\((3^{2}, 1, \bar{1}^{2})\) & \(\mathrm{IIB}\) & \(2 \times V(3,0)\) & \(V(3,0) \sqcup V(1,2)\) & \(\mathbb{Z}^{2}\) & \(\mathbb{Z}^{2}\) & \(\mathbb{C}^{2} \oplus \mathbb{C}^{0}\) & \(\mathbb{C}^{2} \oplus \mathbb{C}^{0}\) \\
\((3, 2^{2}, \bar{1}^{2})\) & \(\mathrm{IIB}\) & \(2 \times V(3,0)\) & \(V(3,0) \sqcup V(1,2)\) & \(\mathbb{Z}^{2}\) & \(\mathbb{Z}^{2}\) & \(\mathbb{C}^{2} \oplus \mathbb{C}^{0}\) & \(\mathbb{C}^{2} \oplus \mathbb{C}^{0}\) \\
\((3, 2, 1^{2}, \bar{1}^{2})\) & \(\mathrm{IIB}\) & \(2 \times V(4,0)\) & \(V(4,0) \sqcup V(1,3)\) & \(\mathbb{Z}^{2}\) & \(\mathbb{Z}^{2}\) & \(\mathbb{C}^{2} \oplus \mathbb{C}^{0}\) & \(\mathbb{C}^{2} \oplus \mathbb{C}^{0}\) \\
\((3, 1^{4}, \bar{1}^{2})\) & \(\mathrm{IIB}\) & \(2 \times V(5,0)\) & \(V(5,0) \sqcup V(2,3)\) & \(\mathbb{Z}^{2}\) & \(\mathbb{Z}^{2}\) & \(\mathbb{C}^{2} \oplus \mathbb{C}^{0}\) & \(\mathbb{C}^{2} \oplus \mathbb{C}^{0}\) \\
\((2^{3}, 1, \bar{1}^{2})\) & \(\mathrm{IIB}\) & \(2 \times V(4,0)\) & \(V(4,0) \sqcup V(1,3)\) & \(\mathbb{Z}^{2}\) & \(\mathbb{Z}^{2}\) & \(\mathbb{C}^{2} \oplus \mathbb{C}^{0}\) & \(\mathbb{C}^{2} \oplus \mathbb{C}^{0}\) \\
\((2^{2}, 1^{3}, \bar{1}^{2})\) & \(\mathrm{IIB}\) & \(2 \times V(5,0)\) & \(V(5,0) \sqcup V(2,3)\) & \(\mathbb{Z}^{2}\) & \(\mathbb{Z}^{2}\) & \(\mathbb{C}^{2} \oplus \mathbb{C}^{0}\) & \(\mathbb{C}^{2} \oplus \mathbb{C}^{0}\) \\
\((2, 1^{5}, \bar{1}^{2})\) & \(\mathrm{IIB}\) & \(2 \times V(6,0)\) & \(V(6,0) \sqcup V(2,4)\) & \(\mathbb{Z}^{2}\) & \(\mathbb{Z}^{2}\) & \(\mathbb{C}^{2} \oplus \mathbb{C}^{0}\) & \(\mathbb{C}^{2} \oplus \mathbb{C}^{0}\) \\
\((1^{7}, \bar{1}^{2})\) & \(\mathrm{IIB}\) & \(2 \times V(7,0)\) & \(V(7,0) \sqcup V(3,4)\) & \(\mathbb{Z}^{2}\) & \(\mathbb{Z}^{2}\) & \(\mathbb{C}^{2} \oplus \mathbb{C}^{0}\) & \(\mathbb{C}^{2} \oplus \mathbb{C}^{0}\) \\
\midrule
\textrm{Total} &  &  &  &  &  & \(\mathbb{C}^{463} \oplus \mathbb{C}^{13}\) & \(\mathbb{C}^{463} \oplus \mathbb{C}^{13}\) \\
\bottomrule
\end{tabular}

\bigskip
\begin{minipage}{0.8\textwidth}
This table and its \LaTeX{} file were generated by Python code produced with assistance by ChatGPT and checked against values from independent Python code written by the authors.
\end{minipage}

\end{sidewaystable}